\documentclass[leqno,11pt]{article}
\usepackage{a4wide,amsmath,amssymb,amsfonts,mathrsfs,exscale,graphics,amsthm,colonequals,comment,ifthen,bbold,url}
\usepackage{hyperref,breakurl}
\usepackage[curve,matrix,arrow,cmtip]{xy}
\newdir^{ (}{{}*!/-3pt/\dir^{(}}   % Usage: \ar@{^{ (}->}[r]^{j}
\newdir_{ (}{{}*!/-3pt/\dir_{(}}   % Usage: \ar@{_{ (}->}[r]^{j}    
\NoComputerModernTips

\numberwithin{equation}{section}

\theoremstyle{plain}
\newtheorem{theorem}[equation]{Theorem}%[section]
\newtheorem{maintheorem}[equation]{Main Theorem}
\newtheorem{lemma}[equation]{Lemma}
\newtheorem{corollary}[equation]{Corollary}
\newtheorem{proposition}[equation]{Proposition}

\theoremstyle{definition}
\newtheorem{definition}[equation]{Definition}

\newtheorem{construction}[equation]{Construction}
\newtheorem{remark}[equation]{Remark}

\DeclareMathOperator{\fil}{fil}

\def\GradQCoh{\mathop{\tt GrQCoh}\nolimits}
\def\FilQCoh{\mathop{\tt FilQCoh}\nolimits}
\def\FilLF{\mathop{\tt FilLF}\nolimits}
\def\GradLF{\mathop{\tt GrLF}\nolimits}
\def\QCoh{\mathop{\tt QCoh}\nolimits}
\def\GFF{\mathop{\tt GFF}\nolimits}

\DeclareMathOperator{\forget}{forg}

\DeclareMathOperator{\End}{End}

\DeclareMathOperator{\Id}{Id}
\DeclareMathOperator{\GL}{GL}

\DeclareMathOperator{\gr}{gr}
\DeclareMathOperator{\res}{res}
\DeclareMathOperator{\ind}{ind}
\DeclareMathOperator{\F}{F}
\DeclareMathOperator{\G}{G}
\DeclareMathOperator{\Mor}{Mor}
\def\Sch{{\bf Sch}}
\DeclareMathOperator{\Spec}{Spec}
\DeclareMathOperator{\Cent}{Cent}
\DeclareMathOperator{\Hom}{Hom}

\DeclareMathOperator{\pr}{pr}
\DeclareMathOperator{\Tors}{Tors}

\DeclareMathOperator{\Aut}{Aut}

\DeclareMathOperator{\Desc}{Desc}
\DeclareMathOperator{\Norm}{Norm}
\DeclareMathOperator{\rk}{rk}
\DeclareMathOperator{\Lie}{Lie}
\DeclareMathOperator{\Gal}{Gal}

\let\into\hookrightarrow

\newcommand{\defeq}{\colonequals}
\newcommand{\rdefeq}{\equalscolon}
\newcommand{\Gm}[1][\empty]{
  \ifthenelse{\equal{#1}{\empty}}
    {\mathbb{G}_m}
    {\mathbb{G}_{m,#1}}}

 \newcommand{\Gred}[1][\empty]{
  \ifthenelse{\equal{#1}{\empty}}
    {G^{\text{red}}}
    {G^{\text{red},#1}}}

 \newcommand{\Rep}[1][\empty]{
  \ifthenelse{\equal{#1}{\empty}}
    {\mathop{\text{\tt Rep}}\nolimits}
    {\mathop{\text{$#1$-{\tt Rep}}}\nolimits}}

 \newcommand{\REP}[1][\empty]{
  \ifthenelse{\equal{#1}{\empty}}
    {\mathop{\text{\tt REP}}\nolimits}
    {\mathop{\text{$#1$-{\tt REP}}}\nolimits}}

\newcommand\lto{\longrightarrow}
\newcommand\ltoover[1]{\mathrel{\smash{\overset{#1}{\lto}}}}
\newcommand\varto[1]{\mathrel{\hbox to #1pt{\rightarrowfill}}}
\let\longto\longrightarrow
\let\onto\twoheadrightarrow
\def\isoto{\stackrel{\sim}{\longto}}
 
\newcommand{\BZ}{{\mathbb{Z}}}
\newcommand{\Bone}{{\mathbb{1}}}

\newcommand{\Fz}{{\mathfrak{z}}}

\newcommand{\CA}{{\mathcal A}}

\newcommand{\CC}{{\mathcal C}}

\newcommand{\CF}{{\mathcal F}}
\newcommand{\CG}{{\mathcal G}}

\newcommand{\CL}{{\mathcal L}}
\newcommand{\CM}{{\mathcal M}}
\newcommand{\CN}{{\mathcal N}}
\newcommand{\CO}{{\mathcal O}}

\newcommand{\CT}{{\mathcal T}}

\newcommand{\CX}{{\mathcal X}}

\def\UHom{\mathop{\underline{\rm Hom}}\nolimits}
\def\UIsom{\mathop{\underline{\rm Isom}}\nolimits}
\def\UAut{\mathop{\underline{\rm Aut}}\nolimits}

\def\UTransp{\mathop{\underline{\rm Transp}}\nolimits}
\def\USpl{\mathop{\underline{\rm Spl}}\nolimits}
\newcommand{\Hfl}{\operatorname{H}^1_\text{fl}}
\newcommand{\leftexp}[2]{{\vphantom{#2}}^{#1}{#2}}
\let\phi\varphi
\begin{document}
\title{Graded and Filtered Fiber Functors on Tannakian Categories}
\author{Paul Ziegler\footnote{Dept. of Mathematics,
 ETH Z\"urich
 CH-8092 Z\"urich,
 Switzerland,
 {\tt paul.ziegler@math.ethz.ch}
\newline
Supported by the Swiss National Science Foundation}
}
%\date{October 9, 2012}

\maketitle

\abstract{We study fiber functors on Tannakian categories which are equipped with a grading or a filtration. Our goal is to give a comprehensive set of foundational results about such functors. A main result is that each filtration on a fiber functor can be split by a grading fpqc-locally on the base scheme.

\section{Introduction}
\label{Introduction}
Let $k$ be a field and $\CT$ a Tannakian category over $k$, for example the category finite-dimensional representations of an affine group scheme $G$ over $k$, which we denote by $\Rep[G]$. In the theory of such categories, the notion of a fiber functor plays a central role. A fiber functor on $\CT$ over a scheme $S$ over $k$ is a $k$-linear exact tensor functor from $\CT$ to the category of locally free sheaves of finite rank on $S$. This article is concerned with fiber functors equipped with a grading or a filtration in the following sense: a \emph{graded (resp. filtered) fiber functor} on $\CT$ is a fiber functor $\phi$ on $\CT$ over some scheme $S$ over $k$ together with a grading (resp. descending filtration) of the locally free sheaf $\phi(X)$ for each object $X$ of $\CT$ such that these gradings (resp. filtrations) are functorial in $X$, compatible with tensor products and exact in a suitable sense. Our goal is to give a comprehensive set of foundational results about such functors. A number of results in this direction have already been obtained by Saavedra Rivano in \cite[IV]{Saavedra}.

%%%%%%%%%%%%%%%%%%%%%%%%%%%%%%%%%%%%%%%%%%
\bigskip
Now we give an overview of our results on graded fiber functors. Since a grading on a quasi-coherent sheaf $\CM$ over a scheme $S$ is the same as an action of $\Gm[S]$ on $\CM$, a graded fiber functor on $\CT$ over a scheme $S$ over $k$ can be described by giving the underlying fiber functor $\omega$ and a cocharacter of $\UAut^\otimes_S(\omega)$ (see Theorem \ref{GFFClassification}). Thus proving a statement about graded fiber functors usually boils down to proving the corresponding statement about cocharacters of a group scheme over $S$.

 Much of our work on graded fiber functors goes into defining and describing the notion of the type of a graded fiber functor. The type of a graded locally free sheaf of finite rank on a connected scheme $S$ over $k$ is the tuple consisting of the ranks of its graded pieces. Two such sheaves have the same type if and only if they are isomorphic Zariski-locally on $S$. This motivates the following definition of the type of a graded fiber functor:

Let $\CC_\CT\colon (\Sch/k)\to (\text{Sets})$ be the fpqc sheaf of isomorphism classes of graded fiber functors, i.e., the sheafification of the presheaf which sends a scheme $S$ over $k$ to the set of isomorphism classes of graded fiber functors over $S$ and which acts on morphisms by pullbacks. For a graded fiber functor $\gamma$ on $\CT$ over some scheme $S$ over $k$ we define the \emph{type of $\gamma$} to be the section $t(\gamma)\in \CC_\CT(S)$ given by the isomorphism class of $\gamma$. Thus two graded fiber functors on $\CT$ over a scheme $S$ over $k$ have the same type if and only they are isomorphic fpqc-locally on $S$. If $\CT=\Rep[G]$ for an affine group scheme $G$ over $k$, then $\CC_\CT$ is the sheaf of conjugacy classes of cocharacters of $G$ (see Lemma \ref{FomegaEquiv}).

In Theorem \ref{ClRepr} we prove:
\begin{theorem} 
  The functor $\CC_\CT$ is representable by a scheme which is \'etale over $k$.
\end{theorem}
Thus types of graded fiber functors are points of this scheme $\CC_\CT$, which we call the \emph{scheme of types} of $\CT$. Let $\bar k$ be an algebraic closure of $k$. In Subsection \ref{ExplicitST} we give a more explicit description of the scheme $\CC_\CT$ by describing the set $\CC_\CT(\bar k)$ together with its $\Aut(\bar k/k)$-action.

In Subsection \ref{GFFConsequences}, we use this notion of the type of a graded fiber functor to establish several results about graded fiber functors. In Theorem \ref{GFFIsom}, given two graded fiber functors $\gamma_i$ over schemes $S_i$ over $k$, we describe the subfunctor of $S_1\times_k S_2$ over which the pullbacks of the $S_i$ are isomorphic. In Theorem \ref{GFFClassification2}, we give a description of graded fiber functors of a given type in terms of certain torsors. In Theorem \ref{TypeRealizable}, we prove that under certain conditions each morphism $S\to \CC_\CT$ of schemes arises as the type of a graded fiber functor.

 In Subsection \ref{StGFF} we consider the moduli space of graded fiber functors and show that it forms an algebraic stack.

%%%%%%%%%%%%%%%%%%%%%%%%%%%%%%%%%
\bigskip
Now we come to filtered fiber functors. To every filtered locally free sheaf of finite rank $\CM$ one can associate the graded sheaf $\gr\CM$ whose graded pieces are the quotients of successive steps of the filtration on $\CM$. This allows us to associate to a filtered fiber functor $\phi$ on $\CT$ a graded fiber functor $\gr\circ\phi$. Inversely, every graded locally free sheaf of finite rank $\CM=\oplus_{i\in\BZ}\CM^i$ can be equipped with a descending filtration by taking the $i$-th step of the filtration to be $\oplus_{j\geq i}\CM^j$. We denote the resulting filtered locally free sheaf of finite rank by $\fil\CM$. This allows us to associate to every graded fiber functor $\gamma$ on $\CT$ the filtered fiber functor $\fil\circ\gamma$ on $\CT$. Filtered fiber functors of this form are called \emph{splittable}. 

Our main result on filtered fiber functors is the following (see Theorem \ref{MainTheorem}). Its proof is given in Section \ref{Proof}.
\begin{maintheorem}
  Every filtered fiber functor on $\CT$ is fpqc-locally splittable.
\end{maintheorem}
For any filtered fiber functor $\phi$, we denote the underlying fiber functor by $\forget\circ\phi$. Under additional assumptions, the main theorem can be strengthened (see Theorem \ref{FFFZLocSplittable}):
\begin{theorem}  \label{FFFZLocSplittableIntro}
 Let $\phi$ be a filtered fiber functor on $\CT$ over an affine scheme $S$. Assume that the group scheme $\UAut^\otimes_S(\forget\circ\phi)$ is pro-smooth over $S$. Then $\phi$ is splittable.
\end{theorem}
Here by a pro-smooth group scheme over $S$ we mean a group scheme which can be written as the inverse limit of a projective system of group schemes which are smooth over $S$.

The result of the main theorem was stated as an open problem by Saavedra Rivano in \cite[IV.2.2.1]{Saavedra}. In \cite[IV.2.4]{Saavedra} he gives a proof of this fact, due to Deligne, in case $\CT$ is neutral and $k$ has characteristic zero and in case $\CT=\Rep[G]$ for a reductive algebraic group $G$ over $k$. In \cite[IV.2.2.2]{Saavedra} and \cite[IV.2.2.5 3)]{Saavedra} Saavedra Rivano also gives proofs, again due to Deligne, of special cases of Theorem \ref{FFFZLocSplittable}. Our proof of the main theorem is based on a part of these arguments.

The main theorem allows us to deduce statements about filtered fiber functors from statements about graded fiber functors. For example, we define the type of a filtered fiber functor $\phi$ to be the type of the associated graded fiber functor $\gr\circ\phi$, and then the results on graded fiber functors described above also hold for filtered fiber functors, see Subsection \ref{FFFConsequences}.

In Subsection \ref{StFFF} we consider the moduli space of filtered fiber functors and show that it forms an algebraic stack.

\bigskip
This work was motivated by and needed for an ongoing collaboration with Richard Pink and Torsten Wedhorn which is concerned with (generalizations of) the Ekedahl-Oort stratifications of certain Shimura varieties in positive characteristic, see \cite{PWZ1}. The central notion there is that of an $F$-zip, which is a locally free sheaf of finite rank on a scheme of positive characteristic together with two filtrations which are tied together by Frobenius in a certain way. In the article \cite{PWZ2} Pink, Wedhorn and the author consider for certain algebraic groups $G$ functors $\Fz$ from $\Rep[G]$ to the category of $F$-zips. Each such $\Fz$ naturally gives rise to two filtered fiber functors on $\Rep[G]$ and in \cite{PWZ2} results from the present article are used to study $\Fz$. Since graded and filtered quasi-coherent sheaves appear in many other areas of mathematics, we hope that our results will find more applications elsewhere. 
\paragraph{Acknowledgement} I am deeply grateful to Richard Pink for his guidance. I thank Brian Conrad, Christophe Cornut and Pierre Deligne for pointing out mistakes in earlier versions of this article.
%%%%%%%%%%%%%%%%%%%%%%%%%%%%%%%%%%%%%%%%%%%%%

\section{Preliminaries}
\label{Preliminaries}

In the section we collect some terminology and facts regarding tensor categories and Tannakian categories. 

\subsection{Categories with a tensor product}

A \emph{symmetric monoidal category} is a category $\CC$ equipped with a functor $\otimes\colon \CC\times \CC\to \CC$ together with certain associativity and commutativity constraints as well as a identity object $\Bone$ (c.f. \cite[I]{Saavedra}, where such an object is called a ``$\otimes$-cat\'egorie ACU'').

Let $\CC$ and $\CC'$ be two symmetric monoidal categories. We call \emph{tensor functor} $F\colon \CC\to \CC'$ a functor $F\colon \CC\to \CC'$ together with certain functorial isomorphisms $F(X)\otimes F(Y)\cong F(X\otimes Y)$ for $X$ and $Y$ in $\CC$ as well as a suitable isomorphism $F(\Bone)\cong \Bone$ (c.f. \cite[I.4]{Saavedra}, where such functors are called ``$\otimes$-foncteur ACU''). 

Let $\CC$ and $\CC'$ be as above and let $F, F'\colon \CC\to\CC'$ be two tensor functors. We call \emph{tensor morphism} $F \to F'$ a natural transformation $F\to F'$ such that the induced morphism $F(\Bone)\to G(\Bone)$ is an isomorphism and such that for all $X,Y\in \CC$ the following square
\begin{equation*}
  \xymatrix{ F(X)\otimes F(Y) \ar[r] \ar[d] & F(X\otimes Y) \ar[d] \\
             G(X)\otimes G(Y) \ar[r]  & G(X\otimes Y)  \\
},
\end{equation*}
in which the vertical maps are those given by the natural transformation and the horizontal maps those part of the tensor functors is commutative (c.f. \cite[I.4]{Saavedra}, where such natural transformations are called ``$\otimes$-morphisme unif\`ere''). We call \emph{tensor isomorphism} $F\to F'$ a tensor morphism $F\to F'$ which is an isomorphism of functors.

 A \emph{rigid object} in a symmetric monoidal category is what is called an object admitting a dual in \cite[2.2]{DeligneGroth}. Any tensor functor maps rigid objects to rigid objects. A symmetric monoidal category is rigid if all its objects are rigid. If $\CC$ is a rigid symmetric monoidal category and $F,F'$ are two tensor functors from $\CC$ to a symmetric monoidal category $\CC'$, any tensor morphism $F\to F'$ is a tensor isomorphism (c.f. \cite[I.5.2.3]{Saavedra}).

\subsection{Tannakian Categories}

For a scheme $S$, we denote by $\QCoh(S)$ the category of quasi-coherent sheaves on $S$. 

Let $k$ be a field. A \emph{Tannakian category} over $k$ is an essentially small symmetric monoidal category $\CT$ which is abelian, $k$-linear and rigid, for which the natural morphism $k\to \End(\Bone)$ induced by the $k$-linear structure of $\CT$ is an isomorphism and for which there exists a non-empty scheme $S$ over $k$ and an exact $k$-linear tensor functor $\omega$ from $\CT$ to $\QCoh(S)$.  C.f. \cite[2.8]{DeligneGroth}
\footnote{The condition that $\CT$ is essentially small does not appear in \cite{DeligneGroth}. However, without this condition the author does not understand the argument there, since in \cite[4.7]{DeligneGroth} a certain construction is made which appears to work only for an essentially small category, which is later applied to Tannakian categories. At any rate, once one knows the result \cite[1.12]{DeligneGroth} that any Tannakian category is the category of representations of a gerbe, it follows from this result that $\CT$ is essentially small. Alternatively the issue could also be avoided by working within a suitable universe, as in \cite{Saavedra}.}
, but note this is not the same definition as in \cite[III.3.2.1]{Saavedra}, since there the condition $k\isoto\End(\Bone)$ is missing. Exact $k$-linear tensor functors $\omega\colon \CT\to \QCoh(S)$ are called \emph{fiber functors} on $\CT$ over $S$. 

Let $\CT$ be a Tannakian category over $k$ and let $S$ be a scheme over $k$. For a fiber functor $\omega$ on $\CT$ over $S$ and a scheme $S'$ over $S$, we denote by $\omega|_{S'}$ the fiber functor over $S'$ obtained by composing $\omega$ with the natural base change functor $\QCoh(S)\to \QCoh(S')$.

For fiber functors $\omega_1$ and $\omega_2$ on $\CT$ over $S$, we denote by $\UIsom^\otimes_S(\omega_1,\omega_2)$ the functor on the category of schemes over $S$ which associates to any scheme $T$ over $S$ the set of tensor isomorphisms $\omega_1|_T\to \omega_2|_T$ and which sends morphisms of schemes over $S$ to pullbacks. In case $\omega_1=\omega_2$, we denote this functor by $\UAut^\otimes_S(\omega_1)$. Composition of tensor morphisms makes $\UAut^\otimes_S(\omega_1)$ into a group-valued functor.

\begin{theorem}[{Deligne, see \cite[1.11-1.13]{DeligneGroth}}]\label{FFLocalIso}
  The functor $\UIsom^\otimes_S(\omega_1,\omega_2)$ is representable by a scheme which is affine and faithfully flat over $S$.
\end{theorem}

The Tannakian category $\CT$ is called \emph{neutral} if there exists a fiber functor $\omega$ on $\CT$ over $k$. By Theorem \ref{FFLocalIso}, associated to such a fiber functor one has the affine group scheme $G\defeq\UAut^\otimes_k(\omega)$ over $k$. For any $X\in \CT$, the group scheme $G$ acts naturally on the vector space $\omega(X)$ and in this way $\omega$ induces a tensor functor $\CT\to \Rep[G]$. The structure of a neutral Tannakian category is described by the following result:
\begin{theorem}[{Deligne-Milne, \cite[Theorem II.2.11]{LNM900}}]
  Let $k$ be a field and $\CT$ a neutral Tannakian category over $k$. The functor $\CT\to\Rep[G]$ induced by a fiber functor $\omega$ on $\CT$ over $k$ is an equivalence of categories.
\end{theorem}

To extend this result to all Tannakian categories, one needs the notion of a gerbe:
\subsection{Gerbes} \label{GerbesSection}

Recall that a stack $\CG$ over a scheme $S$ is a fibered category over $S$ which satisfies effective descent for both objects and morphisms. A gerbe $\CG$ over a scheme $S$ is a stack for the fpqc topology fibered in groupoids over $(\Sch/S)$ satisfying the following two requirements (c.f. \cite[Section III.2.1]{GiraudCNA}):
\begin{itemize}
\item[(i)] There exists a fpqc covering $S'$ of $S$ for which $\CG(S')$ is non-empty.
\item[(ii)] Any two objects of $\CG$ are fpqc-locally isomorphic.
\end{itemize}

The prototypical example of a gerbe is the stack $\Tors_G$ of left $G$-torsors for an fpqc sheaf of groups $G$ over $S$. It is a gerbe over $S$ since $\Tors_G(S)$ contains the trivial torsor and any $G$-torsor is fpqc-locally isomorphic to the trivial torsor. Any gerbe is fpqc-locally isomorphic to such a gerbe because of the following fact:
\begin{theorem}[{\cite[Th\'eor\`eme III.2.5.1]{GiraudCNA}}] \label{GerbeClassification}
  Let $\CG$ be a gerbe over $S$ and $x\in \CG(S')$ for some $S$-scheme $S'$. Then there is an equivalence of gerbes $\CG_{S'}\to \Tors_{\UAut_{S'}(x)}$ which for every $S'$-scheme $S''$ sends objects $y\in \CG(S'')$ to $\UIsom_{S''}(y,x)$ and morphisms $y\to y'$ in $G(S'')$ to the induced morphism $\UIsom_{S''}(y,x)\to \UIsom_{S''}(y',x)$.
\end{theorem}

In particular every gerbe $\CG$ is locally completely determined by the automorphism group of any object. Given two objects $x,x'\in \CG(S')$ for some fpqc cover $S'$ of $S$, there exists, locally on $S'$, an isomorphism $x\cong x'$. Such an isomorphism induces an isomorphism $\UAut(x)\cong \UAut(x')$. If one chooses two different isomorphisms $x\cong x'$ the resulting isomorphisms $\UAut(x)\cong \UAut(x')$ differ by an inner isomorphism. This data of fpqc-locally defined isomorphisms $\UAut(x)\cong\UAut(x')$ which are well-defined up to inner isomorphism can be encapsulated in the so-called $\emph{band}$ of $\CG$, c.f. \cite[Section IV.2.2]{GiraudCNA}. The role that the band of $\CG$ plays for $\CG$ is analogous to the role that a group scheme $G$ plays for $\Tors_G$. We will however not use the notion of a band, we only use it to motivate the following language:

\begin{definition}\label{BandPropertyDef}
  Let $\CG$ be a gerbe over $S$. Let $P$ be a property of a group scheme $G$ over a base scheme $S'$ which is local for the fpqc topology on $S'$. Then we say that the band of $\CG$ has property $P$ over $S$ if and only if $\UAut_{S'}(x)$ has property $P$ over $S'$ for any (equivalently every) object $x\in \CG(S')$ over some fpqc cover $S'$ of $S$.
\end{definition}

Let as above be $k$ a field and $\CT$ a Tannakian category over $k$. For any scheme $S$ over $k$, let $\UHom^\otimes(\CT,\QCoh)(S)$ be the category whose objects are fiber functors on $\CT$ over $S$ and whose morphisms are isomorphisms of tensor functors. With the pullback functors defined above these categories form a category fibered in groupoids over $(\Sch/S)$ which we denote $\UHom^\otimes(\CT,\QCoh)$. By \cite[III.3.2.1.2]{Saavedra} this is a stack and Theorem \ref{FFLocalIso} implies that it is in fact a gerbe over $S$. The affineness statement of Theorem \ref{FFLocalIso} implies that the band of this gerbe is affine over $k$.

Tannakian categories over $k$ form a 2-category, with 1-morphisms given by exact $k$-linear tensor functors and 2-morphisms by tensor isomorphisms. Also, gerbes form a 2-category with 1-morphisms given by morphisms of stacks and 2-morphisms given by equivalences between 1-morphisms. The above construction which associates to a Tannakian category its gerbe of fiber functors can naturally be extended to a 2-functor between these 2-categories. Then Tannaka duality can be formulated as follows:
\begin{theorem}[{Deligne, \cite[1.12]{DeligneGroth}}]
  The above 2-functor which associates to a Tannakian category its gerbe of fiber functors is a 2-equivalence between the 2-category of Tannakian categories over $k$ and the 2-category of gerbes with affine band.
\end{theorem}
 There is also a natural construction which associates to a gerbe its category of representations which gives an inverse to the above 2-functor (c.f. \cite[Section 3]{DeligneGroth}).

Thus a Tannakian category is completely determined by the gerbe $\UHom^\otimes(\CT,\QCoh)$. In case $\CT=\Rep[G]$, the canonical fiber functor of $\Rep[T]$ identifies this gerbe with $\Tors_G$ as in Theorem \ref{GerbeClassification}. Thus for a property $P$ as in Definition \ref{BandPropertyDef}, the statement ``the band of $\UHom^\otimes(\CT,\QCoh)$ has property $P$ over $k$'' about a Tannakian category $\CT$ is the correct generalization of the statement ``$G$ has property $P$ over $k$'' about the category $\Rep[G]$. A statement of this type appears for example in Theorem \ref{FFFZLocSplittableIntro}.

\subsection{Tannakian categories possessing a tensor generator}
\label{TensorGenSection}
Let $\CT$ be a Tannakian category over a field $k$. An object $X$ of $\CT$ is a \emph{tensor generator} of $\CT$ if each object of $\CT$ is a subquotient of a direct sum of objects of the form $X^{\otimes n}$ for $n\geq 0$ (c.f. \cite[II.4.3.1]{Saavedra}). By \cite[II.4.3.2]{Saavedra} a Tannakian category $\CT$ has a tensor generator if and only if the band of the associated gerbe is of finite type over $k$. 

\begin{definition}
  Let $I_\CT$ be the class of strictly full Tannakian subcategories of $\CT$ possessing a tensor generator. This is a set because $\CT$ is essentially small. We endow $I_\CT$ with the partial order induced by inclusion.
\end{definition}
 Then for any fiber functor $\phi$ on $\CT$ over some scheme $S$ over $k$ the group scheme $\UAut^\otimes_S(\phi)$ becomes naturally the inverse limit of the group schemes $\{\UAut^\otimes_S(\phi|_{\CT'})\mid \CT'\in I_\CT\}$ which are of finite type over $S$. It also follows from \cite[III.3.3.3]{Saavedra} that the homomorphisms $\UAut^\otimes_S(\phi)\to \UAut^\otimes_S(\phi|_{\CT'})$ and $\UAut^\otimes_S(\phi|_{\CT'})\to \UAut^\otimes_S(\phi|_{\CT''})$ for $\CT''\subset \CT'$ in $I_\CT$ are faithfully flat. 

  For us a smooth morphism of schemes need not be of finite type, but only locally of finite type as in \cite[Tag 01V5]{StacksProject}. By a pro-smooth group scheme over $S$ we mean a group scheme which is the inverse limit of group schemes which are smooth over $S$. From the above it follows that $\UAut^\otimes_S(\phi)$ is pro-smooth over $S$ if and only if for any full Tannakian subcategory $\CT'$ of $\CT$ possessing a tensor generator the group scheme $\UAut^\otimes_S(\phi_{\CT'})$ is smooth over $S$. 

In the remainder of this subsection, we prove some results which will be needed in the proof of Theorem \ref{UTorsorTrivial} below.
  \begin{lemma} \label{BlaLemma}
    If there exists a fiber functor $\omega$ on $\CT$ over a scheme $S$ over $k$ such that $\UAut^\otimes_S(\omega)$ is smooth over $S$ then $\CT$ has a tensor generator.
  \end{lemma}
  \begin{proof}
    Since $\UAut^\otimes_S(\omega)$ is smooth over $S$ it is locally of finite type over $S$. By Theorem \ref{FFLocalIso} this scheme is affine over $S$. Together these two facts imply that it is of finite type over $S$ which implies the claim.
  \end{proof}

  \begin{definition}
    A \emph{lattice} is a partially ordered set $I$ such that for any two elements $i,j$ of $I$ there exists a least upper bound $i \vee j$ und a greatest lower bound $i\wedge j$ of $i$ and $j$. We will later need:
  \end{definition}

\begin{lemma} \label{ITLattice}
  The partially ordered set $I_\CT$ is a lattice.
\end{lemma}
\begin{proof}
  Let $i,j\in I_\CT$. We need to show that there exist elements $i\vee j$ and $i\wedge j$ in $I_\CT$. 

 For $i\vee j$ we take the strictly full Tannakian subcategory of $\CT$ generated by all elements of $i$ and $j$. If $X_i,X_j$ are tensor generators of $i,j$, then $X_i\oplus X_j$ is a tensor generator of $i\vee j$, so that $i\vee j\in I_\CT$. It follows directly from the construction that $i\vee j$ is a least upper bound for $i$ and $j$ in $I_\CT$.

 Let $i\wedge j\defeq i\cap j$. This is a strictly full Tannakian subcategory of $\CT$. To see that it has a tensor generator, pick a fiber functor $\omega\colon \CT\to\QCoh(k')$ over some overfield $k'$ of $k$. The fact that $i$ has a tensor generator implies that $\UAut^\otimes_{k'}(\omega|_i)$ is of finite type over $k'$. Hence $\UAut^\otimes_{k'}(\omega|_{i\wedge j})$, being a quotient of $\UAut^\otimes_{k'}(\omega|_i)$, is also of finite type over $k'$. This proves that $i\wedge j$ has a tensor generator. Thus $i\wedge j\in I_\CT$ is a greatest lower bound for $i$ and $j$.
\end{proof}

\begin{lemma} \label{FFCartesian}
  Let $\omega\colon \CT\to \QCoh(S)$ be a fiber functor. For all $i,j \in I_\CT$, the square
  \begin{equation*}
    \xymatrix{
      \UAut^\otimes_S(\omega|_i) \ar[d] & \UAut^\otimes_S(\omega|_{i\vee j}) \ar[l]\ar[d] \\
      \UAut^\otimes_S(\omega|_{i\wedge j}) & \UAut^\otimes_S(\omega|_j) \ar[l] \\
      }
  \end{equation*}
  whose arrows are the morphisms induced by restriction is Cartesian.
\end{lemma}
\begin{proof}[Proof of Lemma \ref{FFCartesian}]
  It suffices to prove the claim fpqc-locally on $S$. After replacing $\CT$ by $i\vee j$, we may assume that $\CT$ has a tensor generator. Then there exists a finite field extension $k'$ of $k$ and a fiber functor $\omega'\colon \CT\to \QCoh(k')$. Since the fiber functors $\omega_{S_{k'}}$ and $\omega'_{S_{k'}}$ are isomorphic fpqc-locally on $S_{k'}$ by Theorem \ref{FFLocalIso}, it suffices to prove the claim for $\omega'$ instead of $\omega$. Thus we have reduced to the case that $S=\Spec(k')$ for a finite field extension $k'$ of $k$.

Construction \ref{TannakaBaseChange} implies $(i\wedge j)_{k'}=i_{k'}\wedge j_{k'}$ and \cite[III.3.2.4.1]{Saavedra} implies $(i\vee j)_{k'}=i_{k'}\vee j_{k'}$. Thus after replacing $i,j, i\vee j$ and $i\wedge j$ by their base change to $k'$, which does not change the square in question, we may assume that $k=k'$.

Set $G\defeq \UAut^\otimes_k(\omega)$. For $h\in I_\CT$ let $G_h\defeq \UAut^\otimes_k(\omega|_h)$ and $K_h$ the kernel of the homomorphism $G\to G_h$ induced by restriction. Since giving a strictly full Tannakian subcategory of $\Rep[G]$ is the same as giving the corresponding quotient of $G$, the assignment $h\mapsto K_h$ gives a order-reversing bijection between $I_\CT$ and the lattice of normal subgroup schemes of $G$. This implies $K_{i\vee j}=K_i\cap K_j$ and $K_{i\wedge j}=K_i K_j$. Thus we have reduced to the following claim, which follows from a direct verification.
\end{proof}
\begin{lemma}
  Let $G$ be an affine group scheme of finite type over $k$. Let $K,K'$ be two normal subgroup schemes of $G$. Then the square
  \begin{equation*}
    \xymatrix{
      G/K \ar[d] & G/K\cap K' \ar[d]\ar[l] \\
      G/ KK' & G/ K' \ar[l]\\
    }
  \end{equation*}
  is Cartesian.
\end{lemma}

\subsection{Base change of a Tannakian category}

We recall the notion of the base change of a Tannakian category $\CT$ over to a finite field extension:
\begin{construction} \label{TannakaBaseChange}
Let $\CT$ be a Tannakian category over $k$ and let $k'$ be a finite field extension of $k$. The base change of $\CT$ to $k'$ is the category $\CT_{k'}$ defined as follows: Its objects are pairs $(V,\alpha)$ where $V\in\CT$ and $\alpha\colon k'\to \End(V)$ is a homomorphism of $k$-algebras. A morphism $(V,\alpha)\to(V',\alpha')$ in $\CT_{k'}$ is a morphism $f\colon V\to V'$ such that $f\circ \alpha(\lambda)=\alpha'(\lambda)\circ f$ for all $\lambda\in k'$. The tensor product of two elements $(V,\alpha),(V',\alpha')\in \CT_{k'}$ is defined to be $(V\otimes_{k'} V',\tilde \alpha)$, where $V\otimes_{k'}V$ is the largest quotient of $V\otimes V'$ on which $\alpha(\lambda)\otimes \Id_{V'}$ and $\Id_V\otimes\alpha'(\lambda)$ agree for all $\lambda\in k$ and where $\tilde \alpha$ is induced by $\alpha\otimes \Id_{V'}$ or, equivalently, by $\Id_V\otimes\alpha'$. This defines a functor $\otimes\colon \CT_{k'}\times\CT_{k'}\to \CT_{k'}$. As is shown in \cite{LNM900} between Remark II.3.10 and Proposition II.3.11, this makes $\CT_{k'}$ into a Tannakian category over $k'$.

 There is a natural tensor morphism $\CT\to\CT_{k'}$ which sends $V\in \CT$ to $(V\otimes k',\alpha_V)$ where $\alpha_V$ sends $\lambda\in k'$ to the endomorphism $\Id_V\otimes \lambda$ of $V\otimes k'$. By [loc. cit.] any fiber functor $\omega\colon \CT\to \QCoh(k')$ extends to a fiber functor $\omega'\colon \CT_{k'}\to \QCoh(k')$ such that the functor $\CT\to \CT_{k'}$ induces an isomorphism $\UAut^\otimes_{k'}(\omega')\to \UAut^\otimes_{k'}(\omega)$. In particular $\CT_{k'}$ is neutral if $\CT$ possesses a fiber functor over $k'$.
\end{construction}

\subsection{An exactness criterion for inverse limits}
In this subsection we prove Theorem \ref{InvLimitProjModules} below, which will be needed in the proof of Theorem \ref{UTorsorTrivial}.
\begin{construction}
  Let $(G_i)_{i\in I}$ be an inverse system of groups indexed by a partially ordered set $I$ with transition maps $\pi_j^i\colon G_i\to G_j$ for all $j\leq i$ in $I$.  Let $Z$ be the set of tuples $(g_j^i)_{j\leq i\in I}$ such that $g_j^i\in G_j$ for all $j\leq i$ and $g_k^i=\pi_k^j(g_j^i)g_k^j$ for all $k\leq j\leq i$ in $I$. Let $B\defeq \prod_{i\in I}G_i$. This group acts on $Z$ from the left by $(g_i)_{i\in I}\cdot (g_j^i)_{j\leq i\in I}\defeq (\pi_j^i(g_i) g_j^i g_j^{-1} )_{i\leq j\in I}$. Let $\varprojlim_{i\in I}^1G_i\defeq B\backslash Z$. This is a set which is pointed by the class of $(1_{G_j})_{j\leq i\in I}\in Z$. It is functorial in $(G_i)_{i\in I}$ in a natural way.
\end{construction}
In case the $G_i$ are abelian, the set $\varprojlim_{i\in I}^1G_i$ inherits a group structure. In this case it follows from \cite[Th\'eor\`eme 4.1]{JensenInverseLimits} that $\varprojlim_{i\in I}^1$ is the first right derived functor of the inverse limit functor.
\begin{definition}
 Let $(G_i)_{i\in I}$ be an inverse system of groups indexed by a partially ordered set $I$. A (left) \emph{$(G_i)_{i\in I}$-torsor} is an inverse system of sets $(X_i)_{i\in I}$ such that each $X_i$ is a left $G_i$-torsor and for all $j\leq i$ the transition morphism $X_i\to X_j$ is equivariant with respect to the transition morphism $G_i\to G_j$.

A \emph{morphism} between $(G_i)_{i\in I}$-torsors $(X_i)_{i\in I}$ and $(X_i')_{i\in I}$ is a morphism $(X_i)_{i\in I}\to (X_i')_{i\in I}$ of inverse systems of sets such that each morphism $X_i\to X_i'$ is $G_i$-equivariant.

We call $(G_i)_{i\in I}$ considered as a torsor under itself by left multiplication the \emph{trivial $(G_i)_{i\in I}$-torsor}. A $(G_i)_{i\in I}$-torsor $(X_i)_{i\in I}$ is isomorphic to the trivial torsor if and only if $\varprojlim_{i\in I}X_i\not=\emptyset$.
\end{definition}
\begin{lemma} \label{Lim1Torsor}
  Let $(G_i)_{i\in I}$ be an inverse system of groups indexed by a partially ordered set $I$. There is a natural bijection between the set of isomorphism classes of $(G_i)_{i\in I}$-torsors and $\varprojlim^1_{i\in I}G_i$ under which the distinguished point of $\varprojlim^1_{i\in I}G_i$ corresponds to the class of the trivial torsor $(G_i)_{i\in I}$. 
\end{lemma}
\begin{proof}
  Let $(X_i)_{i\in I}$ be a $(G_i)_{i\in I}$-torsor. For each $i\in I$ we pick an element $x_i\in X_i$. Then each transition map $f_j^i\colon X_i\to X_j$ for $j\leq i$ in $I$ is uniquely determined by the unique element $g_j^i\in G_j$ such that $f_j^i(x_i)=g_j^i\cdot x_j$. The fact that $(X_i)_{i\in I}$ is an inverse system implies $(g_j^i)_{j\leq i\in I}\in Z$. A different choice of $x_i$ would be of the form $(g_j\cdot x_i)_{i\in I}$ for certain $g_i\in G_i$. The element of $Z$ associated to such a choice of $x_i$ would be $( \pi_j^i(g_i) g_j^ig^{-1}_j)_{j\leq i\in I}$. Thus this construction associates to $(X_i)_{i\in I}$ a well-defined element of $\varprojlim_{i\in I}^1G_i$. By a direct verification this element only depends on the isomorphism class of $(X_i)_{i \in I}$. Thus we have a map as claimed and it follows from a direct verification that it is a bijection. The last claim follows directly from the construction.
\end{proof}
\begin{lemma}\label{Lim1ExactSeq}
  Let $0\to (G_i')_{i\in I}\to (G_i)_{i\in I}\to (G_i'')_{i\in I}\to 0$ be an exact sequence of inverse systems of groups indexed by a partially ordered set $I$. Then there is a natural sequence of pointed sets
  \begin{equation*}
    \xymatrix@C-9pt{ 0 \ar[r] & \varprojlim_{i\in I}G'_i \ar[r] & \varprojlim_{i\in I}G_i\ar[r] & \varprojlim_{i\in I}G''_i \ar[r] &\varprojlim_{i\in I}^1G'_i\ar[r] & \varprojlim_{i\in I}^1G_i\ar[r] & \varprojlim_{i\in I}^1G_i''
}
  \end{equation*}
which is exact in the following sense: The image of each map is the fiber over the distinguished point of the following map, and in addition there is a natural left action of $\varprojlim_{i\in I}G''_i$ on $\varprojlim_{i\in I}^1G'_i$ the orbits of which are the fibers of the map $\varprojlim_{i\in I}^1G'_i\to \varprojlim_{i\in I}^1G_i$. 
\end{lemma}
\begin{proof}
The action of $\varprojlim_{i\in I}G''_i$ on $\varprojlim_{i\in I}^1G'_i$ is defined as follows: Let $g''=(g''_i)_{i\in I}\in \varprojlim_{i\in I}G''_i$ and let $(\tilde{ g}^i_j)_{j\leq i\in I}$ with $\tilde{g}^i_j\in G'_j$ be a representative of an element $\tilde{g}\in \varprojlim_{i\in I}^1G'_i$. Pick elements $g_i$ in $G_i$ which lift the elements $g''_i\in \G''_i$. Then $g''$ sends $\tilde g$ to the class of $(\pi^i_j(g_i)\tilde {g}^i_j g_j^{-1})_{j\leq i\in I}$. It follows by a direct verification that this is a well-defined action. The map $\varprojlim_{i\in I}^1G'_i\to \varprojlim_{i\in I}^1G_i$ is the map which sends $g\in \varprojlim_{i\in I}^1G'_i$ to the image of the distinguished point of $\varprojlim_{i\in I}^1G_i$. The remaining maps are the natural functoriality maps. That the sequence is exact follows by direct verification. 
%  The map $\varprojlim_{i\in I}G''_i\to \varprojlim_{i\in I}^1G_i$ is defined as follows: Let $(g_i)_{i\in I}\in \varprojlim_{i\in I}G''_i$. For each $i\in I$ let $X_i$ be the preimage of $g_i$ in $G_i$. For varying $i$, the $X_i$ form a $(G'_i)_{i\in I}$-torsor in a natural way, and the element of $\varprojlim_{i\in I}^1G'_i$ corresponding to this torsor under the bijection of Lemma \ref{Lim1Torsor} is defined to be the image of $(g_i)_{i\in I}$. 
\end{proof}

\begin{theorem} \label{InvLimitProjModules}
  Let $R$ be a ring and $(M_i)_{i\in I}$ an inverse system of $R$-modules indexed by a partially ordered set $I$ such that:
  \begin{enumerate}
  \item [(i)] The set $I$ is a lattice.
  \item [(ii)] Each $M_i$ is a finitely generated projective $R$-module.
  \item [(iii)] For all $j\leq i$ in $I$, the transition morphism $\mu_j^i\colon M_i\to M_j$ is surjective.
  \item [(iv)] For all $i,j$ in $I$, the following square is Cartesian:
    \begin{equation*}
      \xymatrix{
        M_i \ar[d] & M_{i\vee j} \ar[l] \ar[d] \\
        M_{i\wedge j} & M_j \ar[l]
        }
    \end{equation*}
  \end{enumerate}
Then $\varprojlim_{i\in I}^1M_i=0$. 
\end{theorem}
To prove this, we will use the following criterion:
\begin{theorem}[{\cite[Th\'eor\`eme III.7.4.1]{BourbakiTDE}}] \label{InvLimitCrit}
  Let $I$ be filtered partially ordered set and $(X_i)_{i\in I}$ an inverse system of non-empty sets with transition morphisms $f_j^i$ for $i\leq j$ in $I$. Assume that for each $i\in I$ there is a family $S_i$ of subsets of $X_i$ such that:
  \begin{enumerate}
  \item [(i)] Each $S_i$ is closed under arbitrary intersections.
  \item [(ii)] For each $F\subset S_i$ such that each finite intersection of elements of $F$ is not empty, the intersection of all elements of $F$ is not empty.
  \item [(iii)] For all $j\leq i$ in $I$ and $x\in X_j$, the set $(f_j^i)^{-1}(x)$ is in $S_i$.
  \item [(iv)] For all $j\leq i$ in $I$ and $Y\in S_i$, the set $f_j^i(Y)$ is in $S_j$.
  \end{enumerate}
Then the set $\varprojlim_{i\in I}X_i$ is not empty.
\end{theorem}
\begin{proof}[Proof of Theorem \ref{InvLimitProjModules}]
By Lemma \ref{Lim1Torsor} it is enough to show that for any $(M_i)_{i\in I}$-torsor $(X_i)_{i\in I}$, the set $\varprojlim_{i\in I}X_i$ is not empty.
  \begin{lemma} \label{SomeLemma2}
    For all $j,j'\leq i$ in $I$:
    \begin{enumerate}
    \item [(i)] $\ker(\mu_j^i)\cap \ker(\mu_{j'}^i)=\ker{\mu_{j\vee j'}^i}$.
    \item [(ii)] $\mu_{j'}^i(\ker(\mu_j^i))=\ker(\mu_{j\wedge j'}^{j'})$.
    \end{enumerate}
  \end{lemma}
  \begin{proof}
    (i) This follows by a direct verification using the fact that the square
    \begin{equation*}
      \xymatrix{
        M_j \ar[d] & M_{j\vee j'} \ar[l] \ar[d] \\
        M_{j \wedge j'} & M_{j'} \ar[l]
        }
    \end{equation*}
    is Cartesian. 

    (ii) Using the surjectivity of $\mu_{j\vee j'}^i$ we can reduce to the case $i=j\vee j'$. Then again the claim follows by a direct verification from the fact that the above square is Cartesian.
  \end{proof}
For $i\in I$ let $S_i\defeq \{x+\ker(\mu_j^i)\mid x\in X_i, j\leq i\}\cup \{\emptyset, X_i\}$. The assumption that the $\mu_j^i$ are surjective implies that their kernels are direct summands of $M_i$. Thus these kernels are finitely generated projective $R$-modules and it makes sense to speak of their rank as a locally constant function $\Spec(R)\to \BZ_{\geq 0}$. The set of all such functions is partially ordered by the relation $\leq $ for which $f\leq g$ if and only if $f(x)\leq g(x)$ for all $x\in \Spec(R)$. The fact that $\Spec(R)$ has only finitely many connected components implies that each family of such functions contains a minimal element.
\begin{lemma}\label{SomeLemma1}
  If $\ker(\mu_j^i)\subset \ker(\mu_{j'}^i)$ and $\rk(\ker(\mu_j^i))=\rk(\ker(\mu_{j'}^i))$ for certain $j,j'\leq i$ in $I$, then $\ker(\mu_j^i)= \ker(\mu_{j'}^i)$.
\end{lemma}
\begin{proof}
  Since $\ker(\mu_j^i)$ and $\ker(\mu_{j'}^i)$ are direct summands of $M_i$, the module $\ker(\mu_j^i)$ is also a direct summand of $\ker(\mu_{j'}^i)$. This together with the equality of the ranks implies that the two modules are equal.
\end{proof}

We need to verify that these $S_i$ satisfy the conditions (i) to (iv) of Theorem \ref{InvLimitCrit}.

(i) Consider a family $F\subset S_i$. If $\cap_{X \in F}X=\emptyset$, then $\cap_{X\in F}X\in S_i$. Else we can write $F=\{x+\ker(\mu_{j_\alpha}^i)\}$ for a certain family of $j_\alpha\leq i$ and an $x\in \cap_{X\in F}X$. Then Lemma \ref{SomeLemma2} (i) implies that each finite intersection of elements of $F$ is of the form $x+\ker(\mu_j^i)$ for some $j\leq i$. Now pick such a finite intersection with $\rk(\ker(\mu_j^i))$ minimal among all such finite intersections. Then Lemma \ref{SomeLemma1} implies that $\cap_{X\in F}X=x+\ker(\mu_j^i)\in S_i$.

(ii) Let $F\subset S_i$ such that each finite intersection of elements of $F$ is non-empty. Then as in (i) we can write each such finite intersection as $x+\ker(\mu_j^i)$ for certain $x\in X_i$ and $j\leq i$. Pick such a finite intersection with $\rk(\ker(\mu_j^i))$ minimal. Then Lemma \ref{SomeLemma1} implies that this finite intersection is equal to $\cap_{X\in F}X$. Thus $\cap_{X\in F}X$ is not empty.

(iii) This follows directly from the definition of the $S_i$.

(iv) This follows directly from Lemma \ref{SomeLemma2} (ii).
\end{proof}

\section{Graded Fiber Functors} \label{GrFF}
Throughout, we denote by $k$ a field and by $S,S_1,S_2$ arbitrary non-empty schemes over $k$. The pullbacks under a morphism $S'\to S$ of a scheme $X$ over $S$ and a morphism $f$ of schemes over $S$ are denoted $X_{S'}$ and $f_{S'}$ respectively.

\subsection{Graded Quasi-Coherent Sheaves}
 We denote the category of \emph{graded quasi-coherent sheaves} of $\CO_S$-modules by $\GradQCoh(S)$. Its objects are quasi-coherent sheaves $\CM$ together with a decomposition $\CM=\oplus_{n\in \BZ}\CM^n$ into quasi-coherent subsheaves $\CM^n$. It is a $k$-linear abelian category.

The tensor product of two graded quasi-coherent sheaves $\CM$ and $\CN$ is given by the usual tensor product $\CM\otimes\CN$ together with the decomposition into the subsheaves $(\CM\otimes\CN)^n=\oplus_{i+j=n}\CM^i\otimes\CN^j$. This makes $\GradQCoh(S)$ into a symmetric monoidal category.  An object $\CM\in\GradQCoh(S)$ is rigid if and only if its underlying quasi-coherent sheaf is locally free of finite rank.

There is a natural equivalence between $\GradQCoh(S)$ and the category of quasi-coherent sheaves on $S$ with a representation of $\Gm[S]$. It sends $\CM=\oplus_{n\in\BZ}\CM^n$ to the quasi-coherent sheaf $\CM$ with the representation of $\Gm[S]$ for which $\Gm[S]$ acts on $\CM^n$ through the character $t\mapsto t^n$. Under this equivalence the above tensor product corresponds to the usual tensor product of representations.

 The forgetful functor $\forget\colon \GradQCoh(S)\to \QCoh(S)$ is a faithful exact tensor functor.

With the natural notion of pullbacks, the categories $\GradQCoh(S)$ form a fibered category over $(\Sch/k)$ which we denote by $\GradQCoh$. Since quasi-coherent modules satisfy effective descent, the fibered category $\GradQCoh$ is a stack for the fpqc topology.

\subsection{Graded Fiber Functors}

Let $\CT$ be a Tannakian category over $k$. 
\begin{definition} \label{GFFDef} 
  \begin{itemize}
  \item[(i)] A \emph{graded fiber functor} on $\CT$ over $S$ is an exact $k$-linear tensor functor $\gamma\colon \CT\to\GradQCoh(S)$.
  \item[(ii)] A \emph{morphism} between two graded fiber functors on $\CT$ over $S$ is a tensor morphism.
  \item[(iii)] We denote the resulting \emph{category of graded fiber functors} on $\CT$ over the scheme $S$ by $\UHom^\otimes(\CT,\GradQCoh)(S)$.
  \item[(iv)] For a morphism $S'\to S$ over $k$ composition with the base change functor $\GradQCoh(S)\to\GradQCoh(S')$ gives a pullback functor $\UHom^\otimes(\CT,\GradQCoh)(S)\to \UHom^\otimes(\CT,\GradQCoh)(S')$. With these pullback functors the categories $\UHom^\otimes(\CT,\GradQCoh)(S)$ for varying $S$ form a fibered category over $(\Sch/k)$ which we denote by $\UHom^\otimes(\CT,\GradQCoh)$. 
  \end{itemize}
\end{definition}
\begin{remark}
  Let $\gamma\colon \CT\to\GradQCoh(S)$ be a graded fiber functor. Since $\CT$ is rigid, the tensor functor $\gamma$ factors through the full subcategory of rigid objects of $\GradQCoh(S)$ (see \cite[2.7]{DeligneGroth}). Thus for every $X\in \CT$ the underlying quasi-coherent sheaf of $\gamma(X)$ is locally free of finite rank.
\end{remark}
\begin{definition}
  \begin{itemize}
  \item [(i)]For any two graded fiber functors $\gamma_1,\gamma_2$ on $\CT$ over $S$, we let $\UIsom^\otimes_S(\gamma_1,\gamma_2)$ be the sheaf $(\Sch/S)\to (\text{Sets})$ which sends $S'\to S$ to the set of tensor isomorphisms $(\gamma_1)_{S'}\isoto (\gamma_2)_{S'}$ and morphisms to pullback maps.
  \item [(ii)] If $\gamma_1=\gamma_2$, we denote $\UIsom^\otimes_S(\gamma_1,\gamma_2)$ by $\UAut^\otimes_S(\gamma_1)$.
  \item [(iii)] For graded fiber functors $\gamma_1,\gamma_2$ over $k$-schemes $S_1,S_2$ we let
    \begin{equation*}
       \UIsom^\otimes_k(\gamma_1,\gamma_2)\defeq \UIsom^\otimes_{S_1\times S_2}(\pr_1^*\gamma_1,\pr_2^*\gamma_2)
    \end{equation*}
 where $\pr_i\colon S_1\times S_2\to S_i$ are the projections.
  \end{itemize}
\end{definition}
Given graded fiber functors $\gamma_1,\gamma_2,\gamma_3$ on $\CT$ over $S$, composition of functors gives a morphism $\UIsom^\otimes_S(\gamma_1,\gamma_2)\times_S \UIsom^\otimes_S(\gamma_2,\gamma_3)\to \UIsom^\otimes_S(\gamma_1,\gamma_3)$. In this way $\UAut^\otimes_S(\gamma_1)$ and $\UAut^\otimes_S(\gamma_2)$ become sheaves of groups that act on $\UIsom^\otimes_S(\gamma_1,\gamma_2)$ from the right (resp. from the left).
\begin{construction} \label{chiomega}
  Since $\GradQCoh(S)$ can be identified with the category of representations of $\Gm[S]$, by \cite[II.3.1.1]{Saavedra} the action of $\Gm[S]$ gives an isomorphism
  \begin{equation*}
    \Gm[S]\isoto \UAut^\otimes_S(\forget \colon \GradQCoh(S)\to\QCoh(S)).
  \end{equation*}
 Thus to any graded fiber functor $\gamma$ on $\CT$ over $S$ we can associate the cocharacter 
 \begin{equation*}
   \chi(\gamma)\colon \Gm[S]\cong \UAut^\otimes_S(\forget\colon \GradQCoh(S)\to\QCoh(S)) \to \UAut^\otimes_S(\forget\circ\gamma)
 \end{equation*}
 induced by $\gamma$.
\end{construction}

\begin{definition}
  Let $\GFF(S)$ be the following category: 

Its objects are pairs $(\omega,\chi)$ where $\omega$ is a fiber functor on $\CT$ over $S$ and $\chi\colon \Gm[S]\to \UAut^\otimes_S(\omega)$ a cocharacter. 

A morphism $(\omega,\chi)\to (\omega',\chi')$ in $\GFF(S)$ is a tensor morphism $\lambda\colon \omega\to\omega'$ such that the following diagram, in which the vertical morphism is induced by $\lambda$, commutes:
  \begin{equation*}
    \xymatrix@R=1pt{
      & \UAut^\otimes_S(\omega) \ar[dd] \\
      \Gm[S] \ar[ur]^{\chi} \ar[dr]_{\chi'} & \\
      & \UAut^\otimes_S(\omega') \\
}.
  \end{equation*}

\end{definition}
The following result, due to  Saavedra Rivano, describes graded fiber functors in terms of cocharacters:
\begin{theorem}[\cite{Saavedra} IV.1.3] \label{GFFClassification}
  The functor $\UHom^\otimes(\CT,\GradQCoh)(S)\to \GFF(S)$ which sends a graded fiber functor $\gamma$ to the pair $(\forget\circ\gamma,\chi(\gamma))$ and a morphism of graded fiber functors to the induced morphism of the underlying fiber functors is an equivalence.
\end{theorem}
\begin{corollary} \label{GFFClassificationCor}
  Let $\gamma_,\gamma_2$ be graded fiber functors on $\CT$ over $S$ such that $\forget\circ\gamma_1=\forget\circ\gamma_2$ and set $G\defeq \UAut^\otimes_S(\forget\circ\gamma)$. The functor from Theorem \ref{GFFClassification} induces a natural isomorphism $\UIsom^\otimes_S(\gamma_1,\gamma_2)\cong \UTransp_G(\chi(\gamma_1),\chi(\gamma_2))$.
\end{corollary}
\begin{theorem} \label{CocharLocConstant}
  Let $G$ be a group scheme of finite type over $k$ and $\chi\colon \Gm[S]\to G_S$ a cocharacter.
  \begin{itemize}
  \item [(i)] Assume that $\Gred$ is a smooth subgroup scheme of $G$. Then fpqc-locally on $S$, the cocharacter $\chi$ is conjugate to a cocharacter which factors through $(\Gred)_S$.
  \item [(ii)]  If $G$ is smooth over $k$, the field $k$ is algebraically closed and $S$ is connected and of finite type over $k$, there exist an fpqc covering $S'$ of $S$ and a cocharacter $\chi_0\colon \Gm[k]\to G$ such that $\chi_{S'}$ is conjugate to $(\chi_0)_{S'}$.
  \end{itemize}
\end{theorem}
The condition that $\Gred$ is a smooth subgroup scheme of $G$ is always satisfied if $k$ is perfect, but not in general (see \cite[Example A.8.3]{PseudoReductiveGroups}). However, for any group scheme $G$ of finite type over $k$, there exists a finite field extension $k'$ of $k$ such that $(G_{k'})^\text{red}$ is a smooth subgroup scheme of $G_{k'}$.
\begin{proof}
  (i) By \cite[Th\'eor\`eme XI.6.2]{SGA3II}, the functor $\UTransp_S(\chi,(\Gred)_{S})$, whose $S'$-valued points for any scheme $S'$ over $S$ are those $g\in G(S')$ for which $\leftexp{g}{\chi}_S$ factors through $(\Gred)_S$, is representable by a closed subscheme of $G_S$. We need to prove that $\UTransp_S(\chi,(\Gred)_{S})$ is faithfully flat over $S$. In case $S$ is the spectrum of a field, the cocharacter $\chi$ must factor through $(\Gred)_S$ since $\Gm[S]$ is reduced. Thus the claim is true in this case, which implies that in general it is enough to show that $\UTransp_S(\chi,(\Gred)_{S})$ is flat over $S$. For this we may assume that $S$ is the spectrum of a local ring. After base change from $k$ to the residue field of the closed point $s$ of $S$, we may also assume that $s$ is $k$-rational. 

Let $\chi_0\colon \Gm[k]\to G$ be the fiber of $\chi$ in $s$. Since $\chi_0$ factors through $\Gred$, it suffices to show that $(\chi_0)_S$ and $\chi$ are conjugate. For any $V\in \Rep[G]$ and $i\in \BZ$, let $\gr_\chi^i V_S$ (resp. $\gr_{\chi_0}^i V_S$) be the $i$-th step of the decomposition of $V_S$ defined by $\chi$ (resp $(\chi_0)_S$). By Corollary \ref{GFFClassificationCor} it suffices to construct isomorphisms $\gr_\chi^i V_S\cong \gr_{\chi_0}^i V_S$ for all $i$ and $V$ which are natural in $V$ and compatible with tensor products. For $V\in \Rep[G]$ and $i\in \BZ$, consider be the morphism
\begin{equation*}
  \psi^i_V\colon \gr^i_\chi V_S \into V_S\cong \oplus_{j\in \BZ}\gr^j_{\chi_0}V_S \onto \gr^i_{\chi_0}V_S.
\end{equation*}
Since $\chi$ and $(\chi_0)_S$ agree on the fiber over $s$, on this fiber $\psi_V^i$ is an isomorphism. Since $\gr^i_\chi V_S$ and $\gr^i_{\chi_0} V_S$, being direct summands of $V_S$, are locally free $\CO_S$-modules of finite rank, it thus follows from Nakayama's Lemma that $\psi_V^i$ is an isomorphism. Since it follows directly from the construction of the $\psi^i_V$ that they are natural in $V$ and compatible with tensor products we are done.

(ii) Let $s\in S$ be a closed point. Since $S$ is of finite type over $k$ and $k$ is algebraically closed, the residue field of $s$ is $k$. Let $\chi_0\colon \Gm[k]\to G$ be the fiber of $\chi$ in $s$. Since $G$ is smooth, by \cite[Corollaire XI.5.2]{SGA3II} the transporter $\UTransp_S(\chi,(\chi_0)_S)$ is representable by a closed subscheme of $G_S$ which is smooth over $S$. Thus the image $U$ of $\UTransp_S(\chi,(\chi_0)_S)$ in $S$ is an open subscheme of $S$ and if $U=S$ then $\UTransp_S(\chi,(\chi_0)_S)$ is an fpqc covering of $S$ over which $\chi_0$ and $\chi$ become conjugate. Hence the claim will follow from:
\begin{lemma}
  $U=S$.
\end{lemma}
\begin{proof}
 Otherwise there exists a closed point $s'\in S\setminus U$. Repeating the above construction with $s'$ instead of $s$ yields a cocharacter $\chi_0'\colon \Gm[k]\to G$ and an open subset $U'$ of $S$ containing $s'$ over which $\chi$ and $(\chi'_0)_S$ are fpqc-locally conjugate. Since $S$ is connected, the intersection $U\cap U'$ is not empty and over this intersection all three cocharacters $(\chi_0)_S$, $(\chi'_0)_S$ and $\chi$ are fpqc-locally conjugate. Hence $\chi_0$ and $\chi_0'$ are conjugate over some field extension of $k$. This implies $s'\in U$ which is a contradiction.
\end{proof}
\end{proof}

In the following Lemma, by a statement $P(\gamma_1,\ldots,\gamma_r)$ we simply mean a function from the class of tuples $(\gamma_1,\ldots,\gamma_r)$ as below to the set $\{\text{true},\text{false}\}$.
\begin{lemma}\label{GFFPPrinciple}
  Assume that $\CT$ has a tensor generator. Let $P(\gamma_1,\ldots,\gamma_r)$ be a statement involving finitely many graded fiber functors $\gamma_1,\ldots,\gamma_r$ on $\CT$ over the same scheme which satifies the following conditions:
  \begin{itemize}
  \item [(i)] The statement $P$ is invariant under pullback, i.e: If $P(\gamma_1,\ldots,\gamma_r)$ holds, then so does $P(\gamma_1|_{S'},\ldots,\gamma_r|_{S'})$ for any morphism $S'\to S$ of schemes.
  \item[(ii)] The statement $P$ is local for the fpqc topology, i.e: If $\{S_i\to S\mid i\in I\}$ is an fpqc covering of $S$, then $P(\gamma_1,\ldots,\gamma_r)$ is true if and only if for all $i\in I$ the statement $P(\gamma_1|_{S_i},\ldots,\gamma_r|_{S_i})$ is true.
  \item [(iii)] The statement $P$ is invariant under isomorphisms, i.e: If $\gamma'_1,\ldots,\gamma'_r$ are graded fiber functors on $\CT$ over $S$ such that $\gamma_i$ is isomorphic to $\gamma'_i$ for all $i\in I$, then $P(\gamma_1,\ldots,\gamma_r)$ is true if and only if $P(\gamma'_1,\ldots,\gamma'_r)$ is true.
  \item [(iv)] There exists an algebraically closed overfield $k'$ of $k$ such that $P(\gamma_1,\ldots,\gamma_r)$ is true for all graded fiber functors $\gamma_1,\ldots,\gamma_r$ on $\CT$ over $k'$.
  \end{itemize}
Then $P(\gamma_1,\ldots,\gamma_r)$ is true in general.
\end{lemma}
\begin{proof}
  Let $\gamma_1,\ldots,\gamma_r$ be graded fiber functors on $\CT$ over a $k$-scheme $S$. We need to show that $P(\gamma_1,\ldots,\gamma_r)$ is true. Let $\omega$ be a fiber functor on $\CT$ over $k'$ and let $G\defeq \UAut^\otimes_{k'}(\omega)$. Since $\CT$ has a tensor generator, this is a group scheme of finite type over $k'$. Using condition (ii), after replacing $S$ by $S_{k'}$ we may assume that there is a morphism $S\to\Spec(k')$. Since by Theorem \ref{FFLocalIso} the fiber functors $\forget\circ\gamma_i$ and $\omega_S$ are fpqc-locally isomorphic, after replacing $S$ by a suitable covering we may assume that $\forget\circ\gamma\cong \omega_S$. Thus, using condition (iii), we can replace the $\gamma_i$ by isomorphic graded fiber functors such that $\forget\circ\gamma_i=\omega_S$. By Theorem \ref{GFFClassification}, the $\gamma_i$ are determined by the cocharacters $\chi(\gamma_i)\colon \Gm[S]\to G_S$. By Theorem \ref{CocharLocConstant} (i), after replacing $S$ by a suitable fqpc covering we may assume that these cocharacters factor through $(\Gred)_S$. 

By a standard limit argument there exists a scheme $S'$ of finite type over $k'$ together with a $k'$-morphism $S\to S'$ such that the $\chi(\gamma_i)\colon \Gm[S]\to(\Gred)_S$, and hence the $\gamma_i$, are defined over $S'$. Using condition $(i)$ we may thus assume that $S$ is of finite type over $k'$. Using condition (ii) we may assume that $S$ is connected. Then by Theorem \ref{CocharLocConstant} there exist cocharacters $\chi_1,\ldots,\chi_r\colon \Gm[k']\to \Gred$ and an fpqc covering $S'$ of $S$ such that $\chi(\gamma_i)_{S'}$ is conjugate to $(\chi_i)_{S'}$ for all $i$. By Theorem \ref{GFFClassification}, there exist unique graded fiber functors $\gamma'_i$ on $\CT$ over $k'$ such that $\forget\circ\gamma'_i=\omega$ and $\chi(\forget\circ\gamma'_i)=\chi_i$. Again by Theorem \ref{GFFClassification}, the conjugacy of $\chi(\gamma_i)_{S'}$ and $(\chi_i)_{S'}$ implies that $(\gamma_i)_{S'}$ is isomorphic to $(\gamma'_i)_{S'}$ for all $i$.

By condition (iv) the statement $P(\gamma'_1,\ldots,\gamma'_r)$ is true. By conditions (i) and (iii) this implies that $P((\gamma_1)_{S'},\ldots,(\gamma_1)_{S'})$ is true. Thus a final application of condition (ii) shows that $P(\gamma_1,\ldots,\gamma_r)$ is true.
\end{proof}

\begin{theorem} \label{GFFIsomRepr}
  For any graded fiber functors $\gamma_1,\gamma_2$ over $S$ (resp. $\gamma_1$ over $S_1$ and $\gamma_2$ over $S_2$) the functor $\UIsom^\otimes_S(\gamma_1,\gamma_2)$ (resp. $\UIsom^\otimes_k(\gamma_1,\gamma_2)$) is representable by a scheme which is affine and flat over $S$ (resp. over $S_1\times S_2$). If $\CT$ has a tensor generator, these schemes are of finite presentation over $S$ (resp. over $S_1\times S_2$). 

In particular, for any graded fiber functor $\gamma$ over $S$, the functor $\UAut^\otimes_S(\gamma)$ is representable by a group scheme which is affine and flat over $S$.
\end{theorem}
\begin{proof}
  Since $\UIsom^\otimes_k(\gamma_1,\gamma_2)=\UIsom^\otimes_{S_1\times S_2}(\pr_1^*\gamma_1,\pr_2^*\gamma_2)$, it suffices to prove the claim about $\UIsom^\otimes_S(\gamma_1,\gamma_2)$. First we treat the case that $\CT$ has a tensor generator. By Lemma \ref{GFFPPrinciple} we may assume that $S$ is the spectrum of an algebraically closed field $k'$. Since all fiber functors on $\CT$ over $k'$ are isomorphic we may also assume that $\forget\circ\gamma_1=\forget\circ\gamma_2$. Let $G\defeq \UAut^\otimes_{k'}(\forget\circ\omega_1)$ which is a group scheme of finite type over $k'$. Then $\UIsom^\otimes_{k'}(\gamma_1,\gamma_2)\cong \UTransp_{k'}(\chi(\gamma_1),\chi(\gamma_2))$ by Corollary \ref{GFFClassificationCor}. This functor is representable by a closed subscheme of $G$ which is in particular of finite type over $k'$ and trivially flat over $k'$. Thus the claim is true in this case.

  Now let $\CT$ be an arbitrary Tannakian category. The category $\CT$ is the filtered colimit of the set $I_\CT$ of Tannakian subcategories from in Subsection \ref{TensorGenSection} which possess a tensor generator. Then $\UIsom^\otimes_S(\gamma_1,\gamma_2)$ is in a natural way the filtered limit of the functors $\UIsom^\otimes_S(\gamma_1|_{\CT'},\gamma_2|_{\CT'})$ over $\CT'\in I_\CT$. By the previous case all these functors are representable by schemes which are affine and flat over $S$. By \cite[Proposition 8.2.3]{EGA4III} this implies that $\UIsom^\otimes_S(\gamma_1,\gamma_2)$ is representable by a scheme which is affine over $S$ and by \cite[Proposition 8.3.8]{EGA4III} this scheme is flat over $S$.
\end{proof}

\subsection{The Stack of Graded Fiber Functors} \label{StGFF}
Let $\CT$ be a Tannakian category over $k$ possessing a tensor generator. In this subsection we show that $\UHom^\otimes(\CT,\GradQCoh)$ is an algebraic stack. By an algebraic stack we mean the same as in \cite[D\'efinition 4.1]{ChampsAlgebriques}.

\begin{lemma} \label{GFFStack}
  The fibered category $\UHom^\otimes(\CT,\GradQCoh)$ is a stack for the fpqc topology.
\end{lemma}
\begin{proof}
  Since by \cite[III.3.2.1.2]{Saavedra} fiber functors satisfy effective descent for the fpqc topology, this follows from Theorem \ref{GFFClassification}.
\end{proof}
Given two group schemes $G$ and $H$ over a scheme $S$, we denote by $\UHom_S(G,H)$ the fpqc sheaf $(\Sch/S)^\text{op}\to (\text{Sets})$ which associates to a scheme $S'$ over $S$ the set of group homomorphisms $G_{S'}\to H_{S'}$ and acts on morphisms in the natural way.

\begin{construction}
  Let $k'$ be an overfield of $k$ and $\omega$ a fiber functor on $\CT$ over $k'$. Let $G\defeq \UAut^\otimes_{k'}(\omega)$. Since $\CT$ has a tensor generator, this is a group scheme of finite type over $k'$. By \cite[Remarque XI.4.3]{SGA3II} the functor $\UHom_{k'}(\Gm[k'],G)$ is representable by a scheme which is separated and locally of finite type over $k'$. The group scheme $G$ acts on this scheme from the left by conjugation. We denote the associated quotient stack by $[G\backslash \UHom_{k'}(\Gm[k'],G)]$. We construct a morphism of stacks
  \begin{equation*}
    F_\omega\colon \UHom^\otimes(\CT,\GradQCoh)_{k'} \to [G\backslash \UHom_{k'}(\Gm[k'],G)]
  \end{equation*}
as follows:

Recall that for a scheme $S$ over $k'$, an object of $[G\backslash \UHom_{k'}(\Gm[k'],G)](S)$ is a pair $(X,f)$ consisting of a left $G_S$-torsor $X$ over $S$ and a $G_S$-equivariant morphism $f\colon X\to \UHom_{k'}(\Gm[k'],G)_S$ over $S$. A morphism $(X,f)\to (X',f')$ in $[G\backslash \UHom_{k'}(\Gm[k'],G)](S)$ is a $G_S$-equivariant morphism $g\colon X\to X'$ such that $f=f'\circ g$.

Let $\gamma\in \UHom^\otimes(\CT,\GradQCoh)(S)$. Let $X_{\gamma}\defeq \UIsom^\otimes_S(\forget\circ\gamma,\omega)$, which is a left $G_S$-torsor by composition of isomorphisms. Any section $\lambda \in X_{\gamma}(S')$ for some scheme $S'$ over $S$ induces an isomorphism $\UAut^\otimes_{S'}(\forget\circ\gamma)\to G_{S'}$, and we let $f_{\gamma}(\lambda)\in \UHom_S(\Gm[S],G_S)(S')$ be the composition of the cocharacter $\chi(\gamma)_{S'}$ with this isomorphism. This defines a $G_S$-equivariant morphism $f_{\gamma}\colon X_{\gamma}\to \UHom_S(\Gm[S],G_S)$. Altogether we obtain an object $F_\omega(\gamma)\defeq (X_{\gamma},f_{\gamma})\in [G\backslash \UHom_{k'}(\Gm[k'],G)](S)$.

To define $F_\omega$ on morphisms, consider a morphism $\lambda'\colon \gamma\to \gamma'$ in $\UHom^\otimes(\CT,\GradQCoh)(S')$. Such a morphism defines a $G_S$-equivariant isomorphism $F_\omega(\lambda') \colon X_{\gamma}\to X_{\gamma'}$ by composition of isomorphisms. The cocharacter $\chi(\gamma')\colon \Gm[S']\to \UAut^\otimes_{S'}(\gamma')$ is the composition of $\chi(\gamma)\colon \Gm[S']\to \UAut^\otimes_{S'}(\gamma)$ and the isomorphism $\UAut^\otimes_{S'}(\gamma)\to\UAut^\otimes_{S'}(\gamma')$ induced by $\lambda'$. This implies $f_\gamma=f_{\gamma'}\circ F_\omega(\lambda')$. Thus $F_\omega(\lambda')$ is a morphism $F_\omega(\gamma)\to F_\omega(\gamma')$ in $[G\backslash \UHom_{k'}(\Gm[k'],G)](S)$. 

By a direct verification, this construction is functorial and compatible with pullbacks, so that we obtain a morphism $F_\omega$ of stacks, as desired.
\end{construction}
\begin{lemma} \label{FomegaEquiv}
  For each $\omega$ as above, the associated morphism of stacks $F_\omega$ is an equivalence.
\end{lemma}
\begin{proof}
  First we prove that $F_\omega$ is fully faithful. Let $\gamma,\gamma' \in \UHom^\otimes(\CT,\GradQCoh)_{k'}(S)$ for some $k'$-scheme $S$. Since fiber functors on $\CT$ satisfy effective descent and are pairwise fpqc-locally isomorphic, they form a gerbe. Thus Theorem \ref{GerbeClassification} gives an equivalence between fiber functors on $\CT$ over $S$ and left $G_S$-torsors which sends a fiber functor $\omega'$ to $\UIsom^\otimes_S(\omega_S,\omega')$. In particular giving a tensor morphism $\lambda'\colon \forget\circ\gamma\to\forget\circ\gamma'$ is the same as giving the $G$-equivariant morphism $g_{\lambda'}\colon X_{\gamma}\to X_{\gamma'}$ induced by $\lambda$. By a direct verification $f_\gamma=f_{\gamma'}\circ g_{\lambda'}$ if and only if $\lambda$ is a morphism $(\forget\circ\gamma,\chi(\gamma))\to (\forget\circ\gamma',\chi(\gamma'))$ in $\GFF(S)$. Thus it follows from Theorem \ref{GFFClassification} that $F_\omega$ is fully faithful.

It remains to prove that $F_\omega$ is an epimorphism. Let $(X,f)\in [G\backslash \UHom_{k'}(\Gm[k'],G)](S)$ for some $k$-scheme $S$. After passing to a suitable fpqc covering, we may assume that $X$ has a section over $S$. By fixing such a section we can identify $X$ with $G_S$. Then $f$ is determined by the cocharacter $\chi\in \UHom_{k'}(\Gm[k'],G)(S)$ which is the image of this section under $f$. By Theorem \ref{GFFClassification} there is a graded fiber functor $\gamma$ over $S$ such that $\forget\circ\gamma=\omega_S$ and $\chi(\gamma)=\chi$. By chasing through the definition of $F_\omega$ one can directly verify that $F_\omega(\gamma)=(X,f)$. Thus $F_\omega$ is an equivalence.
\end{proof}
The following lemma is probably well-known, but we could not find a reference.
\begin{lemma} \label{AlgebraicStackLocal}
  Let $\CX$ be a fibered category over $(\Sch/S)$ for some scheme $S$ which is a stack for the fppf topology. If there exists an fppf covering $S'$ of $S$ such that $\CX_{S'}$ is an algebraic stack, then $\CX$ is an algebraic stack.
\end{lemma}
\begin{proof}
  By \cite[Th\'eor\`eme 10.1]{ChampsAlgebriques} it is sufficient to show that the diagonal morphism $\CX\to \CX\times_S\CX$ is representable, separated and quasi-compact and that there exists an algebraic space $Y$ and a morphism $Y\to\CX$ which is representable, faithfully flat and locally of finite presentation.

The first condition can be checked fppf-locally on $S$ and thus holds by assumption. Since $\CX_{S'}$ is an algebraic stack, there exists an algebraic space $Y$ together with a morphism $Y\to\CX_{S'}$ which is representable, faithfully flat and of finite presentation. Composing this morphism with the natural morphism $\CX_{S'}\to \CX$ which is also representable, faithfully flat and of finite presentation because it is a pullback of the fppf covering $S'\to S$ yields the required morphism $Y\to\CX$.
\end{proof}
\begin{theorem} \label{GFFStackAlgebraic}
  The stack $\UHom^\otimes(\CT,\GradQCoh)$ is an algebraic stack locally of finite type over $k$.
\end{theorem}
\begin{proof}
  Since $\CT$ is rigid, all morphisms between graded fiber functors are isomorphisms by \cite[I.5.2.3]{Saavedra}. Thus $\UHom^\otimes(\CT,\GradQCoh)$ is fibered in groupoids. By \cite[6.20]{DeligneGroth} there exists a finite field extension $k'$ of $k$ such that there exists a fiber functor $\omega$ on $\CT$ over $k'$. By Lemma \ref{FomegaEquiv} applied to this $\omega$, the stack $\UHom^\otimes(\CT,\GradQCoh)_{k'}$ is algebraic and locally of finite type over $k$. By Lemma \ref{AlgebraicStackLocal} this implies the claim.
\end{proof}

\subsection{The Scheme of Types of $\CT$}
Let $\CT$ be a Tannakian category over $k$ possessing a tensor generator. In this subsection we develop the notion of the type of a graded fiber functor, in analogy with the notion of the type of a graded module.
\begin{definition} \label{SchemeOfTypesDef}
  We denote by $\CC_\CT$ the coarse fpqc sheaf on $k$ associated to the stack $\UHom^\otimes(\CT,\GradQCoh)$, i.e. the sheafification of the presheaf which associates to a $k$-scheme $S$ the set of isomorphism classes of $\UHom^\otimes(\CT,\GradQCoh)(S)$ and acts on morphisms in the natural way. We denote the natural morphism $\UHom^\otimes(\CT,\GradQCoh)\to \CC_\CT$ by $t$.
\end{definition}

\begin{lemma} \label{ConjClassIso}
  Let $G$ be an affine group scheme of finite type over $k$ such that $\Gred$ is a smooth subgroup scheme. Then the morphism of fpqc quotient sheaves
  \begin{equation*}
    \Gred\backslash \UHom_k(\Gm[k],\Gred)\to G\backslash \UHom_k(\Gm[k],G)
  \end{equation*}
 induced by the inclusion $\Gred\into G$ is an isomorphism.
\end{lemma}

\begin{proof}
  That it is an epimorphism is a consequence of Theorem \ref{CocharLocConstant}. To prove that it is a monomorphism we need to show that any two cocharacters $\chi_1,\chi_2$ of $\Gred$ over some scheme $S$ which are conjugate by an element of $G(S)$ are fpqc-locally on $S$ conjugate by an element of $\Gred(S)$. For any point $s\in S$, the fibers of $\chi_1$ and $\chi_2$ in $s$ are conjugate by an element of $\Gred(k(s))=G(k(s))$. Since $\Gred$ is smooth, by \cite[Corollaire XI.5.4]{SGA3II} this implies that $\chi_1$ and $\chi_2$ are conjugate by an element of $\Gred(S)$ fpqc-locally on $S$.
\end{proof}

For us, an \'etale morphism of schemes need not be of finite type, just locally of finite type as in \cite[Tag 02GI]{StacksProject}.
\begin{theorem} \label{ClRepr}
  \begin{itemize}
  \item [(i)] The sheaf $\CC_\CT$ is representable by a scheme which is \'etale over $k$ and the morphism $t\colon\UHom^\otimes(\CT,\GradQCoh)\to \CC_\CT$ is faithfully flat and locally of finite presentation.
  \item [(ii)] The scheme $\CC_\CT$ is the coarse moduli space of $\UHom^\otimes(\CT,\GradQCoh)$, i.e., the morphism $\CC_\CT\to\UHom^\otimes(\CT,\GradQCoh)$ is universal among all morphisms from $\UHom^\otimes(\CT,\GradQCoh)$ to a scheme and induces a bijection between the set of isomorphism classes in $\UHom^\otimes(\CT,\GradQCoh)(k')$ and the set $\CC_\CT(k')$ for any algebraically closed overfield $k'$ of $k$.
  \item [(iii)] The stack $\UHom^\otimes(\CT,\GradQCoh)$ is a gerbe over $\CC_\CT$.
  \end{itemize}
\end{theorem}
\begin{proof}
 (i) By Theorem \ref{GFFIsomRepr}, for any $\gamma\in\UHom^\otimes(\CT,\GradQCoh)(S)$ the sheaf $\UAut^\otimes_S(\gamma)$ is flat and of finite type over $S$. By \cite[Corollaire 10.8]{ChampsAlgebriques} this implies that the coarse fppf sheaf associated to $\UHom^\otimes(\CT,\GradQCoh)$ is representable by an algebraic space over $k$ and that the natural morphism $\UHom^\otimes(\CT,\GradQCoh)\to \CC_\CT$ is faithfully flat and locally of finite presentation. Then, as an algebraic space, this fppf sheaf is an fpqc sheaf by \cite[Th\'eor\`eme A.4]{ChampsAlgebriques} which implies that it coincides with $\CC_\CT$. Since by \cite[Tag 03KX]{StacksProject} any algebraic space which has an \'etale morphism to the spectrum of a field is a scheme, it suffices to prove that $\CC_\CT$ is \'etale over $k$. For this we may replace $k$ by an algebraic closure. 

Then using Lemma \ref{FomegaEquiv} we can identify $\UHom^\otimes(\CT,\GradQCoh)$ with $[G\backslash \UHom_k(\Gm[k],G)]$ where $G=\UAut^\otimes_k(\omega)$ for some fiber functor $\omega$ on $\CT$ over $k$ so that $\CC_\CT$ is identified with the fpqc quotient sheaf $G\backslash \UHom_k(\Gm[k],G)$. Using Lemma \ref{ConjClassIso} we may assume that $G$ is smooth. Let $\chi\in \UHom_k(\Gm[k],G)$. Since $G$ is smooth, by \cite[Corollaire XI.5.3]{SGA3II}, the subsheaf of $\UHom_k(\Gm[k],G)$ consisting of those cocharacters which are fpqc-locally conjugate to $\chi$ is representable by an open and closed subscheme of $\UHom_k(\Gm[k],G)$ which we denote by $U$. This subscheme is preserved by the action of $G$ on $\UHom_k(\Gm[k],G)$. Thus we obtain an open and closed substack $[G\backslash U]$ of $[G\backslash \UHom_k(\Gm[k],G)]$. Let $X$ be the coarse fpqc sheaf associated to $[G\backslash U]$. By an argument analogous to the one used above to show that $\CC_\CT$ is representable, it follows from \cite[Remarque 10.9]{ChampsAlgebriques} that $X$ is representable by an algebraic space. There is a natural inclusion $X \into \CC_\CT$ whose pullback along the fpqc covering $[G\backslash \UHom_k(\Gm[k],G)] \to \CC_\CT$ is the open and closed immersion $[G\backslash U]\into [G\backslash \UHom_k(\Gm[k],G)]$. Thus $X$ is an open and closed subspace of $\CC_\CT$. It follows from the definition of $U$ that for any scheme $S$ over $k$ the set $X(S)$ consists of exactly one point. Thus $X$ is $\Spec(k)$.

 By varying the cocharacter $\chi$ we can cover $\UHom_k(\Gm[k],G)$ with open and closed subsets $U$ as above, and it follows that $\CC_\CT$ is a disjoint union of copies of $\Spec(k)$. This means that $\CC_\CT$ is \'etale over $k$.

(ii) and (iii) follow from (i).
\end{proof}
\begin{definition}
  \begin{itemize}
  \item[(i)] The scheme $\CC_\CT$ is called the \emph{scheme of types} of $\CT$.
  \item[(ii)] For any graded fiber functor $\gamma\in \UHom^\otimes(\CT,\GradQCoh)(S)$, the \emph{type of $\gamma$} is the image $t(\gamma)$ of $\gamma$ in $\CC_\CT(S)$.
  \end{itemize}
\end{definition}

%For a group scheme $G$ over $k$, we denote by $\Rep[G]$ the category of finite-dimensional representations of $G$.
\begin{proposition} \label{SchemeOfTypeIso}
  Let $\CT=\Rep[G]$ for some group scheme $G$ of finite type over $k$ such that $\Gred$ is a smooth subgroup scheme of $G$. Let $\CT'\defeq \Rep[\Gred]$. Then the restriction functor $\CT\to\CT'$ induces a morphism of stacks $\UHom^\otimes(\CT',\GradQCoh)\to \UHom^\otimes(\CT,\GradQCoh)$ by composition of functors, and this in turn induces a morphism of schemes $\CC_{\CT'}\to\CC_\CT$. This morphism $\CC_{\CT'}\to\CC_\CT$ is an isomorphism. 
\end{proposition}
\begin{proof}
  By Lemma \ref{FomegaEquiv}, the stacks $\UHom^\otimes(\CT,\GradQCoh)$ (resp. $\UHom^\otimes(\CT',\GradQCoh)$) can be identified with $[G\backslash \UHom_k(\Gm[k],G)]$ (resp. $[\Gred\backslash \UHom_k(\Gm[k],\Gred)]$). Then $\CC_{\CT}$ (resp. $\CC_{\CT'}$) are the fpqc quotient sheaves $G\backslash \UHom_k(\Gm[k],G)$ (resp. $\Gred\backslash \UHom_k(\Gm[k],\Gred)$) and the morphism $\CC_{\CT'}\to \CC_\CT$ in question is the morphism considered in Lemma \ref{ConjClassIso}. Thus the Theorem follows from Lemma \ref{ConjClassIso}.
\end{proof}
\subsection{Explicit Description of $\CC_\CT$} \label{ExplicitST}
Let $\CT$ be a Tannakian category over $k$ possessing a tensor generator and let $\bar k$ be an algebraic closure of $k$. In this subsection, we will describe $\CC_\CT$ by giving the set $\CC_\CT(\bar k)$ together with its $\Aut(\bar k/ k)$-action. Since $\CC_\CT$ is \'etale over $k$, these data determine $\CC_\CT$ uniquely. In this way we will obtain a description of $\CC_\CT$ not involving the language of stacks.

 The following lemma is a generalization of \cite[Lemma C.3.5]{PseudoReductiveGroups}. The proof we give is essentially the same as the one given in [loc.cit.].
\begin{lemma} \label{CocharT}
  Let $G$ be a group scheme which is locally of finite type over a field $k'$ and let $T$ be a maximal split subtorus of $G$. Then the inclusion $T\into G$ induces a bijection 
  \begin{equation*}
    \Norm_G(T)(k')\backslash \UHom_{k'}(\Gm[k'],T)(k') \isoto G(k')\backslash \UHom_{k'}(\Gm[k'],G)(k')
  \end{equation*}
\end{lemma}
\begin{proof}
  By \cite[Proposition C.4.5]{PseudoReductiveGroups} any two maximal split tori in $G$ are conjugate under $G(k')$. Hence for $\chi\in \UHom_{k'}(\Gm[k'],G)$ there exists $g\in G(k')$ such that $\leftexp{g}{\chi}$ factors through $T$ which shows surjectivity.

  To prove injectivity, consider $\chi\in \UHom_{k'}(\Gm[k'],T)$ and $g\in G(k')$ such that $\leftexp{g}{\chi}$ factors through $T$. Then $\Cent_G(\chi)$ contains the maximal split tori $T$ and $\leftexp{g^{-1}}{T}$. Thus by [loc. cit.] there exists $z\in \Cent_G(\chi)(k')$ such that $\leftexp{z^{-1}g^{-1}}{T}=T$. The element $gz$ lies in $\Norm_G(T)(k')$ and satisfies $\leftexp{gz}{\chi}=\leftexp{g}{\chi}$.
\end{proof}

 For any group scheme $G$ over $\bar k$ we denote by $X_*(G)\defeq \UHom_{\bar k}(\Gm[\bar k],G)(\bar k)$ its set of cocharacters over $\bar k$. Let $\omega$ be a fiber functor on $\CT$ over $\bar k$ and let $G\defeq \UAut^\otimes_{\bar k}(\omega)$, which is an affine group scheme of finite type over $\bar k$. Let $T\subset G$ be a maximal torus. The action of $\Norm_G(T)(\bar k)$ on $X_*(T)$ from the left by conjugation induces an action of $W_G(T)\defeq \Norm_G(T)(\bar k)/\Cent_G(T)(\bar k)$ on $X_*(G)$.

Consider another fiber functor $\omega'$ on $\CT$ over $\bar k$ and let $G'\defeq \UAut^\otimes_{\bar k}(\omega')$. Since we are over an algebraically closed field, the fiber functors $\omega$ and $\omega'$ are isomorphic by Theorem \ref{FFLocalIso}. Any isomorphism $\lambda\colon \omega\to\omega'$ induces an isomorphism $G\to G'$ which is independent of $\lambda$ up to an inner automorphism. In particular the induced bijection $G(\bar k)\backslash X_*(G) \isoto G'(\bar k)\backslash X_*(G')$ is independent of $\lambda$. If $T'\subset G'$ is a maximal torus, using the conjugacy of maximal tori we may choose $\lambda$ such that the induced isomorphism $G\isoto G'$ maps $T$ to $T'$. This yields a canonical bijection $W_{G}(T)\backslash X_*(T) \isoto W_{G'}(T')\backslash  X_*(T')$.

We shall apply this to group schemes $G'$ and $T'$ obtained as follows: For $\sigma \in \Aut(\bar k/ k)$, the automorphism group of the fiber functor $\sigma^*\omega$ is $\UAut^\otimes_{\bar k}(\sigma^*\omega)=\sigma^*G$ and $\sigma^*T$ is a maximal torus of $\sigma^*G$.

\begin{theorem} \label{ConcreteCL}
  There is a bijection $\CC_\CT(\bar k) \cong W_G(T)\backslash X_*(T)$ under which the Galois action on $\CC_\CT(\bar k)$ is given as follows: Let $\sigma\in \Aut(\bar k/k)$ and $\chi\in X_*(T)$. The image under $\sigma$ of the class of $\chi$ is the image of the class of  $\sigma^*\chi\in X_*(\sigma^*T)$ under the canonical bijection $W_{\sigma^*G}(\sigma^*T)\backslash X_*(\sigma^*T)\cong W_G(T)\backslash X_*(T)$ constructed above.
\end{theorem}
\begin{proof}
  By Lemma \ref{FomegaEquiv} the functor $F_\omega$ induces a bijection $\CC_\CT(\bar k)\cong G(\bar k)\backslash X_*(G)$. By composing this bijection with the bijection $W_G(T)\backslash X_*(T) \cong \Norm_G(T)(\bar k)\backslash X_*(T) \cong G(\bar k)\backslash X_*(G)$ given by Lemma \ref{CocharT} we obtain the desired map. The claim about the Galois action follows from a direct verification using the definition of $F_\omega$.
\end{proof}

In case $\CT$ is neutral, the description of $\CC_\CT$ given by the preceding theorem can be simplified:
\begin{theorem}\label{ConcreteCL2}
  Let $\CT=\Rep[G]$ for an affine group scheme $G$ of finite type over $k$. Let $T\subset G$ be a maximal torus. Then there is a bijection $\CC_\CT(\bar k)\cong W_G(T)\backslash X_*(T)$ under which the action of $\Aut(\bar k/k)$ on $\CC_\CT(\bar k)$ corresponds to the natural action of $\Aut(\bar k/k)$ on $W_G(T)\backslash X_*(T)$.
\end{theorem}
\begin{proof}
  This follows from Theorem \ref{ConcreteCL} by taking $\omega$ to be the forgetful functor $\Rep[G]\to\QCoh(\Spec(k))$, since in the present situation the $\Aut(\bar k/k)$-action on $W_G(T)\backslash X_*(T)$ described there coincides with the natural action of $\Aut(\bar k/k)$.
\end{proof}
\subsection{Consequences}\label{GFFConsequences}
Let $\CT$ be a Tannakian category over $k$ possessing a tensor generator. The above results on the scheme of types have the following consequences:
\begin{theorem} \label{GFFIsom}
  Let $\gamma_1,\gamma_2$ be graded fiber functors on $\CT$ over connected schemes $S_1,S_2$. Then exactly one of the following is true:
  \begin{itemize}
  \item[(i)] $\UIsom^\otimes_k(\gamma_1,\gamma_2)$ is the empty scheme.
  \item[(ii)] The types $t(\gamma_1)$ and $t(\gamma_2)$ both have the same same set-theoretic image in $\CC_\CT$.
 \end{itemize}
 In the second case, the image of $t(\gamma_1)$ and $t(\gamma_2)$ consists of one point $x$ whose residue field $k(x)$ is a finite separable field extension of $k$. Then the morphisms $t(\gamma_i)\colon S_i\to \CC_\CT$ factor through $\Spec(k(x))\subset \CC_\CT$ and the morphism $\UIsom^\otimes_k(\gamma_1,\gamma_2) \to S_1\times S_2$ factors through a faithfully flat morphism $\UIsom^\otimes_k(\gamma_1,\gamma_2)\to S_1\times_{k(x)} S_2$.
\end{theorem}
\begin{proof}
Let $S$ be a scheme over $S_1\times S_2$. By the definition of $\CC_\CT$, the pullbacks of $\gamma_1$ and $\gamma_2$ to $S$ are isomorphic fpqc-locally on $S$ if and only if the two morphisms $t(\gamma_i|_S)\colon S\to S_i\ltoover{\gamma_i} \UHom^\otimes(\CT,\GradQCoh) \to \CC_\CT$ are equal, i.e., if and only if the morphism $S\to S_1\times S_2$ factors through $S_1\times_{\CC_\CT} S_2$. Since the $S_i$ are connected and $\CC_\CT$ is \'etale, the set-theoretic images of the two morphisms $t(\gamma_i)$ both consist of a single point $x_i$. If $x_1\not= x_2$, the two morphisms $t(\gamma_i|_S)$ cannot be equal, which implies that $\UIsom^\otimes_k(\gamma_1,\gamma_2)$ is empty. Otherwise we are in the second case. Let $x\defeq x_1=x_2$. Then the morphisms $t(\gamma_i)$ factor through $\Spec(k(x))$ and thus $S_1\times_{\CC_\CT} S_2=S_1\times_{k(x)} S_2$. If we let $S$ be the spectrum of an algebraically closed field, the above implies that the set-theoretic image of $\UIsom^\otimes_k(\gamma_1,\gamma_2)$ in $S_1\times S_2$ is $S_1\times_{k(x)} S_2$. Since $\UIsom^\otimes_k(\gamma_1,\gamma_2)$ is flat over $S_1\times S_2$ by Theorem \ref{GFFIsomRepr} this implies the last statement.
\end{proof}
\begin{theorem} \label{GFFLocIso}
  Two graded fiber functors on $\CT$ over some scheme $S$ have the same type if and only if they are fpqc-locally isomorphic.
\end{theorem}
\begin{proof}
  This is merely a rephrasing of the fact that $\UHom^\otimes(\CT,\GradQCoh)$ is a gerbe over $\CC_\CT$.
\end{proof}

\begin{theorem} \label{GFFClassification2}
  Let $\gamma$ be a graded fiber functor on $\CT$ over some scheme $S$ and $S'$ a scheme over $S$. Then the functor which sends a graded fiber functor $\gamma'$ over $S'$ to $\UIsom^\otimes_{S'}(\gamma',\gamma_{S'})$ and acts on morphisms in the natural way gives an equivalence between graded fiber functors over $S'$ having the same type as $\gamma$ and left $\UAut^\otimes_{S}(\gamma)_{S'}$-torsors.
\end{theorem}
\begin{proof}
  The fiber of the gerbe $\UHom^\otimes(\CT,\GradQCoh)$ over $\CC_\CT$ in the point $t(\gamma)$ is a neutral gerbe over $S$. Its $S'$-valued points are exactly the $\gamma'$ of the theorem. Thus the theorem follows from Theorem \ref{GerbeClassification}.
\end{proof}

\begin{definition}
  Assume that $\CT=\Rep[G]$ for an affine group scheme $G$ over $k$. Let $\chi\colon \Gm[S]\to G_S$ be a cocharacter. By Theorem \ref{GFFClassification} there exists a unique graded fiber functor $\gamma$ on $\CT$ over $S$ such that $\forget\circ\gamma$ is the base change of the forgetful functor $\Rep[G]\to \QCoh(k)$ to $S$ and such that $\chi(\gamma)=\chi$. We call $\gamma$ the \emph{graded fiber functor associated to $\chi$}.
\end{definition}

  Recall that a connected reductive algebraic group $G$ over $k$ is called \emph{quasi-split} if there exists a Borel subgroup of $G$ defined over $k$. We extend the same definition to arbitrary smooth connected affine group schemes over $k$. 

\begin{theorem} \label{TypeRealizable}
  Let $\CT=\Rep[G]$ where $G$ is a smooth connected quasi-split affine group scheme over $k$. For any point $x\in \CC_\CT$ there exists a cocharacter $\chi\colon \Gm[k(x)]\to G_{k(x)}$ such that the type of the associated graded fiber functor is the natural morphism $\Spec(k(x))\to \CC_\CT$ with set-theoretic image $\{x\}$.
\end{theorem}
\begin{proof}
  Let $B\subset G$ be a Borel subgroup and $T\subset B$ a maximal torus. Under the bijection of Theorem \ref{ConcreteCL2}, the point $x$ corresponds to a $W_G(T)$-conjugacy class in $X_*(T)$ which is $\Aut(\bar k/ k(x))$-invariant. The Borel subgroup $B$ defines a set of simple coroots in $X_*(T)$ and thus a set of dominant cocharacters in $X_*(T)$. It follows from \cite[Theorem V.3.3.2]{BourbakiGAL4} that the latter is a system of representatives for the action of $W_G(T)$ on $X_*(T)$. Let $\chi$ be the unique dominant cocharacter in the conjugacy class corresponding to $x$. Then for any $\sigma\in\Aut(\bar k/k(x))$, the image $\sigma^*\chi$ of $\chi$ under $\sigma$ lies in the same conjugacy class. Since $B$ is defined over $k(x)$, the cocharacter $\sigma^*\chi$ is also dominant. Thus $\sigma^*\chi=\chi$ which shows that $\chi$ is defined over $k(x)$. By construction the graded fiber functor associated to $\chi$ has the required type.
\end{proof}
\begin{corollary}
  Let $\CT=\Rep[G]$ where $G$ is a smooth connected quasi-split affine group scheme over $k$. Let $S$ be a connected $k$-scheme and $\gamma$ a graded fiber functor on $\CT$ over $S$. Then the image of $t(\gamma)$ consists of a single point $x$ whose residue field $k(x)$ is a finite separable extension of $k$ and there exists a cocharacter $\chi\colon \Gm[k(x)]\to G_{k(x)}$ such that the associated graded fiber functor has the same type as $\gamma$.
\end{corollary}
\begin{proof}
  That the image of $t(\gamma)$ consists of one point $x$ for which $k(x)$ is finite separable over $k$ follows from the fact that $\CC_\CT$ is \'etale over $k$ by Theorem \ref{ClRepr}. The cocharacter obtained by applying Theorem \ref{TypeRealizable} to $x$ has the required properties.
\end{proof}

In case $k$ is finite, we can also prove the following variant of the preceding result in which the automorphism group of the fiber functor is not required to be connected.
\begin{theorem} \label{FFTypeRealizable}
  Assume that $k$ is finite, that $\CT$ is neutral and that $\UAut^\otimes_k(\omega)$ is smooth for some (equivalently any) fiber functor $\omega$ on $\CT$ over $k$. Let $S$ be a connected scheme over $k$ and $\gamma$ a graded fiber functor on $\CT$ over $S$. Then the image of $t(\gamma)$ consists of a single point $x$ whose residue field $k(x)$ is a finite extension of $k$ and there exists a graded fiber functor $\gamma'\colon \CT\to \GradQCoh(k(x))$ which has the same type as $\gamma$.
\end{theorem}
First we prove the following result, which is a more concrete version of Theorem \ref{FFTypeRealizable}:
\begin{theorem} \label{FFTypeRealizable2}
  Assume that $k$ is finite, and let $G$ be a smooth affine group scheme over $k$. Let $C$ be a conjugacy class of cocharacters of $G$ which is defined over a finite field extension $k'$ of $k$. Then there exists an inner form $(G', \phi\colon G'_{\bar k}\isoto G_{\bar k})$ defined over $k'$ and a cocharacter $\chi'\colon \Gm[k']\to G'$ such that $\phi\circ \chi'_{\bar k}$ lies in $C$.
\end{theorem}
\begin{proof}
  Let $\chi\colon \Gm[\bar k]\to G_{\bar k}$ be a cocharacter in $C(\bar k)$. Let $\sigma\in \Gal(\bar k/k')$ be the Frobenius. Since $C$ is defined over $k'$ there exists $g\in G(\bar k)$ such that $g \cdot \leftexp{\sigma}{\chi} \cdot g^{-1}=\chi$. There exists a finite field extension $k''$ of $k'$ over which $\chi$ and $g$ are defined. Let $n\defeq [k''\colon k']$. Then $\leftexp{\sigma^n}{g}=g$, and $g\cdot \leftexp{\sigma}{g}\cdots \leftexp{\sigma^{n-1}}{g} \in G(k'')$ has finite order since $G(k'')$ is finite. By \cite[Exercise I.5.1.2]{SerreGalCoh}, this implies that there exists a cocycle $(a_\tau)_{\tau\in \Gal(\bar k/ k')}$ with values in $G(\bar k)$ such that $a_\sigma=g$. By twisting $G$ with this cocycle, we obtain an inner form $G'$ of $G$ defined over $k'$ together with an isomorphism $\phi\colon G'_{\bar k}\to G_{\bar k}$ such that $\phi(\leftexp{\sigma}{g'})=a_\sigma\cdot \leftexp{\sigma}{\phi(g')} \cdot a_\sigma^{-1}$ for all $g'\in G'(\bar k)$. Let $\chi'\defeq \phi^{-1}\circ \chi_{\bar k}$. This cocharacter of $G'$ is invariant under $\sigma$ and hence is defined over $k'$. Thus it has the required properties.
\end{proof}
\begin{proof}[Proof of Theorem \ref{FFTypeRealizable}]
  That the image of $t(\gamma)$ consists of one point $x$ whose residue field $k(x)$ is finite over $k$ follows from the fact that $\CC_\CT$ is \'etale over $k$ by Theorem \ref{ClRepr}. Pick a fiber functor $\omega$ on $\CT$ over $k$ and let $G\defeq \UAut^\otimes_k(\omega)$. By Lemma \ref{FomegaEquiv}, the point $x\in \CC_\CT$ corresponds to a conjugacy class of cocharacters of $G$ which is defined over $k(x)$. Let $(G', \phi\colon G'_{\bar k}\isoto G_{\bar k})$ and $\chi'\colon \Gm[k(x)]\to G'$ be the inner form of $G$ defined over $k(x)$ and the cocharacter of $G'$ obtained by applying Theorem \ref{FFTypeRealizable2} to $G$ and $C$. The fact that $G'$ is an inner form of $G$ yields in a natural way a fiber functor $\omega'$ on $\CT$ over $k(x)$ for which $\UAut^\otimes_{k(x)}(\omega')=G'$. By Theorem \ref{GFFClassification} there exists a unique graded fiber functor $\gamma'$ on $\CT$ over $k(x)$ such that $\forget\circ\gamma'=\omega'$ and $\chi(\gamma')=\chi'$. The fact that $\phi\circ \chi'_{\bar k}$ lies in $C$ implies that $\gamma'$ has the same type as $\gamma$. Thus $\gamma'$ has the required properties.
\end{proof}
\section{Filtered Fiber Functors} \label{FilFF}

\subsection{Filtered Quasi-Coherent Sheaves}
 By a \emph{decreasing filtration} on a quasi-coherent sheaf $\CM$ on $S$ we mean a family of quasi-coherent subsheaves $(\F^n\CM)_{n\in\BZ}$ satisfying $\F^n \CM\supset \F^{n+1} \CM$ for all $n\in \BZ$ as well as $\cap_{n\in\BZ}\F^i\CM=0$ and $\cup_{n\in\BZ}\F^n\CM=\CM$. A \emph{morphism between two quasi-coherent sheaves $\CM$ and $\CN$ equipped with a decreasing filtration} is a morphism $f\colon \CM\to\CN$ such that $f(\F^n\CM)\subset \F^n\CN$ for all $n\in\BZ$. We denote the resulting category of \emph{filtered quasi-coherent sheaves on $S$} by $\FilQCoh(S)$. 

The category $\FilQCoh(S)$ is $k$-linear but in general not abelian. It has arbitrary kernels, cokernels and direct sums. The failure of $\FilQCoh(S)$ to be abelian can be compensated by giving it the structure of an exact category in the sense of Quillen. Such an exact category is an additive category $\CA$ together with a class of sequences $0\to A'\to A\to A''\to 0$ in $\CA$ satisfying certain axioms. These sequences are called short exact. Any abelian category is given the exact structure for which the short exact sequences are the short exact sequences in the usual sense. An exact functor between exact categories is an additive functor which sends short exact sequences to short exact sequences. For an overview of exact categories, see \cite{Buehler}. 

The exact structure on $\FilQCoh(S)$ is defined as follows: A morphism $f\colon \CM\to \CN$ in $\FilQCoh(S)$ is \emph{admissible} if $f(\F^n\CM)=\F^n \CN\cap\operatorname{Im} f$ for all $n\in \BZ$, and a sequence in $\FilQCoh(S)$ is short exact if its morphisms are admissible and the underlying sequence of quasi-coherent sheaves is short exact. 

For any two quasi-coherent sheaves equipped with a filtration, we equip their tensor product with the filtration given by $\F^n(\CM\otimes\CN)=\sum_{i+j=n}\F^i \CM\otimes\F^j\CN$. This makes $\FilQCoh(S)$ into a symmetric monoidal category.

There is a natural forgetful functor $\forget\colon \FilQCoh(S)\to \QCoh(S)$, which is a faithful exact tensor functor.

To any $\CM\in\FilQCoh(S)$ one can associate the graded quasi-coherent sheaf $\gr\CM\defeq \oplus_{n\in \BZ}\F^n\CM/\F^{n+1}\CM\in\GradQCoh(S)$. A morphism $f\colon \CM\to \CN$ in $\FilQCoh(S)$ naturally induces a morphism $\gr\CM\to \gr\CN$ and this gives a functor $\gr\colon \FilQCoh(S)\to \GradQCoh(S)$. It is an exact tensor functor. In fact, using the Five Lemma, the short exact sequences in $\FilQCoh(S)$ can be described in the following way:
\begin{lemma} \label{FilShortExact}
  For a sequence $0\to\CL\to\CM\to\CN\to 0$ in $\FilQCoh(S)$ the following are equivalent:
  \begin{itemize}
  \item[(i)] The sequence $0\to \CL\to\CM\to\CN\to 0$ is short exact in $\FilQCoh(S)$.
  \item[(ii)] For all $n\in \BZ$ the sequence $0\to \F^n\CL\to\F^n\CM\to\F^n\CN\to 0$ is exact in $\QCoh(S)$.
  \item[(iii)] The sequence $0\to\gr\CL\to\gr\CM\to\gr\CN\to 0$ is exact in $\GradQCoh(S)$.
  \end{itemize}
\end{lemma}

\begin{lemma} \label{FilRigid}
   An object $\CM$ of $\FilQCoh(S)$ is rigid if and only if $\CM$ is a locally free $\CO_S$-module of finite rank and all $\F^n\CM$ are direct summands of $\CM$ Zariski-locally on $S$.
\end{lemma}
\begin{proof}
  Let $\CM\in\FilQCoh(S)$ be rigid. Since $\forget$ and $\gr$ are tensor functors, by \cite[2.7]{DeligneGroth} the objects $\forget(\CM)\in\QCoh(S)$ and $\gr(\CM)\in\GradQCoh(S)$ are rigid. Thus the underlying module $\CM$ is locally free of finite rank, for each $n\in\BZ$ the sheaf $\F^n\CM/\F^{n+1}\CM$ is locally free of finite rank and for $n$ small enough $\F^n\CM=\CM$. Using the exact sequence
  \begin{equation*}
    0\to \F^n \CM/ \F^{n+1}\CM  \to \CM/ \F^{n+1} \CM \to \CM/\F^n\CM \to 0
  \end{equation*}
it follows by induction on $n$ that $\CM/\F^n\CM$ is locally free of finite rank. This implies that all $\F^n\CM$ are direct summands of $\CM$ Zariski-locally on $S$.

Now let $\CM\in\FilQCoh(S)$ satisfy the requirements of the lemma. Then we endow the dual $\CM^\vee$ of the underlying locally free sheaf $\CM$ with the filtration given by $\F^n\CM^\vee=(\F^{1-n}\CM)^\perp$. To prove that is in fact the dual of $\CM$ in $\FilQCoh(S)$, it suffices to prove that the morphisms $\delta\colon \Bone \to \CM^\vee\otimes \CM$ and $\operatorname{ev}\colon \CM\otimes \CM^\vee\to \Bone$ of sheaves respect the gradings on these sheaves. This can be done locally on $S$ where it follows from a direct calculation using a basis of $\CM$ adapted to the filtration on $\CM$.
\end{proof}

Any graded quasi-coherent sheaf $\CM=\oplus_{n\in\BZ}\CM^i$ can be naturally equipped with a decreasing filtration by $\F^n\CM=\oplus_{i\geq n}\CM^i$. A morphism of graded quasi-coherent sheaves is then also a morphism of filtered quasi-coherent sheaves for these filtrations; thus we obtain a functor $\fil\colon\GradQCoh(S)\to \FilQCoh(S)$. It is a faithful exact tensor functor.

We denote the full subcategory of $\FilQCoh(S)$ (resp. $\GradQCoh(S)$) whose objects are the rigid objects of $\FilQCoh(S)$ (resp. $\GradQCoh(S)$) by $\FilLF(S)$ (resp. $\GradLF(S)$). These are again exact tensor categories and the functors $\fil$ and $\gr$ restrict to functors $\fil\colon \GradLF(S)\to \FilLF(S)$ and $\gr \colon \FilLF(S)\to \GradLF(S)$.

For any morphism of schemes $S'\to S$ the pullback functor $\QCoh(S)\to\QCoh(S')$ induces a pullback functor $\FilLF(S)\to\FilLF(S')$. Thus for varying $S$ the categories $\FilLF(S)$ form a fibered category $\FilLF$ over $(\Sch/k)$. The functors $\gr$, $\fil$ and $\forget$ are compatible with these pullbacks.

Analogously to the above one can define the notion of an \emph{increasing filtration} on a quasi-coherent sheaf $\CM$ on $S$.

Let $\CM$ be as above and let $(\F^n\CM)_{n\in\BZ}$ (resp. $(\G_n\CM)_{n\in\BZ}$) be a decreasing (resp. an increasing) filtration on $\CM$. Two such filtrations are called \emph{opposite} if there exists a grading $\CM=\oplus_{n\in\BZ}\CM^n$ of $\CM$ such that $\F^n\CM=\oplus_{i\geq n}\CM^i$ and $\G_n\CM=\oplus_{i\leq n}\CM^i$ for all $n\in\BZ$.
\begin{lemma} \label{OppositeFiltrations}
  The following are equivalent:
  \begin{enumerate}
  \item [(i)] The filtrations $(\F^n\CM)_{n\in\BZ}$ and $(\G_n\CM)_{n\in\BZ}$ are opposite.
  \item [(ii)] The addition morphism $\F^n\CM\oplus \G_{n-1}\CM\to \CM$ is an isomorphism for all $n\in \BZ$.
  \end{enumerate}

\end{lemma}
\begin{proof}
  That (i) implies (ii) follows directly from the definition.

Now assume that (ii) holds. For $n\in\BZ$ let $\CM^n\defeq \F^n\CM\cap \G_n\CM$. It follows directly from (ii) that the addition morphism $\psi\colon \oplus_{n\in\BZ}\CM^n\to \CM$ is a monomorphism. Since both filtrations are exhaustive, every section of $\CM$ is contained in $\F^a\CM\cap \G_b\CM$ for some integers $a\leq b$. Thus to prove that $\psi$ is an epimorphism, it is enough to show that $\F^a\CM\cap \G_b\CM=\oplus_{a\leq n\leq b}\CM^n$ for all $a\leq b$. For fixed $a$ we prove this by induction on $b$. For $a=b$ the claim is trivial. For general $b$, it follows from (ii) that $\F^a\CM\cap \G_b\CM=(\F^a\CM\cap \G_{b-1}\CM)\oplus (\F^b\CM\cap \G_b\CM)$. This identity allows us to reduce the claim for $b$ to the claim for $b-1$, which finishes the induction.
\end{proof}
\subsection{Filtered Fiber Functors}

Let $\CT$ be a Tannakian category over $k$.
\begin{definition}
  \begin{itemize}
  \item[(i)] A \emph{filtered fiber functor} on $\CT$ over $S$ is an exact $k$-linear tensor functor $\CT\to \FilLF(S)$.
  \item[(ii)] A \emph{morphism of filtered fiber functors} is a tensor morphism.
  \item[(iii)] We denote the resulting \emph{category of filtered fiber functors} on $\CT$ over the scheme $S$ by $\UHom^\otimes(\CT,\FilLF)(S)$.
  \item[(iv)] For a morphism $S'\to S$ over $k$ composition with the pullback functor $\FilLF(S)\to\FilLF(S')$ gives a functor $\UHom^\otimes(\CT,\FilLF)(S)\to \UHom^\otimes(\CT,\FilLF)(S')$. With these pullback functors the categories $\UHom^\otimes(\CT,\FilLF)(S)$ form a fibered category over $(\Sch/k)$ which we denote by $\UHom^\otimes(\CT,\FilLF)$. 
  \end{itemize}
\end{definition}

% \begin{remark}
%   Let $\phi\colon \CT\to\GradQCoh(S)$ be a filtered fiber functor. Since $\CT$ is rigid, the tensor functor $\phi$ factors through the full subcategory of rigid objects of $\FilLF(S)$  (see \cite[2.7]{DeligneGroth}). Thus by Lemma \ref{FilRigid} for every $X\in \CT$ the underlying quasi-coherent sheaf of $\phi(X)$ is locally free of finite rank and the subsheaves $\F^n(\phi(X))$ are direct summands Zariski-locally on $S$. Thus our notion of a filtered fiber functor agrees with the one from \cite[IV.2.1.1]{Saavedra}.
% \end{remark}

\begin{definition}
  \begin{itemize}
  \item [(i)]For any two filtered fiber functors $\phi_1,\phi_2$ on $\CT$ over $S$, we let $\UIsom^\otimes_S(\phi_1,\phi_2)$ be the functor $(\Sch/S)\to (\text{Sets})$ which sends $S'\to S$ to the set of tensor isomorphisms $(\phi_1)_{S'}\isoto (\phi_2)_{S'}$ and morphisms to pullback maps.
  \item [(ii)] If $\phi_1=\phi_2$, we denote $\UIsom^\otimes_S(\phi_1,\phi_2)$ by $\UAut^\otimes_S(\phi_1)$.
  \item [(iii)] For filtered fiber functors $\phi_1,\phi_2$ over $k$-schemes $S_1,S_2$ we set 
    \begin{equation*}
      \UIsom^\otimes_k(\phi_1,\phi_2)\defeq \UIsom^\otimes_{S_1\times S_2}(\pr_1^*\phi_1,\pr_2^*\phi_2)
    \end{equation*}
 where $\pr_i\colon S_1\times S_2\to S_i$ are the projections.
  \end{itemize}

\end{definition}
Given filtered fiber functors $\phi_1,\phi_2,\phi_3$ on $\CT$ over $S$, composition of functors gives a morphism $\UIsom^\otimes_S(\phi_1,\phi_2)\times_S \UIsom^\otimes_S(\phi_2,\phi_3)\to \UIsom^\otimes_S(\phi_1,\phi_3)$. In this way  $\UAut^\otimes_S(\phi_1)$ and $\UAut^\otimes_S(\phi_2)$ become sheaves of groups which act $\UIsom^\otimes_S(\phi_1,\phi_2)$ from the right (resp. from the left).

\begin{theorem} \label{FFFIsomRepr}
  For any filtered fiber functors $\phi_1,\phi_2$ over $S$ (resp. $\phi_1$ over $S_1$ and $\phi_2$ over $S_2$) the functor $\UIsom^\otimes_S(\phi_1,\phi_2)$ (resp. $\UIsom^\otimes_k(\phi_1,\phi_2)$) is representable by a scheme which is affine over $S$ (resp. over $S_1\times S_2$). If $\CT$ has a tensor generator, these schemes are of finite presentation over $S$ (resp. $S_1\times S_2$).

In particular, for any filtered fiber functor $\phi$ over $S$, the functor $\UAut^\otimes_S(\phi)$ is representable by a group scheme which is affine over $S$.
\end{theorem}
\begin{proof}
It suffices to prove the claim about $\UIsom^\otimes_S(\phi_1,\phi_2)$. The functor $\forget\colon \FilLF(S)\to\QCoh(S)$ induces a monomorphism $\UIsom^\otimes_S(\phi_1,\phi_2)\into \UIsom^\otimes_S(\forget\circ\phi_1,\forget\circ\phi_2)$ and the latter is affine over $S$ and of finite presentation over $S$ if $\CT$ has a tensor generator. Thus it suffices to show that this monomorphism is a closed immersion. 

Given any morphism $\psi\colon \CF\to\CG$ of quasi-coherent sheaves of finite type on some scheme $S'$ over $S$ and quasi-coherent subscheaves $\CF'$ of $\CF$ (resp. $\CG'$ of $\CG$), the subfunctor of $S'$ consisting of those points on which $\psi$ maps $\CF'$ into $\CG'$ is representable by a closed subscheme of $X$. This follows from the fact that this subfunctor consists of those points on which $(\psi(\CF')+\CG')/\CG'$ is zero, and the latter is representable by a closed subscheme by \cite[Theorem 11.17]{GoertzWedhorn}.

 We take $S'\defeq \UIsom^\otimes_S(\forget\circ\phi_1,\forget\circ\phi_2)$. There is a universal tensor isomorphism $(\phi_1)_{S'}\to(\phi_2)_{S'}$ which for $X\in \CT$ gives an isomorphism $\psi_X\colon \phi_1(X)_{S'}\to\phi_2(X)_{S'}$. Let $n\in \BZ$. It follows from the preceding paragraph that the subfunctor of $S'$ consisting of those points on which $\psi_X(\F^n(\phi_1(X)))=\F^n(\phi_2(X))$ is represented by a closed subscheme. Since $\UIsom^\otimes_S(\phi_1,\phi_2)\subset \UIsom^\otimes_S(\forget\circ\phi_1,\forget\circ\phi_2)$ is the intersection of these subfunctors over all $X\in \CT$ and $n\in \BZ$ this implies the claim.
\end{proof}

\begin{definition} \label{ULPDef}
  Using Theorems \ref{GFFIsomRepr} and \ref{FFFIsomRepr} we can associate to a filtered fiber functor $\phi$ on $\CT$ over $S$ the following group schemes which are affine over $S$:
  \begin{itemize}
  \item[(i)] $P(\phi)\defeq \UAut^\otimes_S(\phi)$
  \item[(ii)] $L(\phi)\defeq \UAut^\otimes_S(\gr\circ\phi)$.
  \item[(iii)] $U(\phi)\defeq \ker(P(\phi)\ltoover{\gr} L(\phi))$
  \end{itemize}
\end{definition}

The definition of $U(\phi)$ can be generalized as follows, c.f. \cite[IV.2.1.4]{Saavedra}.
\begin{definition}
  For any filtered fiber functor $\phi$ on $\CT$ over $S$ and any integer $\alpha\geq 0$ let $U_\alpha(\phi)$ be the subgroup functor of $P(\phi)$ such that for any scheme $S'$ over $S$ the group $U_\alpha(S')$ consists of those elements $g\in P(S')$ which act as the identity on the sheaf $\F^i(\phi(X))/\F^{i+\alpha}(\phi(X))$ for all $X\in \CT$ and $i \in \BZ$.
\end{definition}
In particular $U_0(\phi)=P(\phi)$ and $U_1(\phi)=U(\phi)$.

\subsection{Splittings}
Let $\CT$ be a Tannakian category over $k$ and let $\phi\colon\CT \to \FilLF(S)$ be a filtered fiber functor.

\begin{definition}
  \label{SplittingDef}
  \begin{itemize}
  \item[(i)]   A \emph{splitting} of $\phi$ is a graded fiber functor $\gamma\colon\CT \to \GradQCoh(S)$ such that $\phi=\fil\circ\gamma$.
  \item[(ii)]  The functor $\phi$ is \emph{splittable} if there exists a splitting of $\phi$.
  \item[(iii)]  The functor $\phi$ is \emph{fpqc-locally splittable} if there exists a fpqc covering $S'\to S$ such that the pullback of $\phi$ to $S'$ is splittable. (In \cite{Saavedra}, such a functor is called admissible.)
  \item[(iv)]  The functor $\USpl(\phi)\colon (\Sch/S)\to (\text{Sets})$ is the functor which sends a scheme $S'\to S$ to the set of splittings of $\phi_{S'}$ and which acts on morphisms by pullbacks.
  \end{itemize}
\end{definition}
\begin{lemma} \label{SplittingEquiv}
  Giving a splitting of $\phi$ is the same as giving a cocharacter $\chi\colon \Gm[S]\to P(\phi)$ such that the composite of $\chi$ with $\gr\colon P(\phi)\to L(\phi)$ is the cocharacter $\chi(\gr\circ\phi)\colon \Gm[S]\to L(\phi)$.
\end{lemma}
\begin{proof}
  Let $\gamma$ be a splitting of $\phi$. Then the cocharacter $\chi\defeq \chi(\gamma)\colon \Gm[S]\to \UAut^\otimes_S(\forget\circ\phi)$ factors through $P(\phi)$ and satisfies $\gr\circ\chi=\chi(\forget\circ\phi)$.

On the other hand, let $\chi\colon \Gm[S]\to P(\phi)$ be as above. Then by Theorem \ref{GFFClassification}, there exists a unique graded fiber functor $\gamma$ such that $\forget\circ\gamma=\forget\circ\phi$ and $\chi(\gamma)=\chi$. The identity $\gr\circ\chi=\chi(\gr\circ\phi)$ implies that $\gamma$ is a splitting of $\phi$.
\end{proof}

\begin{definition}
  We call a cocharacter $\chi$ as in Lemma \ref{SplittingEquiv} a \emph{cocharacter which splits $\phi$}.
\end{definition}
Lemma \ref{SplittingEquiv} implies that $\USpl(\phi)$ is a sheaf for the fpqc topology. It also follows from Lemma \ref{SplittingEquiv} that $U(\phi)$ acts on $\USpl(\phi)$: Given a cocharacter which splits $\phi$, its conjugate under a point of $U(\phi)$ again splits $\phi$.
\begin{lemma}[{\cite[IV.2.2.1]{Saavedra}}] \label{SplUPseudotorsor}
  This action makes $\USpl(\phi)$ into a left $U(\phi)$-pseudotorsor, i.e., for each scheme $S$ over $k$ the group $U(\phi)(S)$ acts simply transitively on $\USpl(\phi)(S)$.
\end{lemma}

The following is our main result on filtered fiber functors. In case $\CT$ is neutral and $k$ has characteristic zero and in case $\CT$ is neutral and $\UAut^\otimes_S(\forget\circ\phi)$ is reductive it is due to Deligne, see \cite[IV.2.4]{Saavedra}.
\begin{maintheorem}
  \label{MainTheorem}
  Any filtered fiber functor on $\CT$ is fpqc-locally splittable. 
\end{maintheorem}
The proof of Theorem \ref{MainTheorem} will be given in Section \ref{Proof}. For the rest of this section we assume that it holds.

Lemma \ref{SplUPseudotorsor} and Theorem \ref{MainTheorem} immediately imply:
\begin{theorem}\label{SplUTorsor}
  The functor $\USpl(\phi)$ is a left $U(\phi)$-torsor. 
\end{theorem}

In light of Theorem \ref{MainTheorem} it is natural to ask whether one can split $\phi$ already over $S$ or at least Zariski- or \'etale-locally on $S$. In \cite[IV.2.2.3.1]{Saavedra}, Saavedra Rivano constructs for every non-perfect field $k$ a group scheme $G$ over $k$ and a filtered fiber functor $\Rep[G]\to\FilLF(\Spec(k))$ which only becomes splittable over a non-separable field extension of $k$. 

There also exist filtered fiber functors which split Zariski-locally on $S$ but not globally: Take $\CT\defeq \Rep[\GL_n]$ for some $n\geq 0$. As explained in Section \ref{GerbesSection}, giving a fiber functor $\CT$ over $S$ is the same as giving a $\GL_n$-torsor $I$ over $S$. For convenience here we work with right $\GL_n$-torsors. The fiber functor $\omega$ associated to $I$ sends $V\in \Rep[\GL_n]$ to $I\times^G V_S$ which is defined to be the quotient of $I\times V_S$ by the left $G$-action given by $g\cdot(i,v)\defeq (ig^{-1},gv)$. Furthermore, there is a well-known equivalence between $\GL_n$-torsors over $S$ and locally free coherent $\CO_S$-modules of rank $n$. It sends a locally free coherent $\CO_S$-module $\CM$ of rank $n$ to the $\GL_n$-torsor of isomorphisms $\CO_S^{\oplus n}\to \CM$. Altogether we get an equivalence between locally free coherent $\CO_S$-modules of rank $n$ on $S$ and fiber functors on $\CT$ over $S$. Under this equivalence a fiber functor $\omega$ corresponds to the image of the standard representation $\GL_n$ on $k^n$ under $\omega$. 

Analogously it follows that for any object $\CM\in \FilLF(S)$ whose underlying locally free coherent $\CO_S$-module has rank $n$ there exists a filtered fiber functor $\phi$ on $\CT$ over $S$ which sends the standard representation to $\CM$ and which is unique up to isomorphism. This fiber functor is splittable if and only if the filtered module $\CM$ is splittable. Thus such filtered fiber functors are in general only splittable Zariski locally on $S$ but not globally on $S$.

%Let $G$ be a reductive algebraic group over $k$ and $\chi\colon \Gm[k]\to G$ be a non-trivial cocharacter. Let $\gamma$ be the graded fiber functor associated to $\chi$ and $P\defeq P(\fil\circ\gamma)$, which is a parabolic subgroup of $G$ (see Theorem \ref{PProperties} below). For any finite-dimensional $k$-vector space let $\BV(V)$ be the associated vector bundle over $k$ and $\BV(V)\times^P G$ be the quotient of $\BV(V)\times G$ by the right $P$-action given by $(w,g)\cdot p=(wp,p^{-1}g)$. The latter is a vector bundle over $S\defeq G/P$. Thus for $W\in \Rep[G]$ we obtain a vector bundle $\BV(W)\times^PG$ over $S$ together with subbundles $\F^i(\BV(W)\times^P G)\defeq \F^i(\fil(\gamma(W)))\times^P G$ for all $i\in\BZ$ which give a filtration on $\BV(W)\times^PG$. This construction defines a filtered fiber functor $\phi$ on $\Rep[G]$ over $S$ and it can be shown that $\phi$ is Zariski-locally splittable but not splittable.

Thus in general Theorem \ref{MainTheorem} is the best possible result. Nevertheless, we have the following:
\begin{theorem} \label{FFFZLocSplittable}
 Let $\phi$ be a filtered fiber functor on $\CT$ over an affine scheme $S$. Assume that the group scheme $\UAut^\otimes_S(\forget\circ\phi)$ is pro-smooth over $S$. Then $\phi$ is splittable.
\end{theorem}

 Theorem \ref{FFFZLocSplittable} is a generalization of \cite[IV.2.2.2]{Saavedra} and \cite[IV.2.2.5 3)]{Saavedra}. By Theorem \ref{SplUTorsor}, it is a special case of the following result:
\begin{theorem}\label{UTorsorTrivial}
  Let $\phi$ be a filtered fiber functor on $\CT$ over an affine scheme $S$. Assume that the group scheme $\UAut^\otimes_S(\forget\circ\phi)$ is pro-smooth over $S$. Then for any $\alpha\geq 1$, every $U_\alpha(\phi)$-torsor for the fpqc-topology is trivial.
\end{theorem}

 For a discussion of the condition that $\UAut^\otimes_S(\forget\circ\phi)$ be pro-smooth see Subsection \ref{GerbesSection}. From the results there it follows in particular this is really a condition on $\CT$ and not on $\phi$ and that this condition is always satisfied if $k$ has characteristic zero, since then any group scheme of finite type over a field is smooth.

The rest of this subsection is devoted to proving Theorem \ref{UTorsorTrivial}. In case $\CT$ has a tensor generator, the proof of Theorem \ref{UTorsorTrivial} requires significantly less work and will be given as Theorem \ref{UTorsorTrivial2} below.

\begin{lemma} \label{FFFBaseChange2}
  Assume that $\CT$ possesses a tensor generator. Let $k'$ be an overfield of $k$ and $\phi\colon \CT\to \FilLF(k')$ a filtered fiber functor which is split by a cocharacter $\chi\colon \Gm[k']\to P(\phi)$. Set $G\defeq \UAut^\otimes_{k'}(\forget\circ\phi)$, and let $\gamma\colon \Rep[G]\to \GradQCoh(k')$ be the unique graded fiber functor such that $\chi(\gamma)=\chi$ and $\forget\circ\gamma=\forget\colon \Rep[G]\to \QCoh(k')$, whose existence is given by Theorem \ref{GFFClassification}. Then $U_\alpha(\phi)=U_\alpha(\fil\circ\gamma)$ for all $\alpha\geq 0$.
\end{lemma}
\begin{proof}
 Let $X$ be a tensor generator of $\CT$. Let $V\defeq \forget(\phi(X))$ and denote the $i$-th step of the filtration on $V$ given by $\phi(X)$ by $\F^i(X)$. The vector space $V$ is equipped with a faithful action of $G$ and $P(\phi)$ is the subgroup scheme of $G$ whose points are the points of $G$ stabilizing the filtration on $V$. Similarly, for each $\alpha\geq 1$, the group scheme $U_\alpha(\phi)$ is the subgroup scheme of $P(\phi)$ whose points are those which act as the identity on $\F^i(X)/\F^{i+\alpha}(X)$ for all $i\in \BZ$. 

If we consider $V$ as an object of $\Rep[G]$, the filtration on $V$ given by $\fil(\gamma(V))$ coincides with the one given by $\phi(V)$. Since the action of $G$ on $V$ is faithful, similarly to the above the group scheme $P(\fil\circ\gamma)$ (resp. $U_\alpha(\fil\circ\gamma)$ for any $\alpha\geq 0$) is the subgroup scheme of $G$ (resp. of $P(\fil\circ\gamma)$) whose points are those which stabilize the filtration on $V$ (resp. which act as the identity on $\F^i(X)/\F^{i+\alpha}(X)$ for all $i\in \BZ$). Thus $U_\alpha(\phi)=U_\alpha(\fil\circ\gamma)$ for all $\alpha\geq 0$.
\end{proof}
\begin{definition}
  For a filtered fiber functor $\phi$ on $\CT$ and $\alpha\geq 0$ we denote $U_\alpha(\phi)/U_{\alpha+1}(\phi)$ by $\gr_\alpha U(\phi)$.
\end{definition}

The next lemma is a generalization of \cite[IV.2.1.4.1]{Saavedra}.
\begin{lemma} \label{UASmooth}
  Let $\phi\colon \CT\to\FilLF(S)$ be a filtered fiber functor and $\alpha\geq 0$.
  \begin{itemize}
  \item[(i)] The sheaf $U_\alpha(\phi)$ is representable by a group scheme which is affine and flat over $S$.
  \item [(ii)]  If $\CT$ has a tensor generator, the sheaf $\gr_\alpha U(\phi)$ is representable by a group scheme which is flat over $S$.
  \item[(iii)] If $\UAut^\otimes_S(\phi)$ is smooth over $S$, then so are $U_\alpha(\phi)$ and $\gr_\alpha U(\phi)$.
  \end{itemize}

\end{lemma}
\begin{proof}
  The category $\CT$ is the filtered colimit of the set $I_\CT$ of Tannakian subcategories possessing a tensor generator from Subsection \ref{TensorGenSection}. Thus $U_\alpha(\phi)$ is in a natural way the filtered limit of the sheaves $U_\alpha(\phi|_{\CT'})$ for $\CT'\in I_\CT$. Hence by \cite[Proposition 8.2.3]{EGA4III} and \cite[Proposition 8.3.8]{EGA4III} it suffices to prove $(i)$ for the functors $\phi|_{\CT'}$, so that we may assume that $\CT$ has a tensor generator.  

Since the claims are local for the fpqc topology, using Theorem \ref{MainTheorem} and Lemma \ref{GFFPPrinciple} we may assume that that $\phi$ is splittable and $S$ is the spectrum of an algebraically closed field. Then using Lemma \ref{FFFBaseChange2} we may assume that $\CT=\Rep[G]$ for $G\defeq \UAut^\otimes_S(\phi)$ and that $\forget\circ\phi=\forget$. Then the representability of $U_\alpha(\phi)$ by an affine scheme is \cite[IV.2.1.4.1 1)]{Saavedra}, and this implies the representability of $\gr_\alpha U(\phi)$. Since we are over a field, both schemes are flat over $S$. 

If $\UAut^\otimes_S(\phi)$ is smooth over $S$, by \cite[IV.2.1.4.1 3)]{Saavedra} so is $U_\alpha(\phi)$. This implies that $\gr_\alpha U(\phi)$ is also smooth and we are done.
\end{proof}

\begin{lemma} \label{UACommutators}
  Let $\phi$ be a filtered fiber functor on $\CT$. For any $\alpha,\beta\geq 0$, the commutator of $U_\alpha(\phi)$ and $U_\beta(\phi)$ is contained in $U_{\alpha+\beta}(\phi)$.
\end{lemma}
\begin{proof}
  This follows by a direct verification using the definition of the $U_\alpha(\phi)$.
\end{proof}
 Since $P(\phi)=U_0(\phi)$, Lemma \ref{UACommutators} directly implies:
\begin{lemma} \label{UAProperties}
  Let $\phi$ be a filtered fiber functor on $\CT$.
  \begin{enumerate}
  \item [(i)] The group scheme $U_\alpha(\phi)$ is normal in $P(\phi)$ for all $\alpha\geq 0$.
  \item [(ii)] The sheaf of groups $\gr_\alpha U(\phi)$ is abelian for all $\alpha\geq 1$.
  \end{enumerate}
\end{lemma}

\begin{lemma} \label{SomeIso}
  Let $\phi\colon \CT\to\FilLF(S)$ be a filtered fiber functor. The morphism $P(\phi)\to L(\phi)$ induced by $\gr$ induces an isomorphism $P(\phi)/U(\phi)\cong L(\phi)$.
\end{lemma}
\begin{proof}
  By definition, the group scheme $U(\phi)$ is the kernel of $P(\phi)\to L(\phi)$. Thus it remains to prove that $P(\phi)\to L(\phi)$ is an epimorphism. For this we may work fpqc-locally on $S$, so that by Theorem \ref{MainTheorem} we may assume that $\phi=\fil\circ\gamma$ for a graded fiber functor $\gamma\colon \CT\to\GradQCoh(S)$. Then there is a canonical isomorphism $\gr\circ\phi\cong \gamma$ and the morphism $\UAut^\otimes_S(\gamma)\to P(\phi)$ induced by $\fil$ splits $P(\phi)\to L(\phi)$.
\end{proof}
\begin{construction}
  Let $\phi\colon \CT\to\FilLF(S)$ be a filtered fiber functor and $\alpha\geq 1$. We define an action of $\Gm[S]$ on $\gr_\alpha U(\phi)$ as follows:
  
By Lemma \ref{UAProperties}, the group scheme $P(\phi)$ acts on $U_\alpha(\phi)$ by conjugation. Lemma \ref{UACommutators} implies that this action induces an action of $P(\phi)/U(\phi)$ on $\gr_\alpha U(\phi)$.

Via the natural isomorphism $P(\phi)/U(\phi)\cong L(\phi)$ from Lemma \ref{SomeIso} we consider $\chi(\gr\circ\phi)\colon \Gm[S]\to L(\phi)$ as a cocharacter of $P(\phi)/U(\phi)$. By restricting the above action of $P(\phi)/U(\phi)$ on $\gr_\alpha U(\phi)$ along $\chi(\gr\circ\phi)$ we obtain an action of $\Gm[S]$ on $\gr_\alpha U(\phi)$ which we denote by $m$.
\end{construction}

For all $\alpha\geq 1$ we denote by $\mu_{\alpha,S}\subset \Gm[S]$ the kernel of the homomorphism $\Gm[S]\to \Gm[S]$, $t\mapsto t^\alpha$.

\begin{proposition} \label{UAVectorBundle}
  Let $\phi\colon \CT\to\FilLF(S)$ be a filtered fiber functor for which the group scheme $\UAut^\otimes_S(\forget\circ\phi)$ is smooth over $S$. For any $\alpha\geq 1$, under the above action of $\Gm[S]$ on $\gr_\alpha U(\phi)$ the subgroup $\mu_{\alpha,S}$ acts trivially on $\gr_\alpha U(\phi)$. 

Thus there is an action $\tilde m$ of $\Gm[S]$ on $\gr_\alpha U(\phi)$ making the following diagram commutative:
\begin{equation*}
  \xymatrix{
    \Gm[S] \times_S \gr_\alpha U(\phi) \ar[r]^-m \ar[d]_{(t\mapsto t^\alpha)\times_S \Id}  & \gr_\alpha U(\phi) \ar@{=}[d] \\
    \Gm[S] \times_S \gr_\alpha U(\phi) \ar[r]^-{\tilde m} & \gr_\alpha U(\phi) 
}
\end{equation*}
This action $\tilde m$ makes $\gr_\alpha U(\phi)$ into a vector bundle over $S$.
\end{proposition}

\begin{proof}
  By Lemma \ref{BlaLemma} the Tannakian category $\CT$ has a tensor generator. For both claims we may work locally for the fpqc topology. Thus using Theorem \ref{MainTheorem} and Lemma \ref{GFFPPrinciple} we may assume that $S$ is the spectrum of an algebraically closed overfield $k'$ of $k$ and that $\phi$ is splittable.

 Let $X$ be a tensor generator of $\CT$ and $V\defeq \forget(\phi(X))$. We denote by $\F^iV$ the $i$-th step of the filtration on $V$ defined by $\phi(X)$. We define a map 
\begin{equation*}
  h\colon U_\alpha(\phi)(k')\to \oplus_{i\in\BZ}\Hom_{k'}(\F^i V/\F^{i+1}V,\F^{i+\alpha}V/\F^{i+\alpha+1}V)\rdefeq W
\end{equation*}
as follows: Let $u\in U_\alpha(\phi)(k')$ and $i\in \BZ$. If we denote the $k'$-linear map $V\to V$ induced by $u$ again by $u$, then by the definition of $U_\alpha(\phi)$, the $k'$-linear map $u-\Id_V$ maps $\F^iV$ into $\F^{i+\alpha}V$ and $\F^{i+1}V$ into $\F^{i+1+\alpha}V$. Thus $u-\Id_V$ induces a $k'$-linear map $\F^iV/\F^{i+1}V\to \F^{i+\alpha}V/\F^{i+\alpha+1}V$ and we let $h(u)$ be the direct sum of these maps over all $i\in \BZ$. Since $V$ is finite-dimensional over $k'$, so is $W$. The following lemma shows that $h$ gives an isomorphism from $\gr_\alpha U(\phi)(k')$ onto an additive subgroup of $W$:
\begin{lemma}
  \begin{itemize}
  \item [(i)] The map $h$ is a group homomorphism from $U_\alpha(\phi)(k')$ to the additive group of $W$.
  \item [(ii)] The kernel of $h$ is $U_{\alpha+1}(\phi)(k')$.
  \end{itemize}
\end{lemma}
\begin{proof}
  (i) Let $u,u'\in U_\alpha(\phi)(k')$ and $i\in\BZ$. Then for all $x\in \F^iV$ 
  \begin{equation*}
    (u(u'(x))-x)-((u(x)-x)+(u'(x)-x))=u(u'(x)-x)-(u'(x)-x)\in \F^{i+2\alpha}V
  \end{equation*}
since $u'(x)-x\in \F^{i+\alpha}V$. Since $\alpha\geq 1$ we have $\F^{i+2\alpha}V\subset \F^{i+\alpha+1}V$ and thus the preceding equation implies $h(uu')=h(u)+h(u')$. 

(ii) An element $u\in U_\alpha(\phi)(k')$ is in the kernel of $h$ if and only if $u-\Id_V$ maps $\F^iV$ into $\F^{i+\alpha+1}V$ for all $i\in \BZ$. Since $X$ is a tensor generator, this is equivalent to $u\in U_{\alpha+1}(\phi)(k')$.
\end{proof}
We let $\Gm(k')$ act on $\gr_\alpha U(\phi)(k')$ via $m$ and on $W$ via $t\mapsto t^\alpha$. Then $h$ is $\Gm(k')$-equivariant by a direct verification. This implies that $\mu_\alpha(k')$ acts trivially on $\gr_\alpha U(\phi)(k')$. Since $\gr_\alpha U(\phi)$ is smooth by Lemma \ref{UASmooth}, this proves the first part of the claim. 

Furthermore, since $\Gm(k')\to \Gm(k'),\; t \mapsto t^\alpha$ is surjective, the $\Gm(k')$-equivariance of $h$ implies that the image of $h$ is a $k'$-subspace of $W$. If we identify $\gr_\alpha U(\phi)(k')$ with this subspace via $h$, the resulting action of $(k')^*$ on $\gr_\alpha U(\phi)(k')$ is the one given by $\tilde m$. Using the smoothness of $\gr_\alpha U(\phi)$ this implies the second part of the claim.
\end{proof}
Proposition \ref{UAVectorBundle} is enough to prove the following result, which is Theorem \ref{UTorsorTrivial} in the case that $\CT$ has a tensor generator:

\begin{theorem} \label{UTorsorTrivial2}
   Assume that $\CT$ has a tensor generator and let $\phi$ be a filtered fiber functor on $\CT$ over an affine scheme $S$ for which the group scheme $\UAut^\otimes_S(\forget\circ\phi)$ is smooth over $S$. Then for any $\alpha\geq 1$, every $U_\alpha(\phi)$-torsor for the fpqc-topology is trivial.
\end{theorem}
\begin{proof}
  We need to prove that the set $\Hfl(S,U_\alpha(\phi))$ has exactly one element for all $\alpha\geq 1$. Since $S$ is affine and since by \cite[Proposition III.3.7]{MilneEtaleCohomology} for a quasi-coherent sheaf the flat and Zariski cohomology groups are isomorphic, by Proposition \ref{UAVectorBundle} the group $\Hfl(S,\gr_\alpha U(\phi))$ is zero for all $\alpha\geq 1$. Thus for all $\alpha\geq 1$, the natural morphism $\Hfl(S,U_{\alpha+1}(\phi))\to \Hfl(S,U_\alpha(\phi))$ is surjective. Since $\CT$ has a tensor generator, for $\alpha$ large enough the group $U_\alpha(\phi)$ is trivial. This implies the claim.
\end{proof}

\begin{lemma} \label{UAVectorBundleMorphism}
  Let $\phi\colon \CT\to\FilLF(S)$ be a filtered fiber functor for which $\UAut^\otimes_S(\forget\circ\phi)$ is smooth, let $\CT'\subset \CT$ be a Tannakian subcategory and $\alpha\geq 1$. If $\gr_\alpha U(\phi)$ and $\gr_\alpha U(\phi|_{\CT'})$ are endowed with the vector bundle structure given by Proposition \ref{UAVectorBundle}, the morphism of fpqc sheaves
  \begin{equation*}
     \gr_\alpha U(\phi) \to \gr_\alpha U(\phi|_{\CT'})
  \end{equation*}
 induced by restriction is an $\CO_S$-linear epimorphism.
\end{lemma}
\begin{proof}
  Since $G'\defeq \UAut^\otimes_S(\phi|_{\CT'})$ is a quotient of $G\defeq \UAut^\otimes_S(\phi)$ it is smooth and thus the statement of the lemma makes sense. That the morphism in question is $\CO_S$-linear follows directly from the definition of the vector bundle structures.

Fix $\alpha\geq 1$. To prove that the morphism is an epimorphism, we may work fpqc-locally on $S$. Thus by Theorem \ref{MainTheorem} and Lemma \ref{GFFPPrinciple} we may assume that $\phi$ is split by a cocharacter $\chi\colon \Gm[S]\to \UAut^\otimes_S(\phi)$ and that $S$ is the spectrum of an algebraically closed field $k'$. Then $\phi|_{\CT'}$ is split by the cocharacter $\chi'\colon \Gm[S]\ltoover{\chi} P(\phi) \to P(\phi|_{\CT'})$, where the morphism $P(\phi) \to P(\phi|_{\CT'})$ is given by restriction.

By Lemma \ref{BlaLemma} the Tannakian category $\CT$ possesses a tensor generator. Using Lemma \ref{FFFBaseChange2} we may assume that $\CT=\Rep[G]$, that $\CT'=\Rep[G']$ and that $\forget\circ\phi=\forget$. We let $G$ act on its Lie algebra $\Lie(G)$ (resp. on $\Lie(G')$) via the adjoint action (resp. via $G\to G'$ composed with the adjoint action of $G'$). Thus $\phi$ endows $\Lie(G)$ and $\Lie(G')$ with a filtration, whose $i$-th step we denote by $\F^i\Lie(G)$ (resp. $\F^i\Lie(G')$). Then $\Lie(U_\alpha(\phi))=\F^\alpha(\Lie(G))$ and $\Lie(U_\alpha(\phi|_{\CT'}))=\F^\alpha(\Lie(G'))$ by \cite[IV.2.1.4.1]{Saavedra}. 

Since the morphism $G\to G'$ is faithfully flat, the induced morphism $\Lie(G)\to \Lie(G')$ is surjective. Hence the exactness of $\phi$ implies the the induced morphism $\Lie(U_\alpha(\phi)) \to \Lie(U_\alpha(\phi|_{\CT'}))$ is surjective. Since by Lemma \ref{UASmooth} the $U_\alpha$ are smooth, this implies that the image of $U_\alpha(\phi)$ in $U_\alpha(\phi|_{\CT'})$ is of finite index. Thus also the image of the $k'$-vector space $\gr_\alpha U(\phi)$ in the $k'$-vector space $\gr_\alpha U(\phi|_{\CT'})$ is of finite index. Since $k'$ is infinite this implies that this image is all of $\gr_\alpha U(\phi|_{\CT'})$, which finishes the proof.
\end{proof}

\begin{proof}[Proof of Theorem \ref{UTorsorTrivial}]
Let $\alpha\geq 1$ and $X$ be a right $U_\alpha(\phi)$-torsor. For all $j\leq i$ in $I_\CT$ there are natural morphisms $U_\alpha(\phi)\to U_\alpha(\phi|_i)$ and $U_\alpha(\phi|_i)\to U_\alpha(\phi|_j)$ induced by restriction. For $i\in I_\CT$, let $X_i\defeq X\times^{U_\alpha(\phi)} U_\alpha(\phi|_i)$ which is a right $U_\alpha(\phi|_i)$-torsor. The fact that $\CT=\varinjlim_{i\in I_\CT}i$ implies $U_\alpha(\phi)\cong \varprojlim_{i\in I_\CT}U_\alpha(\phi|_i)$. Hence $X\cong \varprojlim_{i\in I_\CT}X_i$. By the discussion at the end of Section \ref{Preliminaries}, the condition that $\UAut^\otimes_S(\forget\circ\phi)$ is pro-smooth over $S$ is equivalent to saying that for each $i\in I_\CT$ the group scheme $\UAut^\otimes_S(\forget\circ\phi|_i)$ is smooth over $S$. Hence by Theorem \ref{UTorsorTrivial2} for each $i$ the set $X_i(S)$ is not empty and thus a torsor under the group $G_i^\alpha \defeq U_\alpha(\phi|_i)(S)$. Hence the $X_i(S)$ form a torsor under the inverse system of groups $(G^\alpha_i)_{i\in I_\CT}$. We want to show that $X(S)=\varprojlim_{i\in I_\CT}X_i(S)$ is not empty. This is equivalent to saying that the $(G_i^\alpha)_{i\in I_\CT}$-torsor $(X_i(S))_{i\in I_\CT}$ is isomorphic to the trivial one. By Lemma \ref{Lim1Torsor} it is thus sufficient to prove that $\varprojlim_{i\in I_\CT}^1G_i^\alpha$ is the trivial pointed set for all $\alpha\geq 1$. First we prove:
\begin{lemma} \label{SomeLemma5}
  For all $\alpha\geq 1$, the pointed set $\varprojlim^1_{i\in I_\CT}G_i^\alpha/G_i^{\alpha+1}$ is trivial.
\end{lemma}
\begin{proof}
  Let $R\defeq \Gamma(S,\CO_S)$ and $M_i\defeq G_i^\alpha/G_i^{\alpha+1}$ for all $i\in I_\CT$. The case of Theorem \ref{UTorsorTrivial} already proved implies $M_i=\gr_\alpha U(\phi|_i)(S)$. Thus by Proposition \ref{UAVectorBundle} the $M_i$ are finitely generated projective $R$-modules. By Lemma \ref{UAVectorBundleMorphism} the transition maps $M_i\to M_j$ are $R$-linear for all $j\leq i$ in $I_\CT$. Thus the $M_i$ form an inverse system of $R$-modules and it suffices to verify the conditions of Theorem \ref{InvLimitProjModules}. We already know that condition (ii) is satisfied.

(i) This is Lemma \ref{ITLattice}.

(iii) This is part of Lemma \ref{UAVectorBundleMorphism}.

(iv) Let $i,j\in I_\CT$. By Lemma \ref{FFCartesian} the square
\begin{equation*}
  \xymatrix{ 
    \UAut^\otimes_S(\forget\circ\phi|_i) \ar[d] & \UAut^\otimes_S(\forget\circ\phi|_{i\vee j}) \ar[l]\ar[d] \\
    \UAut^\otimes_S(\forget\circ\phi|_{i\wedge j}) & \ar[l] \UAut^\otimes_S(\forget\circ\phi|_j) 
}
\end{equation*}
is Cartesian. This implies implies that the square
\begin{equation*}
   \xymatrix{ 
    U_\alpha(\phi|_i)(S) \ar[d] & U_\alpha(\phi|_{i\vee j})(S) \ar[l]\ar[d] \\
    U_\alpha(\phi|_{i\wedge j})(S) & \ar[l] U_\alpha(\phi|_j)(S) 
}
\end{equation*}
is Cartesian for all $\alpha\geq 0$. This in turn implies that condition (iv) is satisfied.
\end{proof}
Now we can prove:
\begin{lemma}
  For all $\alpha\geq 1$, the pointed set $\varprojlim_{i\in I_\CT}^1G_i^\alpha$ is trivial. 
\end{lemma}
\begin{proof}
 Let $\alpha\geq 1$ and $x\in \varprojlim_{i\in I_\CT}^1G_i^\alpha$. Lemmas \ref{SomeLemma5} and \ref{Lim1ExactSeq} imply that for all $\alpha'\geq 1$, the natural map $\varprojlim_{i\in I_\CT}^1G_i^{\alpha'+1}\to \varprojlim_{i\in I_\CT}^1G_i^{\alpha'}$ is surjective. Thus the induced map $\varprojlim_{\alpha'\geq \alpha}\varprojlim_{i\in I_\CT}^1G_i^{\alpha'}\to \varprojlim_{i\in I_\CT}^1G_i^\alpha$ is surjective. Hence there exists $(x_{\alpha'})_{\alpha'\geq \alpha}\in \varprojlim_{\alpha'\geq \alpha}\varprojlim_{i\in I_\CT}^1G_i^{\alpha'}$ such that $x_\alpha=x$. For each $\alpha'\geq\alpha$ we pick a representative $(\leftexp{\alpha'}{g}_j^i)_{j\leq i\in I_\CT}$ of $x_{\alpha'}$. For any $\alpha''\geq \alpha'\geq \alpha$ and $j\leq i$ in $I_\CT$ we denote the transition map $G_i^{\alpha''}\to G_i^{\alpha'}$ by $r^{\alpha''}_{\alpha'}$ and the transition map $G_i^{\alpha'}\to G_j^{\alpha'}$ by $\pi_j^i$. The fact that $(x_{\alpha'})_{\alpha'\geq \alpha}\in \varprojlim_{\alpha'\geq \alpha}\varprojlim_{i\in I_\CT}^1G_i^{\alpha'}$ implies that for each $\alpha'\geq \alpha$, there exists $(\leftexp{\alpha'}{g}_i)_{i\in I_\CT}\in \prod_{i\in I_\CT}G_i^{\alpha'}$ such that 
 \begin{equation}\label{SomeEq233}
   \leftexp{\alpha'}{g}_j^i=\pi_j^i(\leftexp{\alpha'}{g}_i) r^{\alpha'+1}_{\alpha'}(\leftexp{\alpha'+1}{g}_j^i) (\leftexp{\alpha'}{g}_j)^{-1}
 \end{equation}
 for all $j\leq i\in I_\CT$.

Since each $i\in I_\CT$ has a tensor generator, for each $i\in I_\CT$, there exists $\alpha_i\geq 1$ such that $G_i^{\alpha'}=1$ for $\alpha'> \alpha_i$. For $i\in I_\CT$ let $g_i\defeq \leftexp{\alpha}{g}_i\cdot r^{\alpha+1}_\alpha(\leftexp{\alpha+1}{g}_i) \cdots r^{\alpha_i}_\alpha(\leftexp{\alpha_i}{g}_i)$. Then for any $i\leq j$ in $I_\CT$ repeated application of \eqref{SomeEq233} shows that $\leftexp{\alpha}{g}_j^i=\pi_j^i(g_i)g_j^{-1}$. Thus $x$, which is the class of $(\leftexp{\alpha}{g}_j^i)_{i\leq j\in I_\CT}$ in $\varprojlim_{i\in I_\CT}^1 G_i^\alpha$, is equal to the distinguished element of $\varprojlim_{i\in I_\CT}^1 G_i^\alpha$. This proves the claim.
\end{proof}
\end{proof}

\subsection{The Stack of Filtered Fiber Functors} \label{StFFF}
Let $\CT$ be a Tannakian category over $k$. In this subsection, we study the fibered category $\UHom^\otimes(\CT,\FilLF)$.
\begin{lemma} \label{FFFStack}
  The fibered category $\UHom^\otimes(\CT,\FilLF)$ is a stack for the fpqc topology.
\end{lemma}
\begin{proof}
  For any fpqc covering $S'\to S$ of $k$-schemes we denote by $\Desc(S'/S)$ the category of objects of $\FilLF(S')$ equipped with a descent datum relative to $S$. Given $\CM,\CM'\in \FilLF(S')$, descent data on $\CM$ and $\CM'$ relative to $S'$ induce a descent datum on $\CM\otimes\CM'$ relative to $S'$. This makes $\Desc(S'/S)$ into a symmetric monoidal category, and the natural functor $b\colon \FilLF(S)\to \Desc(S'/S)$, which is an equivalence since $\FilLF$ is a stack for the fpqc topology, is a tensor functor. Let $b^{-1}\colon \Desc(S'/S)\to\FilLF(S)$ be a tensor functor which is an inverse of $b$. To give a tensor functor $\CT\to \FilLF(S')$ equipped with a descent datum relative to $S$ is the same as giving a tensor functor $\CT\to \Desc(S'/S)$. Thus such a functor can be descended to a tensor functor $\CT\to\FilLF(S)$ by composing it with $b^{-1}$. Similarly a tensor morphism between two tensor functors $\CT\to\FilLF(S')$ which is equipped with a descent datum can be descended by applying $b^{-1}$. It follows from Lemma \ref{FilShortExact} that for a tensor functor $\phi\colon \CT\to\FilLF(S)$, the property of being exact is local on $S$ for the fpqc topology. Altogether this implies the claim.
\end{proof}

\begin{definition}
  Let 
  \begin{equation*}
    \xymatrix{
      \UHom^\otimes(\CT,\GradQCoh)\ar@<1ex>[r]^\fil & \UHom^\otimes(\CT,\FilLF) \ar@<1ex>[l]^\gr \\
      }
  \end{equation*}
 be the morphisms induced by composition with the morphisms $\fil\colon \GradQCoh\to\FilLF$ (resp. $\gr\colon \FilLF\to\GradQCoh$).
\end{definition}

\begin{proposition} \label{FilRepresentable}
  Assume that $\CT$ has a tensor generator. Then the morphism of stacks 
  \begin{equation*}
    \fil\colon \UHom^\otimes(\CT,\GradQCoh)\to \UHom^\otimes(\CT,\FilLF)
  \end{equation*}
is representable by schemes, affine, faithfully flat and of finite presentation.
\end{proposition}
\begin{proof}
  The claim means the following: For any scheme $S$ over $k$ and any filtered fiber functor $\phi\colon S\to \UHom^\otimes(\CT,\FilLF)$ the stack 
  \begin{equation*}
    X\defeq \UHom^\otimes(\CT,\GradQCoh)\times_{\fil,\UHom^\otimes(\CT,\FilLF)} S
  \end{equation*}
is representable by a scheme which is affine, faithfully flat and of finite presentation over $S$. To prove this, we may work fpqc-locally on $S$. Thus, by Theorem  \ref{MainTheorem} we may assume that $\phi$ is splittable. We fix a graded fiber functor $\gamma_0$ splitting $\phi$.

For any scheme $S'$ over $S$, the objects of $X(S')$ are pairs $(\gamma,\lambda)$ consisting of a graded fiber functor $\gamma\in \UHom^\otimes(\CT,\GradQCoh)(S')$ and an isomorphism $\lambda\colon \fil\circ\gamma\isoto \phi_{S'}$, and a morphism $(\gamma,\lambda)\to (\gamma',\lambda')$ in $X(S')$ is a morphism $\mu\colon \gamma\to \gamma'$ such that $\lambda=\lambda'\circ \fil(\mu)$. Since $\fil$ is faithful, for any such $(\gamma,\lambda)$ and $(\gamma',\lambda')$ there can thus be at most one morphism $(\gamma,\lambda)\to (\gamma',\lambda')$. Since for any $(\gamma,\lambda)\in X(S')$ the tensor morphism $\gr(\lambda)$ gives an isomorphism $\gamma\cong \gr\circ\phi_{S'}\cong (\gamma_0)_{S'}$, the category $X(S')$ is equivalent to its full subcategory whose objects are of the form $((\gamma_0)_{S'},\lambda)$. Since $\fil\circ (\gamma_0)_{S'}=\phi_{S'}$ it follows that $P(\phi)(S')\to X(S'),\; \lambda\mapsto ((\gamma_0)_{S'},\lambda)$ is an equivalence. Since this equivalence is functorial in $S$ it gives an equivalence $P(\phi)\to X$ of fibered categories. Thus $X$ has the required properties by Theorem \ref{FFFIsomRepr} and Lemma \ref{UASmooth}.
\end{proof}
\begin{theorem}
  Assume that $\CT$ has a tensor generator. Then $\UHom^\otimes(\CT,\FilLF)$ is an algebraic stack which is locally of finite type over $k$.
\end{theorem}
\begin{proof}
  Since $\CT$ is rigid, by \cite[I.5.2.3]{Saavedra} any morphism in $\UHom^\otimes(\CT,\FilLF)$ is an isomorphism. Thus $\UHom^\otimes(\CT,\FilLF)$ is fibered in groupoids. By \cite[Th\'eor\`eme 10.1]{ChampsAlgebriques}, to prove that $\UHom^\otimes(\CT,\FilLF)$ is an algebraic stack it suffices to prove that the diagonal morphism $\UHom^\otimes(\CT,\FilLF)\to \UHom^\otimes(\CT,\FilLF)\times \UHom^\otimes(\CT,\FilLF)$ is representable, separated and quasi-compact and that there exists an algebraic space $Y$ and a morphism $Y\to \UHom^\otimes(\CT,\FilLF)$ which is representable, faithfully flat and locally of finite presentation.

The condition on the diagonal morphism is equivalent to saying that for any two filtered fiber functors $\phi_1,\phi_2$ over a scheme $S$ the functor $\UIsom^\otimes_S(\phi_1,\phi_2)$ is representable by an algebraic space which is separated and quasi-compact over $S$. This condition is satisfied by Theorem \ref{FFFIsomRepr}.

Since $\UHom^\otimes(\CT,\GradQCoh)$ is an algebraic stack locally of finite type over $k$ by Theorem \ref{GFFStackAlgebraic}, there exists an algebraic space $Y$ which is locally of finite type over $k$ and a morphism $Y\to \UHom^\otimes(\CT,\GradQCoh)$ over $k$ which is representable, faithfully flat and locally of finite presentation. Composing such a morphism with the morphism $\UHom^\otimes(\CT,\GradQCoh)\to\UHom^\otimes(\CT,\FilLF)$ from Proposition \ref{FilRepresentable} yields a morphism $Y\to \UHom^\otimes(\CT,\FilLF)$ which is representable, faithfully flat and locally of finite presentation.

Since $Y$ is locally of finite type over $k$ and the morphism $Y\to\UHom^\otimes(\CT,\FilLF)$ is faithfully flat and locally of finite presentation, the algebraic stack $\UHom^\otimes(\CT,\FilLF)$ is locally of finite presentation over $k$ (c.f. \cite[Lemme 17.7.5]{EGA4IV}).
\end{proof}

Now we turn to the type of a filtered fiber functor.

\begin{definition}
  Assume that $\CT$ has a tensor generator. 
  \begin{itemize}
  \item[(i)]   We denote the composition 
  \begin{equation*}
    \UHom^\otimes(\CT,\FilLF)\ltoover{\gr}\UHom^\otimes(\CT,\GradQCoh)\ltoover{t}\CC_\CT
  \end{equation*}
again by $t$.
\item[(ii)] For a filtered fiber functor $\phi$ on $\CT$ over $S$ its image $t(\phi)$ in $\CC_\CT(S)$, which is the type of the graded fiber functor $\gr\circ\phi$, is called the \emph{type of $\phi$}.
  \end{itemize}
\end{definition}

\begin{theorem}
  Assume that $\CT$ has a tensor generator. Then $t\colon \UHom^\otimes(\CT,\FilLF)\to \CC_\CT$ makes $\CC_\CT$ into the coarse fpqc sheaf associated to $\UHom^\otimes(\CT,\FilLF)$. In particular $\CC_\CT$ is the coarse moduli space of $\UHom^\otimes(\CT,\FilLF)$ and $\UHom^\otimes(\CT,\FilLF)$ is a gerbe over its coarse moduli space.
\end{theorem}
\begin{proof}
  Let $X$ be the coarse fpqc sheaf associated to $\UHom^\otimes(\CT,\FilLF)$. The morphism $t$ induces a morphism $\bar t\colon X\to \CC_\CT$ of sheaves. It is an epimorphism since $\gr\circ\fil$ is canonically isomorphic to the identity on $\UHom^\otimes(\CT,\GradQCoh)$. Theorem \ref{MainTheorem} implies that two filtered fiber functors on $\CT$ over a given scheme are fpqc-locally isomorphic if and only if the associated graded fiber functors are fpqc-locally isomorphic. This implies that $\bar t$ is a monomorphism. The last statement follows from the fact that $\CC_\CT$ is representable by Theorem \ref{ClRepr}.
\end{proof}

\subsection{Consequences}\label{FFFConsequences}
Let $\CT$ be a Tannakian category over $k$. Theorem \ref{MainTheorem} and the above results about the stack $\UHom^\otimes(\CT,\FilLF)$ allow us to prove the following facts about filtered fiber functors:
\begin{theorem} \label{GrFaithfullyFlat}
  \begin{itemize}
  \item [(i)]  For any two filtered fiber functors $\phi_1,\phi_2$ on $\CT$ over a $k$-scheme $S$ the morphism
  \begin{equation*}
    \gr\colon \UIsom^\otimes_S(\phi_1,\phi_2)\to \UIsom^\otimes_S(\gr\circ\phi_1,\gr\circ\phi_2)
  \end{equation*}
  is faithfully flat. In particular $\UIsom_S^\otimes(\phi_1,\phi_2)$ is flat over $S$.
\item [(ii)] If $\phi_1,\phi_2$ are filtered fiber functors on $\CT$ over $k$-schemes $S_1,S_2$, then $\UIsom^\otimes_k(\phi_1,\phi_2)$ faithfully flat over $\UIsom^\otimes_k(\gr\circ\phi_1,\gr\circ\phi_2)$ and flat over $S_1\times S_2$.
  \end{itemize}
\end{theorem}
\begin{proof}
  (i) We may work fpqc-locally on $S$. Thus by Theorem \ref{MainTheorem} we may assume that $\phi_i=\gr\circ\gamma_i$ for graded fiber functors $\gamma_i$. Then $\gr\circ\phi_i\cong\gamma_i$ and the functor $\fil$ induces a splitting of the morphism in question, which implies faithful flatness. Then the claim about flatness follows from Theorem \ref{GFFIsomRepr}. (ii) is a special case of (i).
\end{proof}
\begin{theorem} \label{FFFIsomTorsor}
  Let $\phi_1,\phi_2$ be two filtered fiber functors on $\CT$ over a $k$-scheme $S$. The right action of $U(\phi_1)$ on $\UIsom^\otimes_S(\phi_1,\phi_2)$ by composition of functors makes $\UIsom^\otimes_S(\phi_1,\phi_2)$ into a right $U(\phi_1)$-torsor over $\UIsom^\otimes_S(\gr\circ\phi_1,\gr\circ\phi_2)$.
\end{theorem}
\begin{proof}
  Since $\UIsom^\otimes_S(\phi_1,\phi_2)$ is faithfully flat over $\UIsom^\otimes_S(\gr\circ\phi_1,\gr\circ\phi_2)$ by Theorem \ref{GrFaithfullyFlat} it suffices to show that for each tensor isomorphism $\lambda\colon \gr\circ(\phi_1)_{S'}\to \gr\circ(\phi_2)_{S'}$ over some scheme $S'$ over $S$, the group $U(\phi)(S')$ acts simply transitively on the set of tensor isomorphisms $\lambda'\colon (\phi_1)_{S'}\to(\phi_2)_{S'}$ for which $\gr(\lambda')=\lambda$. This follows directly from the definition of $U(\phi)$.
\end{proof}
Part (iii) of the following theorem is due to Saavedra Rivano:
\begin{theorem} \label{PProperties}
  Let $\phi$ be a filtered fiber functor on $\CT$ over $S$.
  \begin{itemize}
  \item [(i)] The group schemes $P(\phi)$ and $U(\phi)$ are flat over $S$.
  \item [(ii)] If $\UAut^\otimes_S(\forget\circ\phi)$ is smooth over $S$, so are $P(\phi)$ and $U(\phi)$.
  \item [(iii)] Assume that $G\defeq \UAut^\otimes_S(\forget\circ\phi)$ is reductive over $S$. Then:
    \begin{itemize}
    \item $P(\phi)$ is a parabolic subgroup of $G$ with unipotent radical $U(\phi)$. If $\chi\colon \Gm[S]\to G$ splits $\phi$ and $\Lie(G)=\oplus_{n\in\BZ}\Lie(G)^n$ is the weight decomposition induced by $\chi$, then $\Lie(P)=\oplus_{n\geq 0}\Lie(G)$.
    \item If $\chi\colon \Gm[S]\to G$ splits $\phi$, then $\Cent_G(\chi)$ is a Levi subgroup of $P(\phi)$.
    \item The preceding construction gives a bijection between splittings of $\phi$ and Levi subgroups of $P(\phi)$.
    \end{itemize}
  \end{itemize}
\end{theorem}
\begin{proof}
  $(i)$ and $(ii)$ are part of Lemma \ref{UASmooth}. (iii) is \cite[IV.2.2.5]{Saavedra}.
\end{proof}

Theorems \ref{GFFIsom} and \ref{GrFaithfullyFlat} together imply:
\begin{theorem} \label{FFFIsom}
  Assume that $\CT$ has a tensor generator. Let $\phi_1,\phi_2$ be graded fiber functors on $\CT$ over connected schemes $S_1,S_2$. Then exactly one of the following is true:
  \begin{itemize}
  \item[(i)] $\UIsom^\otimes_k(\phi_1,\phi_2)$ is the empty scheme.
  \item[(ii)] The types $t(\phi_1)$ and $t(\phi_2)$ both have the same same set-theoretic image in $\CC_\CT$.
 \end{itemize}
 In the second case, the image of $t(\phi_1)$ and $t(\phi_2)$ consists of one point $x$ whose residue field $k(x)$ is a finite separable field extension of $k$. Then the morphisms $t(\phi_i)\colon S_i\to \CC_\CT$ factor through $\Spec(k(x))\subset \CC_\CT$ and the morphism $\UIsom^\otimes_k(\phi_1,\phi_2) \to S_1\times S_2$ factors through a faithfully flat morphism $\UIsom^\otimes_k(\phi_1,\phi_2)\to S_1\times_{k(x)} S_2$.
\end{theorem}

The following two theorems are proven in exactly the same way as Theorems \ref{GFFLocIso} and \ref{GFFClassification2}.
\begin{theorem} \label{FFFLocIso}
  Assume that $\CT$ has a tensor generator. Two filtered fiber functors on $\CT$ over some scheme $S$ have the same type if and only if they are fpqc-locally isomorphic.
\end{theorem}

\begin{theorem} \label{FFFClassification2}
  Assume that $\CT$ has a tensor generator. Let $\phi$ be a filtered fiber functor on $\CT$ over some scheme $S$ and $S'$ a scheme over $S$. The functor which sends a graded fiber functor $\phi'$ over $S'$ to $\UIsom^\otimes_{S'}(\phi',\phi_{S'})$ and a tensor morphism $\phi'\to\phi''$ to the induced morphism $\UIsom^\otimes_{S'}(\phi',\phi_{S'})\to \UIsom^\otimes_{S'}(\phi'',\phi_{S'})$ gives an equivalence between filtered fiber functors over $S'$ having the same type as $\phi$ and left $\UAut^\otimes_{S}(\phi)_{S'}$-torsors.
\end{theorem}
\section{Proof of the main theorem} \label{Proof}
In this section we give the proof of Theorem \ref{MainTheorem}. Let $\CT$ be a Tannakian category over $k$ and $\phi\colon \CT\to\FilLF(S)$ a filtered fiber functor. The proof will consist of a series of reduction steps. In the initial cases we will have $\CT=\Rep[G]$ for a group scheme $G$ over $k$. In this situation, we denote by $\omega_0\colon \Rep[G]\to \QCoh(\Spec(k))$ the forgetful functor. Since the fiber functors $\forget\circ\phi$ and $\omega_0$ are fpqc-locally isomorphic, after replacing $S$ by a suitable covering we may assume that $\forget\circ\phi=(\omega_0)_S$.

\subsection{Preparations}
First we give two lemmas.
\begin{lemma} \label{PhiExact}
  Let $\phi\colon \CT\to \FilLF(S)$ be a $k$-linear tensor functor. Assume that the functor $\forget\circ\phi\colon \CT\to\QCoh(S)$ is exact and that for each epimorphism $X\to X'$ in $\CT$ and each $n\in \BZ$ the induced morphism $\F^n\phi(X)\to \F^n\phi(X')$ is an epimorphism. Then $\phi$ is exact.
\end{lemma}
\begin{proof}
  Let $0\to X'\to X \to X''\to 0$ be an exact sequence in $\CT$. By Lemma \ref{FilShortExact} it suffices to show that $0\to \F^n\phi(X')\to\F^n \phi(X)\to \F^n\phi(X'')\to 0$ is exact for all $n\in\BZ$.

The fact that $\phi$ is a tensor functor implies $\phi(X^\vee)=\phi(X)^\vee$. Thus $F^n \phi(X^\vee)=(\F^{1-n}\phi(X))^\perp$ and analogously for $X'$ and $X''$. Since $X'\to X$ is a monomorphism, the dual morphism $X^\vee\to (X')^\vee$ is an epimorphism. Hence $(\F^{n}\phi(X))^\perp\to (\F^{n}\phi(X'))^\perp$ is an epimorphism by the assumption on $\phi$ which implies that $\F^n\phi(X')\to \F^n\phi(X)$ is a monomorphism.

The fact that $\phi$ is $k$-linear implies that the composite $\F^n\phi(X')\to \F^n\phi(X)\to \F^n\phi(X'')$ is zero. 

It remains to prove exactness in the middle. Using the anti-automorphism $X \mapsto X^\vee$ of $\CT$ it is enough to show that the sequence
\begin{equation}
  \label{eq:exactseq3}
  0\to(\F^{1-n}\phi(X''))^\perp\to (\F^{1-n}\phi(X))^\perp\to(\F^{1-n}\phi(X'))^\perp\to 0.
\end{equation}
is exact in the middle.

Let $f\in (\F^{1-n}\phi(X))^\perp$ such that $f|_{\forget(\phi(X'))}=0$. Since $\forget\circ\phi$ is exact, the morphism $f$ factors through a unique $\CO_S$-linear morphism $\tilde f\colon \forget(\phi(X'')) \to \CO_S$. Since by assumption the map $\F^{1-n}\phi(X)\to \F^{1-n}\phi(X'')$ is an epimorphism, this $\tilde f$ lies in $(\F^{1-n}\phi(X''))^\perp$. Thus $\tilde f$ is a preimage of $f$ in $(\F^{1-n}\phi(X''))^\perp$. Hence \eqref{eq:exactseq3} is exact in the middle and we are done.
\end{proof}
In the following lemma, the category $\CT_{k'}$ is the base change of $\CT$ to $k'$ given by Construction \ref{TannakaBaseChange}.
\begin{lemma} \label{FFFBaseChange}
  \begin{itemize}
      \item [(i)] Let $\phi\colon \CT\to \FilLF(S)$ be a filtered fiber functor and $k'$ a finite field extension of $k$. Then there exists a filtered fiber functor $\phi'\colon \CT_{k'}\to \FilLF(S_{k'})$ making the following diagram commutative:
  \[
    \xymatrix{
      \CT \ar[r]^-\phi \ar[d] & \FilLF(S) \ar[d]^{\text{pullback}} \\
      \CT_{k'} \ar[r]^-{\phi'} & \FilLF(S_{k'}) \\
    }
  \]
  \item [(ii)] If $\phi'$ is fpqc-locally splittable, so is $\phi$.
  \end{itemize}
\end{lemma}
\begin{proof}
  (i) Let $(V,\alpha)\in \CT_{k'}$. Then we have two $k$-homomorphisms $k'\to \End(\phi(V)_{S_{k'}})$: One given by the usual action of $k'$ on $\phi(V)_{S_{k'}}$ and the other given by the image of $\alpha$ under $\phi_{S_{k'}}$. These give two morphisms $\phi(V)_{S_{k'}}\otimes_k k'\to \phi(V)_{S_{k'}}$ in $\FilLF(S_{k'})$ and we define $\phi'((V,\alpha))$ to be the cokernel of the difference of these two morphisms. This construction is functorial in $(V,\alpha)$, hence we obtain a functor $\phi'\colon \CT_{k'}\to \FilLF(S_{k'})$. By a direct verification, the functor $\phi'$ is a $k$-linear tensor functor and the diagram above commutes.

  To prove that $\phi'$ is exact, by Lemma \ref{PhiExact} it is enough to prove that $\forget\circ\phi'$ is exact and that for all epimorphisms $X\to X'$ in $\CT$ and all $n\in\BZ$ the induced morphism $\F^n\phi'(X)\to\F^n\phi'(X')$ is an epimorphism. A straighforward diagram chasing argument shows that the second condition is fulfilled and that $\forget\circ\phi$ is right exact. By \cite[2.10]{DeligneGroth} the right exactness of $\forget\circ\phi$ implies the exactness of $\forget\circ\phi$. Thus we are done.

  (ii) If $\gamma\colon \CT_{k'} \to \GradQCoh(S')$ is a splitting of the pullback of $\phi'$ to an fpqc cover $S'$ of $S_{k'}$, then the commutativity of the diagram in (i) shows that composing $\gamma$ with the functor $\CT\to\CT_{k'}$ gives a splitting of $\phi_{S'}$.
\end{proof}

\subsection{The case $\CT=\Rep[G]$ for $G$ smooth and connected and $S$ the spectrum of a field}
First we prove the following result, which is a special case of Theorem \ref{FFFZLocSplittable}:
\begin{theorem} \label{SpecialCase1}
  Let $\CT=\Rep[G]$ for an affine connected smooth group scheme $G$ over $k$. Let $\phi$ be a filtered fiber functor on $\CT$ over $S=\Spec(k')$ for some overfield $k'$ of $k$ satisfying $\forget\circ\phi=(\omega_0)_S$. Then $\phi$ is splittable.
\end{theorem}
We will use the following result from SGA 3:
\begin{theorem}[{\cite[Th\'eor\`eme XVII.5.1.1]{SGA3II}}] \label{ExtSplits}
  Let $U$ and $H$ be two group schemes of finite type over $k$. If $U$ is unipotent and smooth and $H$ is of multiplicative type and connected, then every exact sequence
  \begin{equation*}
    1 \to U \to E \to H \to 1
  \end{equation*}
splits, that is there exists a homomorphism $H\to E$ of group schemes which is a section of the homomorphism $E\to H$. 
\end{theorem}
For the following two lemmas, let $G$ and $\phi$ be as in Theorem \ref{SpecialCase1}
\begin{lemma} \label{UUnipotent}
  The group scheme $U(\phi)$ is unipotent and smooth.
\end{lemma}
\begin{proof}
  Since $G$ is smooth, by \cite[IV.2.1.4.1]{Saavedra} so is $U(\phi)$.

  Let $V\in \Rep[G]$ be a faithful representation of $G$. Pick a basis $(v_1,\ldots,v_n)$ of $V$ which is adapted to the filtration of $V$ given by $\phi(V)$, i.e. such that each step of the filtration is the span of $\{v_1,\ldots, v_r\}$ for some $r$. Under the resulting embedding $G\into \GL_{n,k'}$ the subgroup $U(\phi)$ is mapped into the subgroup of upper triangular matrices with entries $1$ on the diagonal. This shows that $U(\phi)$ is unipotent. 
\end{proof}
For a group scheme $H$ of finite type over $k$, we denote by $\rk H$ the dimension of a maximal torus of $H$.
\begin{lemma} \label{rkLemma}
  We have $\rk P(\phi)=\rk L(\phi)$. 
\end{lemma}
\begin{proof}
   After base change to an algebraic closure of $k$ we may assume that $k$ is algebraically closed. Let $G'\defeq \UAut^\otimes(\forget \circ\phi)$. This is a group scheme over $k$ containing $L(\phi)$. By Theorem \ref{FFLocalIso} the group schemes $G$ and $G'$ are isomorphic. Since by Lemma \ref{UUnipotent} the group scheme $U(\phi)$ is unipotent, we obtain
  \begin{equation*}
    \rk P(\phi)=\rk P(\phi)/U(\phi)\leq \rk L(\phi) \leq \rk G'=\rk G.
  \end{equation*}
Hence it suffices to prove $\rk P(\phi)=\rk G$. 

If $G$ is reductive, then $P(\phi)$ is a parabolic subgroup of $G$ by \cite[IV.2.4.3.2]{Saavedra}. Thus in this case we have the desired equality.

For an arbitrary smooth connected group scheme $G$, let $V$ be its unipotent radical. Let $H\defeq G/V$ and let $\phi'\colon \Rep[H]\to \Rep[G]\to \FilLF(S)$ be the induced filtered fiber functor. Since $G$ and $H$ are smooth, by \cite[IV.2.1.4.1]{Saavedra} so are $P(\phi)$ and $P(\phi')$. Using the exactness of $\phi'$, it follows from \cite[IV.2.4.1.1 2)]{Saavedra} that the induced homomorphism $P(\phi)\to P(\phi')$ is surjective on the level of Lie algebras. Hence it is an epimorphism of algebraic groups. Its kernel is contained in $V$ and hence unipotent. Thus we obtain $\rk P(\phi)=\rk P(\phi')$. By the previous paragraph we have $\rk P(\phi')=\rk H$. Since $\rk H=\rk G$ this proves $\rk P(\phi)=\rk G$ and we are done.
\end{proof}
\begin{proof}[Proof of Theorem \ref{SpecialCase1}]
By Lemma \ref{SplittingEquiv}, showing that $\phi$ is splittable is equivalent to showing that the cocharacter $\chi(\gr\circ\phi)\colon \Gm[k']\to L(\phi)$ can be lifted to $P(\phi)$. 

Let $T$ be a maximal torus of $L(\phi)$. By Lemmas \ref{UUnipotent} and \ref{rkLemma} we have $\rk P(\phi)/U(\phi)=\rk P(\phi)=\rk L(\phi)$. Hence $T$ is also a maximal torus of $P(\phi)/U(\phi)$. Let $E$ be the preimage of $T$ under the natural homomorphism $P(\phi)\to P(\phi)/U(\phi)$. It fits into an exact sequence
\begin{equation*}
  1 \to U(\phi) \to E \to T \to 1.
\end{equation*}
Theorem \ref{ExtSplits} and Lemma \ref{UUnipotent} imply that this exact sequence splits. Since $\chi(\gr\circ\phi)$ is a central cocharacter of $L(\phi)$, it factors through $T$. Composing $\chi(\gr\circ\phi)$ with a splitting homomorphism $T\to E$ gives the desired lift of $\chi$ to $P(\phi)$.
% Let $E\defeq P(\phi) \times_{L(\phi),\chi(\gr\circ\phi)}\Gm[k']$. It fits into an exact sequence
% \begin{equation*}
%   1 \to U(\phi) \to E \to \Gm[k'] \to 1.
% \end{equation*}
% Thus Theorem \ref{ExtSplits} and Lemma \ref{UUnipotent} imply that there exists a cocharacter $\Gm[k']\to E$ splitting this sequence. But such a cocharacter is the same as a cocharacter $\Gm[k']\to P(\phi)$ which lifts $\chi(\gr\circ\phi)$. Thus $\phi$ is splittable.
\end{proof}

\subsection{The case $\CT=\Rep[G]$ for $G$ of finite type over $k$}
Now we prove the following special case of Theorem \ref{MainTheorem}:
\begin{theorem} \label{SpecialCase2}
  Let $\CT=\Rep[G]$ for any group scheme $G$ of finite type over $k$. Any filtered fiber functor on $\CT$ is fpqc-locally splittable.
\end{theorem}
First we collect some facts about restriction and induction functors. We denote by $\REP[G]$ the category of (not necessarily finite-dimensional) representations of $G$. Let $H$ be a closed subgroup scheme of $G$. We denote by $\res_H^G\colon \REP[G]\to\REP[H]$ the restriction functor. By \cite[I.3.4]{Jantzen}, it has a right adjoint $\ind_H^G\colon \REP[H]\to\REP[G]$. For $W \in \REP[H]$, we denote the adjunction map $\res_H^G\ind_H^G W\to W$ by $\varepsilon_W$. 

By \cite[I.3.3]{Jantzen} the functor $\ind_H^G$ can described explicitly as follows: For $W\in\REP[H]$ let $W_a$ be the functor from $k$-algebras to sets which sends a $k$-algebra $A$ to $W\otimes_k A$ and acts on morphisms in the natural way. Consider $G$ as a functor from $k$-algebras to sets in the usual way and let $\Mor(G,W_a)$ be the set of natural transformations from $G$ to $W_a$. The $k$-vector space structure on $W$ gives a natural $k$-vector space structure on $\Mor(G,W_a)$. Then
\begin{align}
  \label{ExplicitInd}
  \ind_H^G(W)=&\{f\in\Mor(G,W_a)\mid f(gh)=h^{-1}f(g)\\
  &\text{ for all $k$-algebras $A$ and all } g\in G(A), h\in H(A)\} \notag
\end{align}
and the action of $G$ is by left translations. Under this description the map $\varepsilon_W$ becomes
\begin{equation}
  \label{ExplicitVarepsilon}
  \varepsilon_W\colon \res_H^G(\ind_H^G(W))\to W,\; f\mapsto f(1).
\end{equation}
\begin{theorem}[{\cite[I.3.6]{Jantzen}}]
  \label{ProjectionFormula}
  Let $V\in\REP[G]$ and $W\in\REP[H]$. Then there is a canonical isomorphism of $G$-modules
  \begin{equation*}
    (\ind_H^G W)\otimes_k V\isoto \ind_H^G(W\otimes_k\res_H^G V).
  \end{equation*}
  Under the explicit description \eqref{ExplicitInd} this isomorphism sends $f\otimes_k v$ to the natural transformation $G\to (W\otimes_kV)_a$ which sends $g\in G(A)$ to $f(g)\otimes_k g^{-1}v$ for all $k$-algebras $A$.
\end{theorem}

\begin{corollary}
  \label{IndRes}
  Let $V\in \REP[G]$. Endow $k[G/H]$ with the left regular representation of $G$. Then there is a natural isomorphism of $G$-modules
  \begin{equation*}
    \ind_H^G(\res_H^G(V))\cong k[G/H]\otimes_k V.
  \end{equation*}
 Under this isomorphism, the map $\varepsilon_{\res_H^G(V)}$ corresponds to the map $e_{G/H}\otimes_k V\colon k[G/H]\otimes_kV\to k\otimes_k V\cong V$ where $e_{G/H}\colon k[G/H]\to k$ is the homomorphism corresponding to the composition of the unit section $\Spec(k)\to G$ of $G$ and the canonical morphism $G\to G/H$.
\end{corollary}
\begin{proof}
  The isomorphism is obtained by setting $W=1$ in Theorem \ref{ProjectionFormula}. The description of $\varepsilon_{\res_H^G(W)}$ can be directly verified using \eqref{ExplicitInd}, \eqref{ExplicitVarepsilon} and Theorem \ref{ProjectionFormula}.
\end{proof}
\begin{theorem}
  \label{IndProperties}
  Assume that $G/H$ is affine.
  \begin{itemize}
  \item [(i)] The functor $\ind_H^G$ is exact.
  \item [(ii)] The map $\varepsilon_W$ is surjective for all $W\in\REP[H]$.
  \end{itemize}
\end{theorem}
\begin{proof}
  For (i), see \cite[I.5.13]{Jantzen}. Part (ii) is proven in \cite[Lemma 4.2]{CPS1} in case $G$ and $H$ are affine algebraic groups. But the proof given there works for arbitrary affine group schemes $G$ and $H$ of finite type over $k$.
\end{proof}
\begin{construction} \label{IndConstruction}
  Let $H$ be a closed subgroup scheme of $G$ such that $G/H$ is finite. Let $S$ be a reduced scheme and $\phi\colon\Rep[G]\to \FilLF(S)$ be a filtered fiber functor such that $\forget\circ\phi=(\omega_{0})_S$.

We construct a functor $\phi'\colon \Rep[H]\to\FilLF(S)$ such that $\forget\circ\phi'=(\omega_{0})_S$ as follows: Since each $V\in \REP[G]$ is the union of its finite-dimensional subrepresentations, the functor $\phi$ can be naturally extended to an exact $k$-linear tensor functor $\REP[G]\to \FilLF(S)$. For $n\in\BZ$ and $V\in \REP[G]$ we denote the $n$-th step of the filtration on $V_S$ defined by this functor by $\F^n(V_S)$. For $n\in \BZ$ and $W\in\Rep[H]$ define $\F^n(W_S)\defeq (\varepsilon_W)_S(\F^n(\ind_H^G(W)_S))$. Since by Theorem \ref{IndProperties} the map $(\varepsilon_{W})_S$ is surjective, this gives a descending filtration on $W_S$. Since this construction is functorial in $W$, this defines a functor $\phi'\colon\Rep[H]\to \FilLF(S)$.
\end{construction}
\begin{proposition} \label{IndConstructionProps}  
  Let $H$ be a closed subgroup scheme of $G$ such that $G/H$ is finite. Let $S$ be a reduced scheme and $\phi\colon\Rep[G]\to \FilLF(S)$ be a filtered fiber functor such that $\forget\circ\phi=(\omega_{0})_S$. The functor $\phi'\colon \Rep[H]\to\FilLF(S)$ given by Construction \ref{IndConstruction} is a filtered fiber functor satisfying $\phi=\phi'\circ\res_H^G$.
\end{proposition}
\begin{proof}
First we show:
  \begin{lemma} \label{blub}
    Let $e_{G/H}\colon k[G/H]\to k$ be the morphism from Corollary \ref{IndRes} and equip $k[G/H]$ with the left regular representation. Then $(e_{G/H})_S(\F^i(k[G/H]_S))=0$ for all $i> 0$.
  \end{lemma}
  \begin{proof}
    Consider $x\in\F^i(k[G/H]_S)$. Since the multiplication $k[G/H]\otimes_{k}k[G/H]\to k[G/H]$ is $G$-equivariant, the element $x^j$ lies in $\F^{ij}(k[G/H]_S)$ for all $j\geq 1$. Since $\phi$ is a tensor functor, each $\F^j(k[G/H]_S)$ is locally a direct summand of $k[G/H]_S$ for all $j\in\BZ$. Since by assumption $k[G/H]_S$ is free of finite rank as an $\CO_S$-module, this implies that for $j$ large enough $\F^j(k[G/H]_S)=0$. Thus $x$ is nilpotent. Since $\CO_S$ is reduced, this implies $(e_{G/H})_S(\F^i(k[G/H]_S))=0$.
  \end{proof}

\begin{lemma} \label{blubblub}
  $\phi=\phi'\circ\res_H^G$.
\end{lemma}
\begin{proof}
  Let $V\in \Rep[G]$. By Corollary \ref{IndRes} we can identify $\ind_H^G(\res_H^G(V))_S$ as a $G$-module with $(k[G/H]\otimes_k V)_S$ such that $(\varepsilon_{\res_H^G(V)})_S$ corresponds to $(e_{G/H}\otimes_k V)_S$. Hence 
  \begin{equation}
    \label{eq:Fn}
    \F^n(\res_H^G(V)_S)=\sum_{i+j=n}(e_{G/H})_S(\F^i(k[G/H]_S))\cdot \F^j(V_S).
  \end{equation}

By Lemma \ref{blub}, the terms with $i> 0$ in \eqref{eq:Fn} vanish. If $k$ is given the trivial $G$-action, the homomorphism $k\to k[G/H]$ is $G$-equivariant. Hence $1\in \F^0(k[G/H]_S)$, which implies that the term $(e_{G/H})_S(\F^i(k[G/H]_S))\cdot \F^j(V_S)$ in \eqref{eq:Fn} is equal to $\F^j(V_S)$ for $i\leq 0$. Thus \eqref{eq:Fn} shows $\F^n(\res_H^G(V)_S)=\F^n(V)$, which means that $\phi$ and $\phi'\circ\res_H^G$ agree on objects. Since $\forget\circ\phi=(\omega_{0})_S=\forget\circ\phi'\circ\res_H^G$ these functors also agree on morphisms. 
\end{proof}
  \begin{lemma}
    The functor $\phi'$ is a tensor functor. 
  \end{lemma}
  \begin{proof}
    Lemma \ref{blubblub} implies that $\phi'$ maps the trivial representation of $H$ to the tensor unit of $\FilLF(S)$. It remains to prove that $\phi'$ is compatible with tensor products.

    In the following we will sometimes use the canonical isomorphism $X_S\otimes_{\CO_S}X'_S\cong (X\otimes_kX')_S$ for certain $k$-vector spaces $X$ and $X'$ without mention.

 Let $W,W'\in\Rep[H]$. Then there is a commutative diagram
    \begin{equation*}
      \xymatrixcolsep{10pc}\xymatrix{
        \ind_H^G(W)_S\otimes_{\CO_S}\ind_H^G(W')_S \ar[d]_{\sim} \ar@{->>}[r]^-{(\varepsilon_W)_S\otimes_{\CO_S}(\varepsilon_{W'})_S} &W_S\otimes_{\CO_S} W'_S \ar[dd]^\sim\\ 
        \ind_H^G(W\otimes_k \res_H^G(\ind_H^G(W')))_S \ar@{->>}[d]_{\ind_H^G(W\otimes_k \varepsilon_{W'})_S}  & \\
        \ind_H^G(W\otimes_k W')_S \ar@{->>}[r]^{(\varepsilon_{W\otimes_k W'})_S}&(W\otimes_kW')_S
      }
    \end{equation*}
    where the top left vertical isomorphism is given by Theorem \ref{ProjectionFormula} for $V=\ind_H^G(W')$ and the bottom left vertical and the horizontal maps are surjective by Theorem \ref{IndProperties}. To check the commutativity of this diagram we may assume $S=\Spec(k)$ and then it can be directly verified using the explicit description of the involved vector spaces and maps given by \eqref{ExplicitInd}, \eqref{ExplicitVarepsilon} and Theorem \ref{ProjectionFormula}. 

    Note that the two maps on the left are $G$-equivariant. Using the definition of $\phi'$ and the fact that $\phi$ is a tensor functor the commutativity of the diagram implies the lemma:
    \begin{align*}
      \F^n&((W\otimes_kW')_S)\\
      &= (\varepsilon_W\otimes_k\varepsilon_{W'})_S(\F^n(\ind_H^G(W\otimes_k W')_S)) \\
      &= ((\varepsilon_W)_S\otimes_{\CO_S}(\varepsilon_{W'})_S)(\F^n(\ind_H^G(W)_S\otimes_{\CO_S}\ind_H^G(W')_S)) \\
      &= ((\varepsilon_W)_S\otimes_{\CO_S}(\varepsilon_{W'})_S)(\sum_{i+j=n}\F^i(\ind_H^G(W)_S)\otimes_{\CO_S} \F^j(\ind_H^G(W')_S))\\
      &=\sum_{i+j=n}\F^i(W_S)\otimes_{\CO_S} \F^j(W'_S).
    \end{align*}
  \end{proof} 

It remains to show:
  \begin{lemma}
    The functor $\phi'$ is a filtered fiber functor.
  \end{lemma}
  \begin{proof}
    It follows directly from the construction of $\phi'$ that it is $k$-linear. Furthermore $\forget\circ\phi'=\omega_0$ is a fiber functor. Thus by Lemma \ref{PhiExact} is suffices to show that for any surjection $W\to W'$ in $\Rep[H]$ and any $n\in\BZ$ the induced morphism $\F^n W\to\F^n W'$ is an epimorphism. This follows from the definition of $\phi'$ by a straightforward diagram chasing argument.
  \end{proof}

\end{proof}

 \begin{proof}[Proof of Theorem \ref{SpecialCase2}]
Let $\phi$ be a filtered fiber functor on $\CT$ over some scheme $S$ over $k$. Since the fiber functors $\forget\circ\phi$ and $\omega_0$ are fpqc-locally isomorphic by Theorem \ref{FFLocalIso}, after replacing $S$ by a suitable covering we may assume that $\forget\circ\phi=(\omega_0)_S$. Let $H$ be the maximal reduced subscheme of the connected component of the identity of $G$. By Lemma \ref{FFFBaseChange} we can replace $k$ by a finite field extension such that $H$ is a smooth subgroup scheme of $G$. 

Consider a point $s\in S$. First we show that the fiber $\phi_s$ of $\phi$ in $s$ is splittable. Proposition \ref{IndConstructionProps} applied to $H$ as above allows us to factor $\phi_s$ as $\phi'\circ\res_{H}^G$ for a filtered fiber functor $\phi'\colon \Rep[H] \to \FilLF(k(s))$ satisfying $\forget\circ\phi'=(\omega_0)_{k(s)}$. Since $H$ is smooth and connected, by Theorem \ref{SpecialCase1} there exists a graded fiber functor $\gamma'\colon \Rep[H]\to\GradQCoh(k(s))$ which splits $\phi'$. Then the graded fiber functor $\gamma_0\defeq \gamma'\circ\res^G_{H}$ splits $\phi_s$.

The fact that $\forget\circ\phi'=(\omega_0)_{k(s)}$ implies $\forget\circ\gamma'=(\omega_0)_{k(s)}$. Thus $\gamma'$ is determined by the cocharacter $\chi(\gamma')\colon \Gm[k(s)]\to H_{k(s)}$. Since $H$ is smooth, by \cite[Th\'eor\`eme XI.5.8]{SGA3II} there exists an \'etale morphism $S'\to S$, a point $s'\in S'$ over $s$ with trivial residue field extension $k(s')/k(s)$ and a cocharacter $\chi\colon \Gm[S']\to H_{S'}$ whose fiber in $s'$ is $\chi(\gamma')$. By construction, the fiber of $\chi$ in $s'$ splits the fiber of $\phi_{S'}$ in $s'$.

Next we show that $\phi_{\CO_{S',s'}}$ is splittable. For $V\in\Rep[G]$ and $n\in\BZ$, we denote by $\F^n(V_{\CO_{S',s'}})$ the $n$-th step of the filtration on $V_{\CO_{S',s'}}$ defined by $\phi_{\CO_{S',s'}}$ and by $\gr^n_\chi(V_{\CO_{S',s'}})$ the $n$-th step of the grading of $V_{\CO_{S',s'}}$ defined by $\chi_{\CO_{S',s'}}$. For $n\in \BZ$ let $\G_n(V_{\CO_{S',s'}})\defeq \oplus_{i\leq n}\gr^i_\chi(V_{\CO_{S',s'}})$; this defines a decreasing filtration of $V_{\CO_{S',s'}}$. For $n\in \BZ$ consider the addition morphism $\psi\colon \F^n(V_{\CO_{S',s'}})\oplus \G_n(V_{\CO_{S',s'}})\to V_{\CO_{S',s'}}$. Since $\chi$ splits $\phi$ in the fiber over $s'$, in this fiber the two filtrations defined above are opposite. Therefore $\psi$ is an isomorphism in this fiber by Lemma \ref{OppositeFiltrations}. Since both domain and codomain of $\psi$ are locally free $\CO_{S',s'}$-modules of finite rank, it follows from Nakayama's Lemma that $\psi$ is an isomorphism. Hence the two filtrations are opposite by Lemma \ref{OppositeFiltrations}. Therefore $\gr^n(V_{\CO_{S',s'}})\defeq \F^n(V_{\CO_{S',s'}})\cap \G_n(V_{\CO_{S',s'}})$ gives a grading of $V_{\CO_{S',s'}}$ which splits both filtrations. Since this grading is functorial in $V$ and compatible with tensor products, we have constructed a graded fiber functor $\gamma$ on $\Rep[G]$ over $\CO_{S',s'}$ which splits $\phi_{\CO_{S',s'}}$. 

Finally we show that $\phi_{S'}$ is splittable on some open neighbourhood of $s'$. Since $\forget\circ\gamma=(\omega_0)_{\CO_{S',s'}}$, the graded fiber functor $\gamma$ is determined by $\chi\defeq \chi(\gamma)\colon \Gm[\CO_{S',s'}]\to P(\phi_{\CO_{S',s'}})$. Since $G$ is of finite type over $k$, so is $P(\phi)$. Thus it follows from \cite[Th\'eor\`eme 8.8.2]{EGA4III} that $\chi$ can be extended to a cocharacter $\chi\colon \Gm[U]\to P(\phi)_U$ for some open subset $U$ of $S'$ containing $s'$, which we take to be connected. Then $\gamma$ can also be extended to $U$. For $V\in \Rep[G]$ and $n\in \BZ$ we denote by $\F^n(V_U)$ (resp. $\gr_\gamma^n(V_U)$) the $n$-th step of the filtration (resp. grading) of $V_U$ given by $\phi$ (resp. $\gamma$). Since $\chi$ is a cocharacter of $P(\phi)$, for any $V\in\Rep[G]$, each $\F^n(V_U)$ can be decomposed as $\F^n(V_U)=\oplus_{i\in\BZ}\F^n(V_U)\cap\gr_\gamma^i(V_U)$. Since $\gamma$ splits $\phi$ in the fiber over $s'$, this decomposition takes the form $\F^n(V_{k(s')})=\oplus_{i\geq n}\gr_\gamma^i(V_{k(s')})$ in this fiber. Since the terms $\F^n(V_U)\cap\gr_\gamma^m(V_U)$ are direct summands of $\F^n(V_U)$, they are locally free $\CO_U$-modules of finite rank. Since $U$ is connected, their ranks are constant on $U$. Hence $\F^n(V_U)=\oplus_{i\geq n}\gr_\gamma^i(V_U)$, which shows that $\gamma$ splits $\phi_U$.

The image of $U$ in $S$ is an open subset of $S$ containing $s$ over which $\phi$ is fpqc-locally splittable. By varying $s$ over $S$ we find that $\phi$ is fpqc-locally splittable.
 \end{proof}

\subsection{The case that $\CT$ has a tensor generator}
\begin{theorem} \label{SpecialCase3}
  Let $\CT$ be a Tannakian category possessing a tensor generator. Any filtered fiber functor on $\CT$ is fpqc-locally splittable.
\end{theorem}
\begin{proof}
   By \cite[6.20]{DeligneGroth} there exists a finite field extension $k'$ of $k$ over which $\CT$ has a fiber functor. Thus using Lemma \ref{FFFBaseChange} we may replace $k$ by $k'$ so that we may assume that $\CT$ is neutral. Then we are in the previous case.
\end{proof}
\subsection{The case $\CT$ arbitrary}
\begin{proof}[Proof of Theorem \ref{MainTheorem}]
 The category $\CT$ is the filtered colimit of the set $I_\CT$ of Tannakian subcategories of $\CT$ from Subsection \ref{TensorGenSection} which possess a tensor generator. 
 \begin{lemma} \label{SplFF}
   For each $\CT'\in I_\CT$, the functor $\USpl(\phi|_{\CT'})$ is representable by a scheme which is affine and faithfully flat over $S$.
 \end{lemma}
 \begin{proof}
    By Theorem \ref{SpecialCase3}, the filtered fiber functor $\phi|_{\CT'}$ is fqpc-locally splittable. Thus, since it is sufficient to prove the claim fpqc-locally on $S$, we may assume that $\phi|_{\CT'}$ is splittable. Using Lemma \ref{GFFPPrinciple} we reduce to the case that $S$ is the spectrum of an algebraically closed field. Since $\phi|_{\CT'}$ is splittable, by Lemma \ref{SplUPseudotorsor} the functor $\USpl(\phi|_{\CT'})$ is a $U(\phi|_{\CT'})$-torsor and thus representable by an affine scheme which is trivially flat over $S$.
 \end{proof}

The functor $\USpl(\phi)$ is in a natural way the filtered limit of the functors $\USpl(\phi|_{\CT'})$ over $\CT'\in I_\CT$. Since by Lemma \ref{SplFF} the functors $\USpl(\phi|_{\CT'})$ are representable by schemes which are affine and faithfully flat over $S$, by \cite[Proposition 8.2.3]{EGA4III} and \cite[Proposition 8.3.8]{EGA4III} the functor $\USpl(\phi)$ is representable by a scheme which is affine and faithfully flat over $S$. Thus there exists a splitting of $\phi$ fpqc-locally on $S$.
\end{proof}
%%%%%%%%%%%%%%%%%%%%%%%%%%%%%%%%%%%%%%%%%%%%%%%%%%%%%%

%%%%%%%%%%%%%%%%%%%%%%%%%%%%%%%%%%%%%%%%%%%%%%%%%%%%%%%%%%%%%%%%%%%%%%%%%%%%%%%%%%%%%%%%%%%

\bibliography{references}
\bibliographystyle{plain}
%%%%%%%%%%%%%%%%%%%%%%%%%%%%%%%%%%%%%%%%%%%%%%%%%%%%%%%%%%%%%%%%%%%%%%%%%%%%%%%%%%%%%%%%%%%
\end{document}